\documentclass[a4paper,11pt,fleqn]{article}
\usepackage{amsmath,amssymb,amsthm,graphicx,subfigure,float,caption,epstopdf,tabularx,color,bm,amsfonts,epic}
\usepackage[top=1in, bottom=1in, left=1.25in, right=1.25in]{geometry}
\usepackage{appendix}
\usepackage{multirow}
\usepackage{diagbox}
\allowdisplaybreaks
\newtheorem{thm}{Theorem}[section]
\newtheorem{lem}{Lemma}[section]
\newtheorem{coro}{Corollary}[section]

\newtheorem{rmk}{Remark}[section]
\newtheorem*{prf}{Proof}
\numberwithin{equation}{section}
\newtheorem{prop}[thm]{Proposition}

\usepackage{indentfirst}
\graphicspath{{Fig}}

\begin{document}
\title{A linearly implicit and local energy-preserving scheme for the sine-Gordon equation based on the invariant energy quadratization approach}
\author{Chaolong Jiang$^1$, \quad Wenjun Cai$^2$, \quad Yushun Wang$^2$\footnote{Correspondence author. Email:
wangyushun@njnu.edu.cn.}\\
{\small $^1$ School of Statistics and Mathematics, }\\
{\small Yunnan University of Finance and Economics, Kunming 650221, China}\\
{\small $^2$ Jiangsu Provincial Key Laboratory for NSLSCS,}\\
{\small School of Mathematical Sciences,  Nanjing Normal University,}\\
{\small  Nanjing 210023, China}\\
}
\date{}
\maketitle

\begin{abstract}
In this paper, we develop a novel, linearly implicit and local energy-preserving scheme for the sine-Gordon equation. The basic idea is from the invariant energy quadratization approach to construct energy stable schemes for gradient systems, which are energy dispassion. We here take the sine-Gordon equation as an example to show that the invariant energy quadratization approach is also an efficient way to construct linearly implicit and local energy-conserving schemes for energy-conserving systems. Utilizing the invariant energy quadratization approach, the sine-Gordon equation is first reformulated into an equivalent system, which inherits a modified local energy conservation law. The new system are then discretized by the conventional finite difference method and a semi-discretized system is obtained, which can conserve the semi-discretized local energy conservation law. Subsequently, the linearly implicit structure-preserving method is applied for the resulting semi-discrete system to arrive at a fully discretized scheme. We prove that the resulting scheme can exactly preserve the discrete local energy conservation law. Moveover, with the aid of the classical energy method, an unconditional and optimal error estimate for the scheme is established in discrete $H_h^1$-norm. Finally, various numerical examples are addressed to confirm our theoretical analysis and
demonstrate the advantage of the new scheme over some existing local structure-preserving schemes. \\[2ex]
\end{abstract}

\section{Introduction}
In this paper, we consider the following sine-Gordon equation in two dimensions
\begin{align}\label{2SG:eq:1.1}
u_{tt}=\Delta u-\sin(u),\ \ (x,y)\in\Omega\subset\mathbb{R}^{2},\ 0<t\leq T,
\end{align}
with 
initial conditions
\begin{align*}
u(x,y,0)=f(x,y),\ u_t(x,y,0)=g(x,y),\ (x,y)\in\bar{\Omega},
\end{align*}
where $\Delta =\partial_{xx}+\partial_{yy}$, $\Omega=(x_L,x_R)\times(y_L,y_R)\subset\mathbb{R}^{2}$,
 and $f(x,y)$ and $g(x,y)$ are wave modes or kinks
and their velocity, respectively \cite{Josephson65}.


\begin{prop}\rm
The sine-Gordon equation \eqref{2SG:eq:1.1} admits the following local energy conservation law
\begin{align}\label{2SG:1.4}
\big(\frac{1}{2}v^2+\frac{1}{2}u_x^2+\frac{1}{2}u_y^2+(1-\cos(u))\big)_t-(u_xv)_x-(u_yv)_y=0.
\end{align}
\end{prop}
\begin{proof}\rm
Eq. \eqref{2SG:eq:1.1} can be rewritten as
\begin{align}\label{2SG:eq:1:2}
  &u_t=v,\\\label{2SG:eq1:2}
  &v_t=\Delta u-\sin(u),
  \end{align}
Multiplying of \eqref{2SG:eq1:2} with $v$, we have
\begin{align}\label{LSP-SG:1-4}
v_tv-\Delta uv+\sin(u)v=0.
\end{align}
By the continuous Leibniz rule, together with \eqref{2SG:eq:1:2}, we have
\begin{align}\label{LSP-SG:1-5}
 &vv_t=\frac{1}{2}(v^2)_t,\ \sin(u)v=\sin(u)u_t=(1-\cos(u))_t,
 \end{align}
 and
 \begin{align}\label{LSP-SG:1-6}
&\Delta uv=(u_xv)_x+(u_yv)_y-u_xv_x-u_yv_y=(u_xv)_x+(u_yv)_y-\frac{1}{2}(u_x^2+u_y^2)_t.
\end{align}
From \eqref{LSP-SG:1-4}-\eqref{LSP-SG:1-6}, we can obtain the local energy conservation law \eqref{2SG:1.4}.
\end{proof}
Under suitable boundary conditions, such as periodic boundary conditions, integrating the energy conservation law \eqref{2SG:1.4} over the spatial domain leads to the following global energy conservation law
 \begin{align}\label{2SG:eq:1.3}
 \mathcal{H}(t)=\frac{1}{2}\iint_{\Omega}(v^2+u_x^2+u_y^2+2(1-\cos(u)))dxdy\equiv \mathcal{H}(0).
 \end{align}
%

Various numerical schemes were proposed and studied for solving the sine-Gordon equation, including finite difference methods \cite{Bratsos07b,CL81}, a finite element method \cite{Argyris91}, a meshless method \cite{DG10}, a split cosine method \cite{SKV10} and other
effective methods (e.g., see Refs. \cite{DS08,JPM12,LiuW17}). However, the mentioned numerical methods
cannot preserve the discrete analogue of the continuous energy conservation property of
the sine-Gordon equation.

In recent years, due to the superior properties in long time numerical computation
over traditional numerical methods, structure-preserving methods, which can preserve one or more intrinsic properties of the system exactly, become more and more powerful in scientific research. As the most important components of structure-preserving methods, symplectic and multi-symplectic methods, which can preserve symplectic and multisymplectic structures of Hamiltonian systems, respectively, have gained remarkable
success in numerical stimulations of the sine-Gordon equation (e.g., see Refs. \cite{AHS97,HJLL07,LSQ14,McLachlan94,Reich00,SW08,WWJQ03,ZTHTW10} and references therein). 
Besides the geometric structure, the sine-Gordon equation also possesses the energy conservation law.
It is well-known that the conservation of
energy is a crucial property of mechanical systems and plays an important role in the proof of stability, convergence and existence of solution for numerical methods (e.g., see Refs. \cite{LQ95,ZGL95}). Thus, how to design numerical schemes, which preserve rigorously a discrete energy of the sine-Gordon equation, attracts a lot of interest (e.g. see Refs. \cite{BCI15,BI16,CGM12,Furihata01,GPRV86,JSLW17,KFCG17,LQ95,ZV91} and references therein).
 However, most of existing energy-conserved schemes are fully implicit, which implies that one has to solve a system
  of nonlinear equations, at each time step. An explicit energy-preserving scheme is much easier for implementation, however, it is often conditionally stable (e.g., see Refs. \cite{GPRV86,ZV91}) so that it may require very small temporal step-size and suffer impractical computational
costs for long time computation. The golden middle strategy is to develop linearly implicit schemes.
At each time step, the linearly implicit scheme only requires to solve a linear system, which leads to considerably lower costs than the implicit one. In Ref. \cite{MF01}, Matsuo and Furihata proposed a new procedure for designing linearly implicit schemes for complex-valued nonlinear partial differential equations, which inherit energy conservation. Later on, Dahlby and Owren \cite{DO11} generalized the ideas and developed a general framework for deriving linearly implicit and energy-preserving schemes for partial differential equations (PDEs) with polynomial nonlinearities. Recently, Cai et al. \cite{CLW18} proposed the partitioned averaged vector field method, which provides an efficient procedure to develop linearly implicit and energy-preserving schemes for some special Hamiltonian PDEs. However, these procedures are invalid for conserving-systems with nonpolynomial nonlinearities, such as the sine-Gordon equation.

 In addition, most of existing linearly implicit schemes can only preserve the global energy which depends on the suitable boundary conditions, such as periodic boundary conditions, otherwise these schemes will be failed. To address this drawback, Wang et al. \cite{WWQ08} introduced the concept of local structure-preserving methods. The basic idea of the local structure-preserving methods is to generalize the preserved structures of the structure-preserving methods on the global time level to local area so that the structures can be preserved in any local areas or any points in time-space region. Therefore, the boundary conditions are not necessary any more. In recent years, the local structure-preserving method has been of high interest in studying Hamiltonian PDEs (e.g., see Refs. \cite{CWL13,GCW14b,LW15} and references therein). However, to the best of our knowledge, there has been no reference considering a linearly implicit scheme for energy-conserving systems, which can preserve the local energy conservation law. In this paper, following the idea of invariant energy quadratization (IEQ) approach, we develop a linearly implicit and local energy-preserving scheme for the sine-Gordon equation \eqref{2SG:eq:1.1}. Furthermore, based on the classical energy method, an unconditionally optimal error estimate for the proposed scheme in discrete $H_h^1$-norm is established.
 The IEQ approach was recently proposed by Yang and his collaborators \cite{YZW17,YZWS17,ZYGW17} to develop linearly implicit and energy stable schemes for gradient flows. To our knowledge, there has been no reference considering the IEQ approach for developing linearly implicit schemes for energy-conserving systems, which inherit the local energy conservation law, so is the error estimate for the scheme proposed by the IEQ approach. Taking the sine-Gordon equation \eqref{2SG:eq:1.1} for example, we first explore the feasibility of the IEQ approach and establish the first result on the error estimate of the scheme obtained by the IEQ approach without any restriction
on the mesh ratio.

The outline of this paper is
organized as follows. In Section \ref{2SG:Sec3}, the sine-Gordon equation \eqref{2SG:eq:1.1} is rewritten as an equivalent system based on the IEQ approach, which inherits a modified local energy conservation law. A semi-discrete system, which can exactly preserve the semi-discrete modified local energy conservation law, is then obtained by using the conventional finite difference method for the spatial discretization. Finally, we apply a linearly implicit structure-preserving method for the semi-discrete system to arrive at fully
discrete schemes, which are shown to preserve the discrete local energy conservation law exactly. The unique solvability of the solution and the convergence analysis of the proposed scheme are presented in Section \ref{2SG:Sec4}. Some numerical experiments are reported in Section \ref{2SG:Sec5}. We draw some conclusions in Section \ref{2SG:Sec6}.

\section{Construction of the linearly implicit and local energy-preserving scheme}\label{2SG:Sec3}
In this section, we propose a linearly implicit and local energy-preserving scheme for the sine-Gordon equation \eqref{2SG:eq:1.1}. Inspired by the IEQ approach, we first introduce a Lagrange multiplier (auxiliary variable) $r=\sqrt{2-\cos(u)}$, and rewrite \eqref{2SG:eq:1:2}-\eqref{2SG:eq1:2} as
\begin{align}\label{2SG:eq:2.4}
&\left\lbrace
  \begin{aligned}
  &u_t=v,\\
  &v_t=\Delta u-\frac{\sin(u)}{\sqrt{2-\cos(u)}}r,\\
  &r_t=\frac{\sin(u)}{2\sqrt{2-\cos(u)}}v.
  \end{aligned}\right.\ 
  \end{align}

\begin{thm}\rm
The system \eqref{2SG:eq:2.4} satisfies the following modified local energy conservation law
\begin{align}\label{2SG:eq:1-4}
\big(\frac{1}{2}v^2+\frac{1}{2}u_x^2+\frac{1}{2}u_y^2+r^2\big)_t-(u_xv)_x-(u_yv)_y=0.
\end{align}
\end{thm}
\begin{proof}\rm
Multiplying of the second equality of \eqref{2SG:eq:2.4} with $v$, we have
\begin{align}\label{LSP-SG:2.63}
v_tv-\Delta uv+\frac{\sin(u)}{\sqrt{2-\cos(u)}}rv=0.
\end{align}
With the third equality of \eqref{2SG:eq:2.4}, we can deduce
\begin{align}\label{LSP-SG:2.67}
\frac{\sin(u)}{\sqrt{2-\cos(u)}}rv=2rr_t=(r^2)_t.
\end{align}
From \eqref{LSP-SG:1-5}, \eqref{LSP-SG:1-6} and \eqref{LSP-SG:2.67}, we obtain
\begin{align*}
\big(\frac{1}{2}v^2+\frac{1}{2}u_x^2+\frac{1}{2}u_y^2+r^2\big)_t-(u_xv)_x-(u_yv)_y=0.
\end{align*}
This completes the proof.
\end{proof}
\begin{coro}\rm With the periodic boundary
conditions, the system \eqref{2SG:eq:2.4} possesses the modified global energy conservation law, as follows:
 \begin{align}\label{2SG:eq:2.6}
 \frac{d}{dt}E=0,\ E=\iint_{\Omega}\frac{1}{2}\Big(v^2+u_x^2+u_y^2+2r^2\Big)dxdy.
 \end{align}
\end{coro}
Let $\Omega_{h}=\{(x_{j_1},y_{j_2})|x_{j_{1}}=x_L+j_{1}h_{1},y_{j_{2}}=y_L+j_{2}h_{2};\ 0\leq j_{r}\leq N_{r}-1, r=1,2\}$ be a partition of $\Omega$
with mesh sizes $h_1=(x_R-x_L)/N_{1}$ and $h_2=(y_R-y_L)/N_{2}$, respectively, where $N_{1}$ and $N_{2}$ are two even numbers. 
A discrete mesh function ${ u}_{j_1,j_2},\ (x_{j_1},y_{j_2})\in \Omega_{h}$ is said to satisfy the periodic boundary conditions if and only if
\begin{align}\label{SG-PBS}
&\left\lbrace
  \begin{aligned}
&x-\text{periodic}:\ u_{0,j_2}=u_{N_1,j_2},\ u_{-1,j_2}=u_{N_1-1,j_2},\ j_2=0,1,2\cdots,N_2-1,\\
&y-\text{periodic}:\ u_{j_1,0}=u_{j_1,N_2},\ u_{j_1,-1}=u_{j_1,N_2-1},\ j_1=0,1,2\cdots,N_1-1.
\end{aligned}\right.\
\end{align}
For a positive integer $M$, let $\Omega_{\tau}=\{t_{n}|t_{n}=n\tau; 0\leq n\leq M\}$
be a uniform partition of $[0,T]$ with time step $\tau=T/M$. Let $\Omega_{h\tau}=\Omega_{h}\times\Omega_{\tau}$ and denote
$u_{j_1,j_2}^n$ be a mesh function defined on $\Omega_{h\tau}$.
 For any grid function ${u}_{j_1,j_2}^n, (x_{j_1},x_{j_2},t_n)\in\Omega_{h\tau}$, we denote
\begin{align*}
&\delta_{t} {u}_{j_1,j_2}^{n}=\frac{u_{j_1,j_2}^{n+1}-u_{j_1,j_2}^{n}}{\tau},\ A_t{u}_{j_1,j_2}^{n}= \frac{{u}_{j_1,j_2}^{n+1}+{u}_{j_1,j_2}^{n}}{2},\
\hat{u}_{j_1,j_2}^{n+\frac{1}{2}}=\frac{3{u}_{j_1,j_2}^{n}-{u}_{j_1,j_2}^{n-1}}{2},\\
&\delta_{x} u_{j_1,j_2}^n=\frac{u_{j_{1}+1,j_{2}}^n-u_{j_{1},j_{2}}^n}{h_{1}},\
\delta_{y} u_{j_1,j_2}^n=\frac{u_{j_{1},j_{2}+1}^n-u_{j_{1},j_{2}}^n}{h_{2}}.
\end{align*}
Then, for any grid functions ${u}_{j_1,j_2}^n,{v}_{j_1,j_2}^n, (x_{j_1},x_{j_2},t_n)\in\Omega_{h\tau}$, we have \cite{WWQ08}:
\begin{enumerate}
\item Commutative law
\begin{align*}
\delta_t\delta_wu_{j_1,j_2}^n=\delta_w\delta_tu_{j_1,j_2}^n,\ A_t\delta_wu_{j_1,j_2}^n=\delta_wA_tu_{j_1,j_2}^n,\ \delta_x\delta_yu_{j_1,j_2}^n=\delta_y\delta_xu_{j_1,j_2}^n,
\end{align*}
for $w=x,y$.
\item Discrete Leibniz rule
\begin{enumerate}
\item $\delta_t(u_{j_1,j_2}^n v_{j_1,j_2}^n)=A_tu_{j_1,j_2}^{n}\delta_tv_{j_1,j_2}^n+\delta_tu_{j_1,j_2}^n A_tv_{j_1,j_2}^{n}$;
\item $\delta_t(u_{j_1,j_2}^{n} v_{j_1,j_2}^{n})=\delta_tu_{j_1,j_2}^{n}v_{j_1,j_2}^n+u_{j_1,j_2}^{n+1}\delta_tv_{j_1,j_2}^n$;
\item $\delta_t(u_{j_1,j_2}^{n}v_{j_1,j_2}^{n})=\delta_tu_{j_1,j_2}^{n} v_{j_1,j_2}^{n+1}+u_{j_1,j_2}^{n}\delta_tv_{j_1,j_2}^{n}$.
\end{enumerate}
\end{enumerate}
Notice that the discrete Leibniz rules are also valid for the operators $\delta_x$ and $\delta_y$.

For simplicity, let  $u_{j_1,j_2}^n=u(x_{j_1},y_{j_2},t_n)$, and $U_{j_1,j_2}^n$ be the numerical approximation of $ u(x_{j_1},y_{j_2},t_n)$ for $j_r=0,1,2,\cdots,N_r-1,r=1,2$ and $n=0,1,2,\cdots,M$.
%

\subsection{Local energy-preseving spatial semi-discretization}
Many energy-preserving schemes have been designed and investigated for solving the sine-Gordon equation in the literature, but,
 little attention was paid to the local energy-preserving properties brought by spatial discretizations.
In this section, the conventional finite difference method is applied for the spatial discretization and we prove that the resulting semi-discrete system can exactly preserve the semi-discrete local energy conservation law.

Applying the finite difference method to the system \eqref{2SG:eq:2.4} in space, we obtain the following semi-discrete system
\begin{align}\label{2SG:eq:2.7}
&\left\lbrace
  \begin{aligned}
  &\frac{d}{dt}U_{j_1,j_2}=V_{j_1,j_2},\\
  &\frac{d}{dt}V_{j_1,j_2}=\delta_x^2U_{j_1-1,j_2}+\delta_y^2U_{j_1,j_2-1}-\frac{\sin(U_{j_1,j_2})}{\sqrt{2-\cos(U_{j_1,j_2})}}R_{j_1,j_2},\\
  &\frac{d}{dt}R_{j_1,j_2}=\frac{\sin(U_{j_1,j_2})}{2\sqrt{2-\cos(U_{j_1,j_2})}}V_{j_1,j_2},\ 0\le j_r\le N_r-1,\ r=1,2,
  \end{aligned}\right.
  \end{align}
where $U_{j_1,j_2}, V_{j_1,j_2}$ and $R_{j_1,j_2}$ are the numerical approximations of $u(x_{j_1},y_{j_2},t),v(x_{j_1},y_{j_2},t)$ and $r(x_{j_1},y_{j_2},t)$ for  $j_r=0,1,2,\cdots,N_r-1,r=1,2$, respectively.
%
%
%
%
%
%
%

\begin{thm}\rm\label{2SG-lem2.1}  The semi-discrete system \eqref{2SG:eq:2.7} possesses the semi-discrete modified local energy conservation law
\begin{align}\label{LEP-C}
\frac{d}{dt}\big(\frac{1}{2}V_{j_1,j_2}^2&+\frac{1}{2}(\delta_xU_{j_1,j_2})^2+\frac{1}{2}(\delta_yU_{j_1,j_2})^2+R_{j_1,j_2}^2\big)\nonumber\\
&-\delta_x(\delta_xU_{j_1-1,j_2}V_{j_1,j_2})-\delta_y(\delta_yU_{j_1,j_2-1}V_{j_1,j_2})=0,
\end{align}
for $j_r=0,1,2,\cdots,N_r-1,\ r=1,2.$
\end{thm}
%
\begin{prf} \rm
Multiplying the second equality of \eqref{2SG:eq:2.7} with $V_{j_1,j_2}$, we have
 \begin{align}\label{LSP-SG:2.73}
\frac{d}{dt}V_{j_1,j_2}V_{j_1,j_2}&-\delta_x^2U_{j_1-1,j_2}V_{j_1,j_2}\nonumber\\
&-\delta_y^2U_{j_1,j_2-1}V_{j_1,j_2}+\frac{\sin(U_{j_1,j_2})}{\sqrt{2-\cos(U_{j_1,j_2})}}V_{j_1,j_2}R_{j_1,j_2}=0.
 \end{align}
It is clear to see
 \begin{align}\label{LSP-SG:2.74}
 \frac{d}{dt}V_{j_1,j_2} V_{j_1,j_2}
 =\frac{1}{2}\frac{d}{dt}(V_{j_1,j_2}^2).
 \end{align}
 With the third equality of \eqref{2SG:eq:2.7}, we have
 \begin{align}\label{LSP-SG-73}
 \frac{\sin(U_{j_1,j_2})}{\sqrt{2-\cos(U_{j_1,j_2})}}R_{j_1,j_2}V_{j_1,j_2}=2\frac{d}{dt}R_{j_1,j_2}R_{j_1,j_2}=\frac{d}{dt}R_{j_1,j_2}^2.
 \end{align}
 According to the discrete Leibniz rule, we can deduce
 \begin{align}\label{LSP-SG:2.75}
\delta_x^2U_{j_1-1,j_2}V_{j_1,j_2}=\delta_x(\delta_xU_{j_1-1,j_2}V_{j_1,j_2})-\frac{1}{2}\frac{d}{dt}(\delta_xU_{j_1,j_2})^2,
 \end{align}
 and
  \begin{align}\label{LSP-SG:2.76}
 \delta_y^2U_{j_1,j_2-1}V_{j_1,j_2}=\delta_y(\delta_yU_{j_1,j_2-1}V_{j_1,j_2})-\frac{1}{2}\frac{d}{dt}(\delta_yU_{j_1,j_2})^2.
 \end{align}
 From \eqref{LSP-SG:2.73}-\eqref{LSP-SG:2.76}, we have
 \begin{align*}
\frac{d}{dt}\big(\frac{1}{2}V_{j_1,j_2}^2+\frac{1}{2}(\delta_xU_{j_1,j_2})^2&+\frac{1}{2}(\delta_yU_{j_1,j_2})^2+R_{j_1,j_2}^2\big)\nonumber\\
&-\delta_x(\delta_xU_{j_1-1,j_2}V_{j_1,j_2})-\delta_y(\delta_yU_{j_1,j_2-1}V_{j_1,j_2})=0.
\end{align*}
 This completes the proof.
\end{prf}
\begin{coro}\label{coro2.1}\rm For the periodic boundary conditions \eqref{SG-PBS}, the system \eqref{2SG:eq:2.7} admits
the semi-discrete modified global energy conservation law
\begin{align*}
\frac{d}{dt}\mathcal{E}=0,\ \mathcal{E}=h_1h_2\sum_{j_1=0}^{N_1-1}\sum_{j_2=0}^{N_2-1}\Big(\frac{1}{2}V_{j_1,j_2}^2+\frac{1}{2}(\delta_xU_{j_1,j_2})^2
+\frac{1}{2}(\delta_yU_{j_1,j_2})^2+R_{j_1,j_2}^2\Big).
\end{align*}
%

\end{coro}
\subsection{Linearly implicit and local energy-preserving scheme}
We discretize the semi-discrete system \eqref{2SG:eq:2.7} using the linearly implicit structure-preserving method in time to arrive at a full-discrete scheme as follows:
\begin{align}\label{2SG:eq:3.1}
&\delta_tU_{j_1,j_2}^n=A_tV_{j_1,j_2}^{n},\\\label{2SG:eq:3.2}
& \delta_tV_{j_1,j_2}^n=\delta_x^2 A_tU_{j_1-1,j_2}^{n}+\delta_y^2 A_tU_{j_1,j_2-1}^{n}-\frac{\sin(\hat{U}_{j_1,j_2}^{n+\frac{1}{2}})}{\sqrt{2
-\cos(\hat{U}_{j_1,j_2}^{n+\frac{1}{2}})}}A_tR_{j_1,j_2}^{n},\\\label{2SG:eq:3.3}
&\delta_tR_{j_1,j_2}^n=\frac{\sin(\hat{U}_{j_1,j_2}^{n+\frac{1}{2}})}{2\sqrt{2
-\cos(\hat{U}_{j_1,j_2}^{n+\frac{1}{2}})}}A_tV_{j_1,j_2}^{n},\\
&0\le j_r\le N_r-1,\ r=1,2,\ 0\le n\le M-1,\nonumber
 \end{align}
which comprises our linearly implicit and local energy-preserving (LI-LEP) scheme for the sine-Gordon equation \eqref{2SG:eq:1.1}.
\begin{rmk} \rm It should remark that, since the LI-LEP scheme \eqref{2SG:eq:3.1}-\eqref{2SG:eq:3.3} is a three-level, we obtain ${U}_{j_1,j_2}^1,\ {V}_{j_1,j_2}^1$ and ${R}_{j_1,j_2}^1$ by using ${U}_{j_1,j_2}^{n}$ instead of $\hat{U}_{j_1,j_2}^{n+\frac{1}{2}}$ in \eqref{2SG:eq:3.2} and \eqref{2SG:eq:3.3} for the first step.
 \end{rmk}
\begin{rmk}\rm Here, we only give the LI-LEP scheme for the two dimensional sine-Gordon equation. In fact, the LI-LEP scheme can be easily extended to the sine-Gordon equation in one dimension and three dimensions, respectively.
\end{rmk}
\begin{rmk}\rm It should note that Gong et al \cite{GWW18} constructed two linearly implicit
and energy-preserving schemes for the KdV equation by utilizing the IEQ approach, however, these two schemes cannot conserve the discrete local energy conservation law of the original equation.
\end{rmk}
\begin{thm}\rm\label{2SG-lem3.2}  The LI-LEP scheme \eqref{2SG:eq:3.1}-\eqref{2SG:eq:3.3} possesses the discrete modified local energy conservation law
\begin{align}\label{LEP-F}
&\delta_t\big(\frac{1}{2}(V_{j_1,j_2}^n)^2+\frac{1}{2}(\delta_xU_{j_1,j_2}^n)^2+\frac{1}{2}(\delta_yU_{j_1,j_2}^n)^2
+(R_{j_1,j_2}^n)^2\big)\nonumber\\
&~~~~~~~~~~~~~~~~~~~~-\delta_x(A_t\delta_xU_{j_1-1,j_2}^nA_tV_{j_1,j_2}^n)
-\delta_y(A_t\delta_yU_{j_1,j_2-1}^nA_tV_{j_1,j_2}^n)=0,
\end{align}
for $j_r=0,1,2\cdots,N_r-1,\ r=1,2,$ and $n=0,1,2,\cdots,M-1$.

\end{thm}
\begin{prf} \rm
Multiplying \eqref{2SG:eq:3.2} with $A_tV_{j_1,j_2}^n$, we have
 \begin{align}\label{LSP-SG-2.73}
  \delta_tV_{j_1,j_2}^nA_tV_{j_1,j_2}^n&-\delta_x^2 A_tU_{j_1-1,j_2}^{n}A_tV_{j_1,j_2}^n-\delta_y^2 A_tU_{j_1,j_2-1}^{n}A_tV_{j_1,j_2}^n\nonumber\\
  &~~~~~~~~~~~~~~~~~~~+\frac{\sin(\hat{U}_{j_1,j_2}^{n+\frac{1}{2}})}{\sqrt{2
-\cos(\hat{U}_{j_1,j_2}^{n+\frac{1}{2}})}}A_tV_{j_1,j_2}^nA_tR_{j_1,j_2}^{n}=0.
 \end{align}
 According to the commutative law and discrete Leibniz rule, we can deduce
 \begin{align}\label{LSP-SG-2.74}
 &\delta_tV_{j_1,j_2}^n\cdot A_t V_{j_1,j_2}^n=\frac{1}{2}\delta_t(V_{j_1,j_2}^n)^2,\\\label{LSP-SG-2.75}
 &\delta_x^2A_tU_{j_1-1,j_2}^n\cdot A_tV_{j_1,j_2}^n=\delta_x(A_t\delta_xU_{j_1-1,j_2}^n\cdot A_tV_{j_1,j_2}^n)-\frac{1}{2}\delta_t(\delta_xU_{j_1,j_2}^n)^2,\\\label{LSP-SG-2.76}
&\delta_y^2A_tU_{j_1,j_2-1}^n\cdot A_tV_{j_1,j_2}^n=\delta_y(A_t\delta_yU_{j_1,j_2-1}^n\cdot A_tV_{j_1,j_2}^n)-\frac{1}{2}\delta_t(\delta_yU_{j_1,j_2}^n)^2.
 \end{align}
 With \eqref{2SG:eq:3.3} and the discrete Leibniz rule, we obtain
  \begin{align}\label{LSP-SG2.77}
 \frac{\sin(A_t{U}_{j_1,j_2}^{n})}{\sqrt{2-\cos(A_t{U}_{j_1,j_2}^{n})}}A_tR_{j_1,j_2}^n A_tV_{j_1,j_2}^n=2\delta_tR_{j_1,j_2}^nA_tR_{j_1,j_2}^{n}=\delta_t (R_{j_1,j_2}^n)^2.
 \end{align}
 From \eqref{LSP-SG-2.73}-\eqref{LSP-SG2.77}, we have
\begin{align*}
\delta_t\big(\frac{1}{2}(V_{j_1,j_2}^n)^2&+\frac{1}{2}(\delta_xU_{j_1,j_2}^n)^2+\frac{1}{2}(\delta_yU_{j_1,j_2}^n)^2
+(R_{j_1,j_2}^n)^2\big)\nonumber\\
&~~~~~-\delta_x(A_t\delta_xU_{j_1-1,j_2}^nA_tV_{j_1,j_2}^n)-\delta_y(A_t\delta_yU_{j_1,j_2-1}^nA_tV_{j_1,j_2}^n)=0.
\end{align*}
This completes the proof.\qed
\end{prf}
\begin{coro}\rm \label{2SG:thm3.1} Under the periodic boundary conditions \eqref{SG-PBS}, the LI-LEP scheme \eqref{2SG:eq:3.1}-\eqref{2SG:eq:3.3} possesses the discrete modified global energy conservation law
\begin{align*}
\mathcal{E}^n=\cdots=\mathcal{E}^0,
\end{align*}
where
\begin{align*}
 \mathcal{E}^n=h_1h_2\sum_{j_1=0}^{N_1-1}\sum_{j_2=0}^{N_2-1}\Big(\frac{1}{2}(V_{j_1,j_2}^n)^2+\frac{1}{2}(\delta_xU_{j_1,j_2}^n)^2
+\frac{1}{2}(\delta_yU_{j_1,j_2}^n)^2+(R_{j_1,j_2}^n)^2\Big).
\end{align*}
\begin{rmk}\rm  We should note that the modified local energy conservation law \eqref{2SG:eq:1-4} is equivalent to the local energy conservation law \eqref{2SG:1.4} in continuous sense, but not for the discrete sense. This indicates that the proposed scheme \eqref{2SG:eq:3.1}-\eqref{2SG:eq:3.3} cannot preserve the following discrete local energy conservation law
\begin{align*}
&\delta_t\big(\frac{1}{2}(V_{j_1,j_2}^n)^2+\frac{1}{2}(\delta_xU_{j_1,j_2}^n)^2+\frac{1}{2}(\delta_yU_{j_1,j_2}^n)^2
+(1-\cos(U_{j_1,j_2}^n))\big)\nonumber\\
&~~~~~~~~~~~~~~~~~~~~~-\delta_x(A_t\delta_xU_{j_1-1,j_2}^nA_tV_{j_1,j_2}^n)
-\delta_y(A_t\delta_yU_{j_1,j_2-1}^nA_tV_{j_1,j_2}^n)=0,
\end{align*}
for $j_r=0,1,2,\cdots,N_r-1,\ r=1,2,$ and $n=0,1,2,\cdots,M-1$.

\end{rmk}

\end{coro}

\section{Numerical analysis}\label{2SG:Sec4}
In this section, we will discuss the unique solvability and the convergence of the LI-LEP scheme \eqref{2SG:eq:3.1}-\eqref{2SG:eq:3.3}. Let
\begin{small}
\begin{align*}
\mathbb{ V}_{h}:&=\big\{{ U}|{ U}=(U_{0,0},U_{1,0},\cdots,U_{N_{1}-1,0},U_{0,1},U_{1,1},\cdots, U_{N_{1}-1,1},\cdots,U_{0,N_{2}-1},U_{1,N_{2}-1},\\
 &\cdots,U_{N_1-1,N_{2}-1})^{T}\big\}
\end{align*}
\end{small}
be the space of mesh functions defined on $\Omega_h$ and satisfy the periodic boundary conditions \eqref{SG-PBS}. For any grid functions, ${ U},{ V}\in\mathbb{ V}_{h}$, we define discrete inner product, as follows:
\begin{align*}
&\langle {U},{ V}\rangle_{h}=h_1h_2\sum_{j_1=0}^{N_1-1}\sum_{j_2=0}^{N_2-1}U_{j_1,j_2}V_{j_1,j_2},\ \langle \delta_w{U},\delta_w{ V}\rangle_{h}=h_1h_2\sum_{j_1=0}^{N_1-1}\sum_{j_2=0}^{N_2-1}\delta_wU_{j_1,j_2}\delta_wV_{j_1,j_2},
\end{align*}
for $w=x,y$.
The discrete $L^{2}$-norm of ${ U}^{n}\in\mathbb{V}_h$ and
its difference quotients are defined, respectively, as
\begin{align*}
&||{ U}||_{h}=\sqrt{\langle{ U},{U}\rangle_{h}},\ ||{\delta_{x} U}||_h=\sqrt{\langle{\delta_{x}U},\delta_{x}{ U}\rangle_{h}}, ||{\delta_{y} U}||_h=\sqrt{\langle{\delta_{y} U},\delta_{y}{ U}\rangle_{h}}.\\
\end{align*}
We also define discrete $H_h^1$ and $L^{\infty}$-norms, respectively, as
\begin{align*}
& ||{ U}||_{H_h^1}=\sqrt{||{ U}||_h^2+||{\delta_{x} U}||_h^2+||{\delta_{y} U}||_h^2},\ ||{ U}||_{h,\infty}=\max\limits_{0\le j_r\le N_r-1,r=1,2}|U_{j_1,j_2}|.
\end{align*}
In addition,  for simplicity, we denote $`\cdot$' as the componentwise product of the vectors ${ U},{V}\in\mathbb{ V}_{h}$, that is,
\begin{align*}
{ U}\cdot { V}=&\big(U_{0,0}V_{0,0},\cdots,U_{N_1-1,0}V_{N_1-1,0},\cdots,U_{0,N_2-1}V_{0,N_2-1},\cdots,U_{N_1-1,N_2-1}V_{N_1-1,N_2-1}\big)^{T}.
\end{align*}
\subsection{Solvability}
In this section, we will prove that the proposed scheme is uniquely solvable.
\begin{thm}\rm \label{2SG:thm3.2} The LI-LEP scheme \eqref{2SG:eq:3.1}-\eqref{2SG:eq:3.3} is uniquely solvable.

\end{thm}
\begin{prf}\rm Let
$$\mathcal{D}=\text{diag}\Big({b}(\hat{U}_{0,0}^{n+\frac{1}{2}}),{b}(\hat{U}_{1,0}^{n+\frac{1}{2}}),\cdots,{b}(\hat{U}_{N_1-1,0}^{n+\frac{1}{2}})
\cdots,{b}(\hat{U}_{0,N_2-1}^{n+\frac{1}{2}}),{b}(\hat{U}_{1,N_2-1}^{n+\frac{1}{2}}),\cdots,{b}(\hat{U}_{N_1-1,N_2-1}^{n+\frac{1}{2}})\Big)^T,$$
where $b(x)=\frac{\sin(x)}
{\sqrt{2-\cos(x)}}$.
Note that
\begin{align}\label{new:3.1}
&V_{j_1,j_2}^{n+1}=\frac{2U_{j_1,j_2}^{n+1}-2U_{j_1,j_2}^{n}}{\tau}-V_{j_1,j_2}^{n},\\\label{new:3.2}
& R_{j_1,j_2}^{n+1}=R_{j_1,j_2}^n
+\frac{\sin(\hat{U}_{j_1,j_2}^{n+\frac{1}{2}})}{2\sqrt{2-\cos(\hat{U}_{j_1,j_2}^{n+\frac{1}{2}})}}\Big(U_{j_1,j_2}^{n+1}-U_{j_1,j_2}^{n}\Big),
\end{align}
for $0\le j_r\le N_r-1,\ r=1,2,$ and $0\le n\le M-1$. Eq. \eqref{2SG:eq:3.2} can be rewritten as
\begin{align}\label{LIM-EP:3.11}
\mathbb{A}{U}^{n+1}={b}^n,\ \mathbb{A}={I}-\frac{\tau^2}{4}{ \Delta}_h+\frac{\tau^2}{8}\mathcal{D}^2,\ { U}^{n+1},\ { b}^n\in\mathbb{V}_h,\ 0\le n\le M-1,
\end{align}
where $I$ is the $N_1\times N_2$ identity matrix and the matrix ${ \Delta}_h$ represents the operator $\delta_x^2+\delta_y^2$. With noting the symmetric positive definite property of $\mathbb{A}$, we finish the proof.
\qed
\end{prf}
\begin{rmk}\rm It should be noted that, in the practice computation, we solve the linear equations \eqref{LIM-EP:3.11} to get the numerical solution ${ U}^{n+1}$. Then, from \eqref{new:3.1} and \eqref{new:3.2}, the numerical solutions $V^{n+1}$ and $R^{n+1}$ are obtained, respectively. This indicates that the IEQ approach introduced an intermediate variable, but the intermediate variable can be eliminated in our computations. Thus, our new scheme can be implemented efficiently.
\end{rmk}

\subsection{Convergence analysis}
In this section, we will establish an optimal priori estimate for the LI-LEP scheme \eqref{2SG:eq:3.1}-\eqref{2SG:eq:3.3} in discrete $H_h^1$-norm.
Let $C_{p}^{r,r}(\Omega)=\{u\in C^{r,r}(\Omega)|u(x,y,t)=u(x+l_1,y,t),u(x,y,t)=u(x,y+l_2,t)\}$ and $h=\max\{h_1,h_2\}$, and denote $C$ as a positive
constant which is independent of $h$ and $\tau$, and may be different in different case.

\begin{lem}\rm \label{2SG:lem4.1}
 (Gronwall inequality \cite{zhou90}).
Suppose that the discrete function $\big\{\omega^{n}|n=0,1,2,\cdots,K;K\tau=T\big\}$ is nonnegative and satisfies the inequality
\begin{align*}
\omega^{n}\leq A+\tau\sum_{l=1}^{n}B_{l}\omega^{l},
\end{align*}
where $A$ and $B_k,\ (k=1,2,\cdots,K)$ are nonnegative constants. Then
\begin{align*}
\max\limits_{0 \leq n\leq
K}|\omega^{n}|\leq A e^{2\sum_{k=1}^{K}B_{k}\tau},
\end{align*}
where $\tau$ is sufficiently small, such that $\tau\big(\max\limits_{k=0,1,\cdots,
K}B_{k}\big)\leq\frac{1}{2}$.
\end{lem}
\begin{lem}\rm \label{2SG:lem4.2} For the function $\ b(x)=\frac{\sin x}{\sqrt{2-\cos x }},\ \forall x\in\mathbb{R}$, we have
\begin{align*}
|b(x)|\le 1,\ |b^{'}(x)|=\Big|\frac{\cos x }{\sqrt{2-\cos x }}-\frac{\sin^2 x }{2(2-\cos x )^{\frac{3}{2}}}\Big|\le \frac{3}{2},\
\end{align*}
and
\begin{align*}
|b^{''}(x)|=\Big|-b(x)-\frac{3\sin 2x }{4(2-\cos x )^{\frac{3}{2}}}+\frac{3\sin^3 x}{4(2-\cos x )^{\frac{5}{2}}}\Big|\le \frac{5}{2}.
\end{align*}

\end{lem}

Define the local truncation errors of the LI-LEP scheme \eqref{2SG:eq:3.1}-\eqref{2SG:eq:3.3} as
\begin{align}\label{2SG:eq:4.3}
&\delta_t^{+}u_{j_1,j_2}^n=A_tv_{j_1,j_2}^{n}+(\xi_1)_{j_1,j_2}^n,\\\label{2SG:eq:4.4}
& \delta_t^{+}v_{j_1,j_2}^n=\delta_x^2A_tu_{j_1-1,j_2}^{n}+\delta_y^2A_tu_{j_1,j_2-1}^{n}
-\frac{\sin(\hat{u}_{j_1,j_2}^{n+\frac{1}{2}})}{\sqrt{2-\cos(\hat{u}_{j_1,j_2}^{n+\frac{1}{2}})}}A_tr_{j_1,j_2}^{n}+(\xi_2)_{j_1,j_2}^n,\\\label{2SG:eq:4.5}
&\delta_t^{+}r_{j_1,j_2}^n=
\frac{\sin(\hat{u}_{j_1,j_2}^{n+\frac{1}{2}})}{2\sqrt{2-\cos(\hat{u}_{j_1,j_2}^{n+\frac{1}{2}})}}A_tv_{j_1,j_2}^{n}+(\xi_3)_{j_1,j_2}^n,
 \end{align}
for $0\le j_r\le N_r-1,\ r=1,2,$ and $1\le n\le M-1$. Assuming $u(x,y,t)\in C^{4}\Big(0,T; C_p^{4,4}(\Omega)\Big)$, then using the Taylor expansion, we have
\begin{align}\label{2SG-lte:4.7}
&|({\xi}_1)_{j_1,j_2}^n|\le C(h^2+\tau^2),\ |\delta_x({ \xi}_1)_{j_1,j_2}^n|+|\delta_y({ \xi}_1)_{j_1,j_2}^n|\le C(h^2+\tau^2),\\\label{2SG-lte:4.8}
&|({\xi}_2)_{j_1,j_2}^n|\le C(h^2+\tau^2),\ |({\xi}_3)_{j_1,j_2}^n|\le C(h^2+\tau^2),
\end{align}
for $0\le j_r\le N_r-1,\ r=1,2,$ and $1\le n\le M$.
Defining ``error functions" as
\begin{align*}
(e_1)_{j_1,j_2}^n=u_{j_1,j_2}^n-U_{j_1,j_2}^n,\ (e_2)_{j_1,j_2}^n=v_{j_1,j_2}^n-V_{j_1,j_2}^n,\ (e_3)_{j_1,j_2}^n=r_{j_1,j_2}^n-R_{j_1,j_2}^n,
\end{align*}
for $0\le j_r\le N_r-1,\ r=1,2,$ and $0\le n\le M$.
Then subtracting \eqref{2SG:eq:3.1}-\eqref{2SG:eq:3.3} from \eqref{2SG:eq:4.3}-\eqref{2SG:eq:4.5}, respectively, we obtain the error equations
\begin{align}\label{2SG:eq:4.9}
&\delta_t^{+}(e_1)_{j_1,j_2}^n=A_t(e_2)_{j_1,j_2}^{n}+(\xi_1)_{j_1,j_2}^n,\\\label{2SG:eq:4.10}
& \delta_t^{+}(e_2)_{j_1,j_2}^n=\delta_x^2 A_t(e_1)_{j_1-1,j_2}^{n}+\delta_y^2 A_t(e_1)_{j_1,j_2-1}^{n}\nonumber\\
&~~~~~~~~~~~~~~~~~~~~~~~~~~~~~~~-b(\hat{u}_{j_1,j_2}^{n+\frac{1}{2}})A_tr_{j_1,j_2}^{n}
+b(\hat{U}_{j_1,j_2}^{n+\frac{1}{2}})A_tR_{j_1,j_2}^{n}+(\xi_2)_{j_1,j_2}^n,\\\label{2SG:eq:4.11}
&\delta_t^{+}(e_3)_{j_1,j_2}^n=
\frac{1}{2}b(\hat{u}_{j_1,j_2}^{n+\frac{1}{2}})A_tv_{j_1,j_2}^{n}
-\frac{1}{2}b(\hat{U}_{j_1,j_2}^{n+\frac{1}{2}})A_tV_{j_1,j_2}^{n}+(\xi_3)_{j_1,j_2}^n,\\
& (e_1)_{0,j_2}^n=(e_1)_{N_1,j_2}^n,\ (e_1)_{-1,j_2}^n=(e_1)_{N_1-1,j_2}^n,\nonumber\\
 &(e_1)_{j_1,0}^n=(e_1)_{j_1,N_2}^n,\ (e_1)_{j_1,-1}^n=(e_1)_{j_1,N_2-1}^n\nonumber\\
& (e_1)_{j_1,j_2}^0=0,\ (e_2)_{j_1,j_2}^0=0,\ (e_3)_{j_1,j_2}^0=0,\nonumber
\end{align}
for $0\le j_r\le N_r-1,\ r=1,2$ and $0\le n\le M-1$.

\begin{thm}\rm \label{2SG:thm4.2} We assume $u(x,y,t)\in C^{4}\Big(0,T; C_p^{4,4}(\Omega)\Big)$. Then, we have the
following error estimate for the proposed scheme
\begin{align*}
&||{u}^{n}-{U}^{n}||_{H_h^1}+||{v}^{n}-{ V}^{n}||_{h}+||{ r}^{n}-{ R}^{n}||_{h}\leq C(h^2+\tau^{2}),\ u^n,U^n,v^n,V^n,r^n,R^n\in\mathbb{V}_h,
\end{align*}
for $1\le n\le M$.
\end{thm}
\begin{prf}\rm 
Taking the discrete inner product of \eqref{2SG:eq:4.9}-\eqref{2SG:eq:4.11} with $A_t{e}_1^{n}$, $A_t{e}_2^{n}$ and
$A_t{ e}_3^{n}$, respectively, we obtain
\begin{align}\label{2SG:eq:4.25}
&\frac{1}{2}\delta_t^+||{ e}_1^n||_h^2=\eta_1^n,\\\label{2SG:eq:4.26}
&\frac{1}{2}\delta_t^+||{ e}_2^n||_h^2+\frac{1}{2}\delta_t^+||\delta_x{ e}_1^n||_h^2+\delta_t^+||\delta_y{ e}_1^n||_h^2=\eta_2^n,\\\label{2SG:eq:4.27}
&\frac{1}{2}\delta_t^{+}||{ e}_3^n||_h^2=\eta_3^n,
\end{align}
where
\begin{align*}
\eta_1^n:&=\langle A_t{ e}_2^{n},A_t{ e}_1^{n}\rangle_h+\langle{ \xi}_1^n,A_t{ e}_1^{n}\rangle_h,\\
\eta_2^n:&=-\langle \big(b(\hat{ u}^{n+\frac{1}{2}})-b(\hat{ U}^{n+\frac{1}{2}})\big)\cdot A_t{r}^{n}+b(\hat{ U}^{n+\frac{1}{2}})\cdot A_t{ e}_3^{n},A_t{ e}_2^{n}\rangle_h+\langle{ \xi}_2^n,A_t{ e}_2^{n}\rangle_h\nonumber\\
&-\langle\delta_x^2 A_t{ e}_1^{n}+\delta_y^2 A_t{ e}_1^{n},{\xi}_1^n\rangle_h, \\
\eta_3^n:&=\frac{1}{2}\langle \big(b(\hat{ u}^{n+\frac{1}{2}})-b(\hat{U}^{n+\frac{1}{2}})\big)\cdot A_t{ v}^{n}+b({ U}^{n+\frac{1}{2}})\cdot A_t{ e}_2^{n},A_t{ e}_3^{n}\rangle_h+\langle{ \xi}_3^n,A_t{ e}_3^{n}\rangle_h.
\end{align*}
By using Lemma \ref{2SG:lem4.2}, the H\"older inequality and the Cauchy mean value theorem, we can deduce
\begin{align}\label{2SG:eq4.27}
&||\big(b(\hat{ u}^{n+\frac{1}{2}})-b(\hat{ U}^{n+\frac{1}{2}})\big)\cdot A_t{ r}^{n}||_h\le ||A_t{ r}^{n}||_{h,\infty}||b(\hat{ u}^{n+\frac{1}{2}})-b(\hat{ U}^{n+\frac{1}{2}})||_h\nonumber\\
&~~~~~~~~~~~~~~~~~~~~~~~~~~~~~~~~~~~~~~~~~~\le C||\hat{ e}_1^{n+\frac{1}{2}}||_h\leq C(||{ e}_1^{n-1}||_h+||{ e}_1^{n}||_h),\\\label{2SG:eq:4.28}
&||\big(b(\hat{ u}^{n+\frac{1}{2}})-b(\hat{ U}^{n+\frac{1}{2}})\big)\cdot A_t{ v}^{n}||_h\le ||A_t{ v}^{n}||_{h,\infty}||b(\hat{ u}^{n+\frac{1}{2}})-b(\hat{ U}^{n+\frac{1}{2}})||_h\nonumber\\
&~~~~~~~~~~~~~~~~~~~~~~~~~~~~~~~~~~~~~~~~~~\le C||\hat{e}_1^{n+\frac{1}{2}}||_h\leq C(||{ e}_1^{n-1}||_h+||{ e}_1^{n}||_h),\\\label{2SG:eq:4.29}
&||b(\hat{ U}^{n+\frac{1}{2}})\cdot A_t{ e}_2^{n}||_h\le ||b(\hat{ U}^{n+\frac{1}{2}})||_{h,\infty}||A_t{ e}_2^{n}||_h\nonumber\\
&~~~~~~~~~~~~~~~~~~~~~~~~~\le C||A_t{e}_2^{n}||_h\leq C(||{ e}_2^{n}||_h+||{e}_2^{n+1}||_h),
\end{align}
and
\begin{align}\label{2SG:eq:4.30}
||b(\hat{U}^{n+\frac{1}{2}})\cdot A_t{e}_3^{n}||_h&\le ||b(\hat{ U}^{n+\frac{1}{2}})||_{h,\infty}||A_t{ e}_3^{n}||_h\nonumber\\
&\le C||A_t{e}_3^{n}||_h\leq C(||{e}_3^{n}||_h+||{ e}_3^{n+1}||_h).
\end{align}

%
%
From \eqref{2SG-lte:4.7}-\eqref{2SG-lte:4.8} and \eqref{2SG:eq4.27}-\eqref{2SG:eq:4.30}, we then have
\begin{align}\label{2SG:eq:4.31}
&\eta_1^n\le C\big(||{ e}_1^{n}||_h^2+||{ e}_1^{n+1}||_h^2+||{ e}_2^{n}||_h^2+||{ e}_2^{n+1}||_h^2
\big)+C(h^2+\tau^2)^2,\\\label{2SG:eq:4.32}
&\eta_2^n\le C\big( ||{ e}_1^{n-1}||_h^2+||{ e}_1^{n}||_{H_1^h}^2+||{ e}_1^{n+1}||_{H_h^1}^2+||{ e}_2^{n}||_h^2\nonumber\\
&~~~~~~~~~~~~~~~~~~~~~~~~~~~~~~+||{ e}_2^{n+1}||_h^2+||{ e}_3^{n}||_h^2+||{ e}_3^{n+1}|_h^2\big)+C(h^2+\tau^2)^2,
\end{align}
and
\begin{align}\label{2SG:eq:4.33}
\eta_3^n\le&C\big( ||{ e}_1^{n-1}||_h^2+||{ e}_1^{n}||_h^2+||{ e}_2^{n}||_h^2\nonumber\\
&+||{ e}_2^{n+1}||_h^2+||{ e}_3^{n}||_h^2+||{e}_3^{n+1}|_h^2\big)+C(h^2+\tau^2)^2.
\end{align}
We introduce the following ``energy function" for the error equations
\begin{align*}
F^n=||{ e}_1^n||_{H_h^1}^2+||{e}_2^n||_h^2+||{ e}_3^n||_h^2,\ 1\le n\le M.
\end{align*}
Adding \eqref{2SG:eq:4.26} and \eqref{2SG:eq:4.27} and noting \eqref{2SG:eq:4.32}-\eqref{2SG:eq:4.33}, we then obtain
\begin{align}\label{2SG:eq4.34}
\delta_t^+\big(||\delta_x{ e}_1^n||_h^2&+||\delta_y{ e}_1^n||_h^2+||{ e}_2^n||_h^2+||{ e}_3^n||_h^2\big)\nonumber\\
&\le C(||{ e}_1^{n-1}||_{h}^2+||{ e}_1^n||_{H_h^1}^2+||{ e}_1^{n+1}||_{H_h^1}^2+||{ e}_2^{n}||_{h}^2\nonumber\\
&+||{e}_2^{n+1}||_{h}^2+||{ e}_3^n||_h^2+||{ e}_3^{n+1}||_h^2)+C(h^2+\tau^2)^2.
\end{align}
Combing \eqref{2SG:eq:4.25} and \eqref{2SG:eq4.34}, together with \eqref{2SG:eq:4.31}, we have
\begin{align}\label{2SG:eq:4.34}
F^{n}-F^{n-1}&\le C\tau (F^{n-1}+F^{n})+C\tau||{ e}_1^{n-2}||_h^2+C\tau(h^2+\tau^2)^2.
\end{align}
Summing up for the superscript $n$ from 2 to $m$ and then replacing $m$ by $n$, we can deduce from \eqref{2SG:eq:4.34} that
\begin{align}\label{2SG:eq:4.35}
F^{n}&\le F^1+C\tau \sum_{l=1}^nF^l+C\tau||{ e}_1^0||_h^2+CT(h^2+\tau^2)^2\nonumber\\
&\le C\tau \sum_{l=1}^nF^l+CT(h^2+\tau^2)^2.
\end{align}
Applying Lemma \ref{2SG:lem4.1} to Eq. \eqref{2SG:eq:4.35}, we get
 \begin{align*}
 ||{e}_{1}^n||_{H_h^1}^2+||{e}_2^n||_h^2+||{e}_3^n||_h^2\le C e^{2CT}(h^2+\tau^2)^2,
 \end{align*}
 which further implies that
 \begin{align*}
 ||{e}_{1}^n||_{H_h^1}+||{e}_2^n||_h+||{e}_3^n||_h\le C(h^2+\tau^2),
 \end{align*}
 where $\tau$ is sufficiently small, such that $C\tau \le \frac{1}{2}$. 
The proof is completed.\qed
\end{prf}

\begin{rmk}\rm It is noted that we have used the error estimate for the first step in \eqref{2SG:eq:4.35} i.e.,
\begin{align*}
||{e}_{1}^1||_{H_h^1}+||{ e}_2^1||_h+||{ e}_3^1||_h\le C(h^2+\tau^2).
\end{align*}
which can be established by the similar argument in Theorem \ref{2SG:thm4.2}. 
\end{rmk}

\section{Numerical examples}\label{2SG:Sec5}
In this section, we report the performance of the linear-implicit and local energy-preserving scheme \eqref{2SG:eq:3.1}-\eqref{2SG:eq:3.3} (denoted by LI-LEPS) for the sine-Gordon equation in one and two dimensions, respectively. The motivation of this paper is to provide a novel procedure to develop local structure-preserving schemes, thus,  it is valuable to compare our new scheme with some existing local structure-preserving schemes of same order in both space and time, as follows:
\begin{itemize}
\item SFDS: the symplectic midpoint finite difference scheme stated in Refs. \cite{CGM12,WWQ08} for the one dimensional sine-Gordon equation;
\item MBS: the multi-sympletic box scheme described in Refs. \cite{Reich00,SW08,WWQ08} for the one dimensional sine-Gordon equation;
\item EP-FDS: the energy-preserving finite difference scheme proposed in Refs. \cite{CGM12,LQ95,WWQ08} for the one dimensional sine-Gordon equation and in Ref. \cite{KFCG17} for the two dimensional case, respectively.
\end{itemize}
As a summary, a detailed table on the properties of each scheme has been
given in Tab. \ref{Tab_2SG:1}.

In the practical computation, we use standard fixed-point iteration for the fully implicit schemes and the preconditioned conjugate gradients method (named pcg in MATLAB functions) for the linear system (see \eqref{LIM-EP:3.11}) given by LI-LEPS. We also remark that the preconditioned conjugate gradients method is also adopt as the solver of linear systems given by the fully implicit schemes for each fixed-point iteration and we set $10^{-14}$ as the error tolerance for all the problems. In what follows spatial mesh steps are uniformly chosen as $h_1=h_2=h$ for simplicity, $L^2$- and $L^{\infty}$-errors are denoted as the $L^2$- and $L^{\infty}$-norms of the error between the numerical solution $U_{j_1,j_2}^n$ and the exact solution $u(x_{j_1},y_{j_2},t_n)$, respectively, and the energy deviation represents the relative energy error.
%
%

\begin{table}[H]
\tabcolsep=9pt
\small
\renewcommand\arraystretch{1.1}
\centering
\caption{Comparison of properties of different numerical schemes}\label{Tab_2SG:1}
\begin{tabular*}{\textwidth}[h]{@{\extracolsep{\fill}}c c c c c c}\hline 
 \diagbox{Property}{Scheme}& LI-LEPS& SFDS& MBS & EP-FDS\\\hline
 Multi-symplectic conservation & No&Yes&Yes&No   \\ [1ex]  
 Local energy conservation& Yes&No&No&Yes \\[1ex]
 Fully implicit&No&Yes&Yes&Yes\\[1ex]
 Linearly implicit&Yes&No&No&No\\[1ex]
\hline
\end{tabular*}
\end{table}

\subsection{One dimensional nonlinear sine-Gordon equation}
In this section, we first report the performance of LI-LEPS for the one dimensional nonlinear sine-Gordon equation:
\begin{align}\label{1SG:eq:5.1}
u_{tt}=u_{xx}-\sin(u),\ -20\le x\le 20,
\end{align}
with the initial conditions
\begin{align*}
&f(x)=0,\\
&g(x)=4\text{sech}(x).
\end{align*}
and the periodic boundary condition. Eq. \eqref{1SG:eq:5.1} possesses the analytical solution \cite{BCI15,BI16}
\begin{align*}
u(x,t)=4\tan^{-1}\big[t\text{sech}(x)\big].
\end{align*}



 The errors and convergence orders of LI-LEPS, SFDS, MBS and EP-FDS at time $t=1$ are given in Tab. \ref{Tab_2SG:2}. From Tab. \ref{Tab_2SG:2}, we can draw the following observations: (i) All schemes have second order accuracy in time and space errors; (ii) The error provided by the MBS is largest, while the error provided by SFDS is smallest; (iii) The error provided by LI-LEPS has the same order of magnitude as the one provided by SFDS.

 In Fig. \ref{figure-1}, we carry out comparisons on the computational costs
among the four schemes by refining the mesh size gradually. It is clear to see that the cost of MBS is most expensive while the one of LI-LEPS is cheapest.
Moreover, as the refinement of mesh sizes, the advantage of LI-LEPS emerges, which implies that our scheme is more preferable for large scale simulations than the structure-preserving schemes SFDS, MBS and EP-FDS.

In Fig. \ref{figure-2}, we display double-polesolution (left plot) and the energy error deviation (right plot). It is clear to observe that LI-LEPS can preserve the shape of the soliton accurately and shows a remarkable advantage in the conservation of energy.

\begin{table}[H]
\tabcolsep=9pt
\small
\renewcommand\arraystretch{1.1}
\centering
\caption{{Numerical error and convergence order of different schemes at time $t=1$.}}\label{Tab_2SG:2}
\begin{tabular*}{\textwidth}[h]{@{\extracolsep{\fill}}c l l l l l}\hline
{Scheme\ \ }&{($h,\tau$)}  &{$L^{2}$-error} & {order}  & {$L^{\infty}$-error} & {order} \\     
\hline

 \multirow{4}{*}{LI-LEPS}&{($\frac{1}{10},\frac{1}{100}$)}  & {1.2515e-03} & {-} & {1.3017e-03}&{-}\\[1ex]
 {}&{($\frac{1}{20},\frac{1}{200}$)}  & {3.1285e-04} & {2.00} & {3.2508e-04}&{2.00}\\[1ex]
  {}&{($\frac{1}{40},\frac{1}{400}$)} & {7.8211e-05} & {2.00}  & {8.1248e-05} & {2.00}  \\[1ex]   
 {}&{($\frac{1}{80},\frac{1}{800}$)}  & {1.9553e-05} & {2.00} & {2.0311e-05}&{2.00} \\\hline
 \multirow{4}{*}{SFDS}&{($\frac{1}{10},\frac{1}{100}$)}  & {1.1039e-03} & {-} & {1.0385e-03}&{-}\\[1ex]
 {}&{($\frac{1}{20},\frac{1}{200}$)}  & {2.7595e-04} & {2.00} & {2.5927e-04}&{2.00}\\[1ex]
  {}&{($\frac{1}{40},\frac{1}{400}$)} & {6.8986e-05} & {2.00}  & {6.4795e-05} & {2.00}  \\ [1ex]  
 {}&{($\frac{1}{80},\frac{1}{800}$)}  & {1.7246e-05} & {2.00} & {1.6197e-05}&{2.00} \\\hline
 \multirow{4}{*}{MBS}&{($\frac{1}{10},\frac{1}{100}$)}  & {2.1309e-03} & {-} & {2.1556e-03}&{-}\\[1ex]
 {}&{($\frac{1}{20},\frac{1}{200}$)}  & {5.3158e-04} & {2.00} & {5.3902e-04}&{2.00}\\[1ex]
  {}&{($\frac{1}{40},\frac{1}{400}$)} & {1.3282e-04} & {2.00}  & {1.3476e-04} & {2.00}  \\ [1ex]  
 {}&{($\frac{1}{80},\frac{1}{800}$)}  & {3.3201e-05} & {2.00} & {3.3691e-05}&{2.00} \\\hline
 \multirow{4}{*}{EP-FDS}&{($\frac{1}{10},\frac{1}{100}$)}  & {1.1112e-03} & {-} & {1.0535e-03}&{-}\\[1ex]
 {}&{($\frac{1}{20},\frac{1}{200}$)}  & {2.7777e-04} & {2.00} & {2.6301e-04}&{2.00}\\[1ex]
  {}&{($\frac{1}{40},\frac{1}{400}$)} & {6.9442e-05} & {2.00}  & {6.5729e-05} & {2.00}  \\  [1ex] 
 {}&{($\frac{1}{80},\frac{1}{800}$)}  & {1.7360e-05} & {2.00} & {1.6431e-05}&{2.00} \\\hline
\end{tabular*}
\end{table}

\begin{figure}[H]
\centering \begin{minipage}[t]{80mm}
\includegraphics[width=80mm]{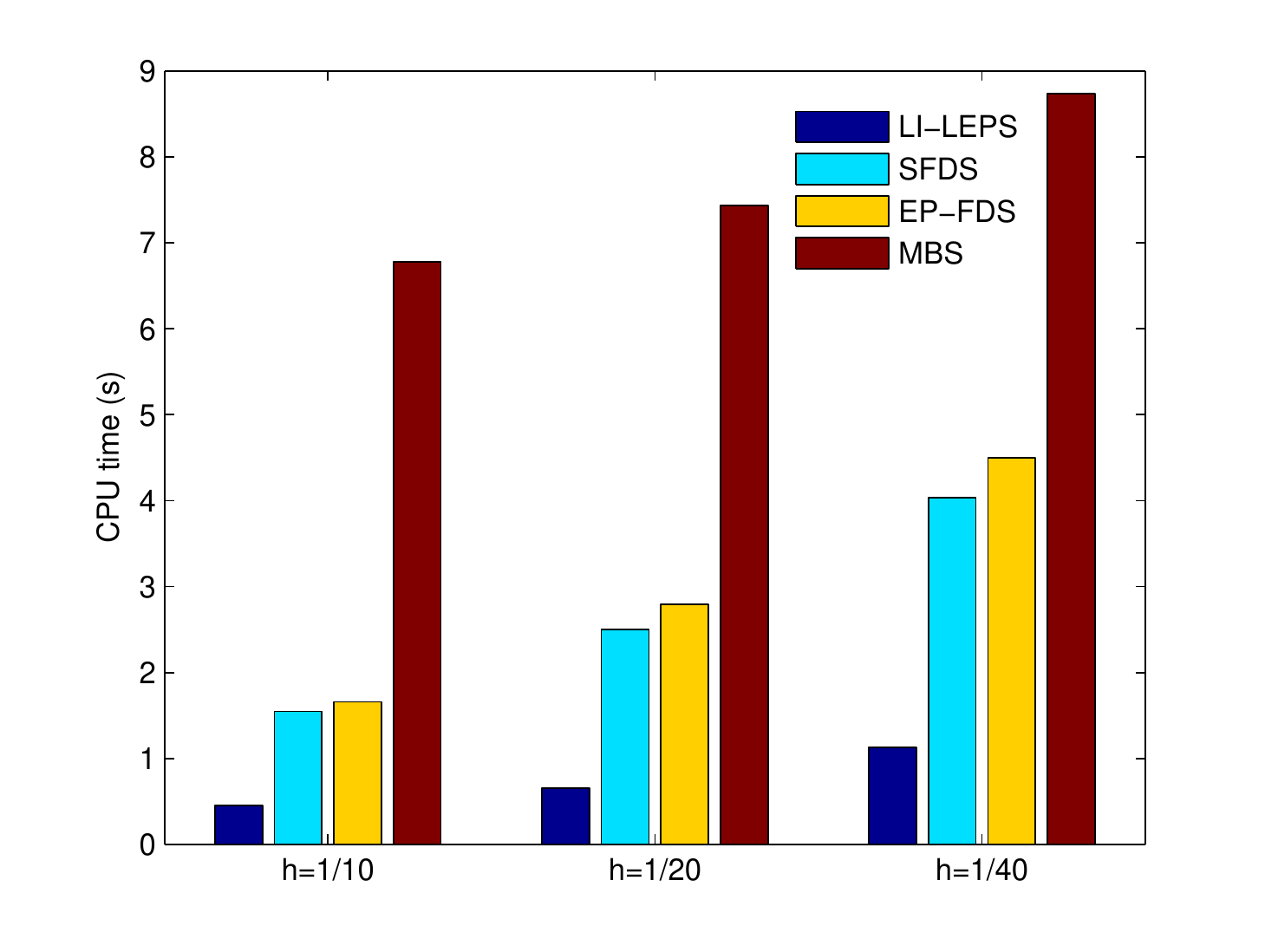}
\end{minipage}
\caption{\footnotesize CPU time of the four schemes for the soliton with different mesh
sizes till $t=1$ under $\tau=0.01$. The computation is carried out via Matlab 7.0 with AMD A8-7100 and RAM 4GB.}\label{figure-1}
\end{figure}


\begin{figure}[H]
\centering \begin{minipage}[t]{70mm}
\includegraphics[width=70mm]{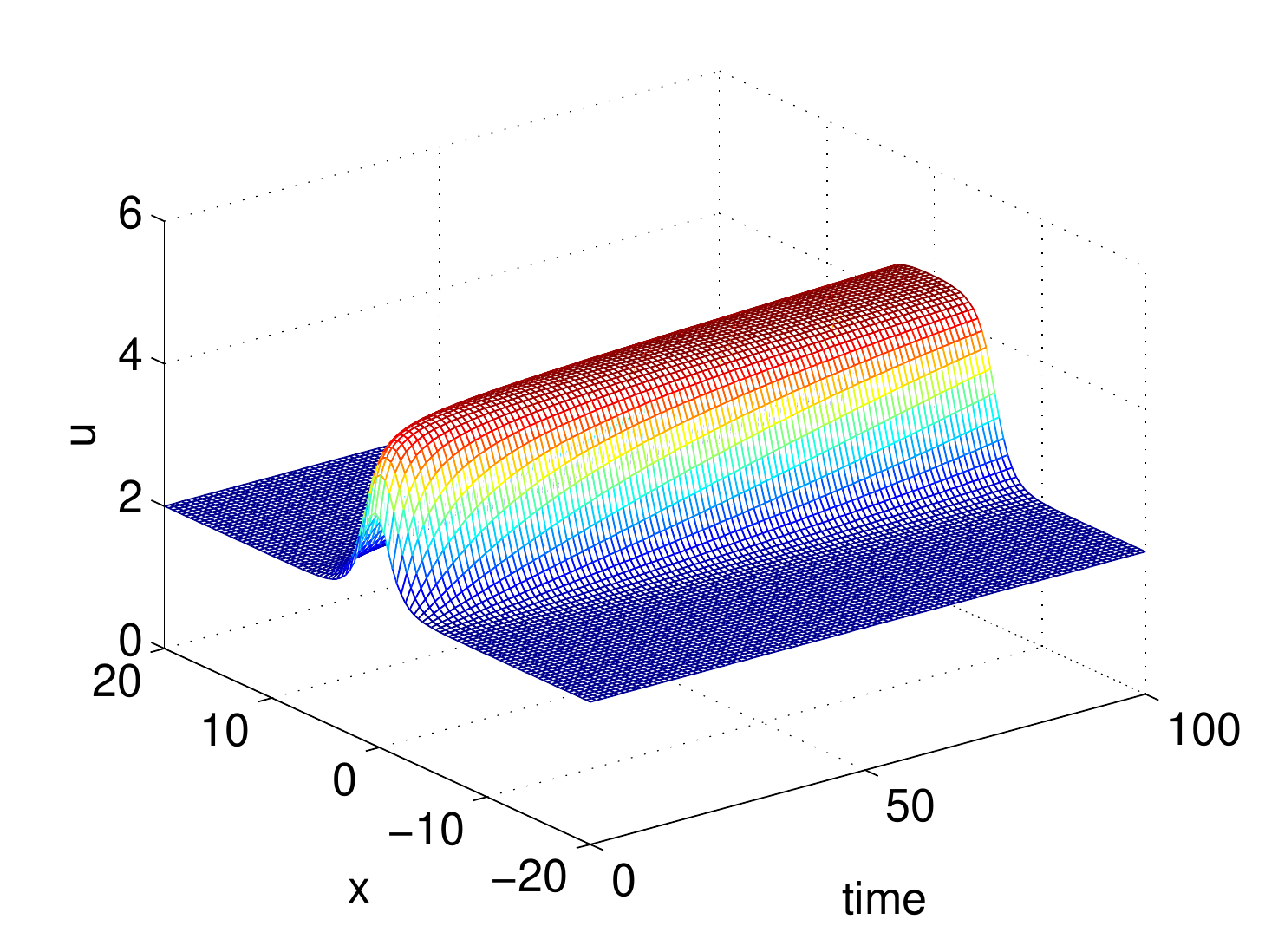}
\end{minipage}
\begin{minipage}[t]{70mm}
\includegraphics[width=70mm]{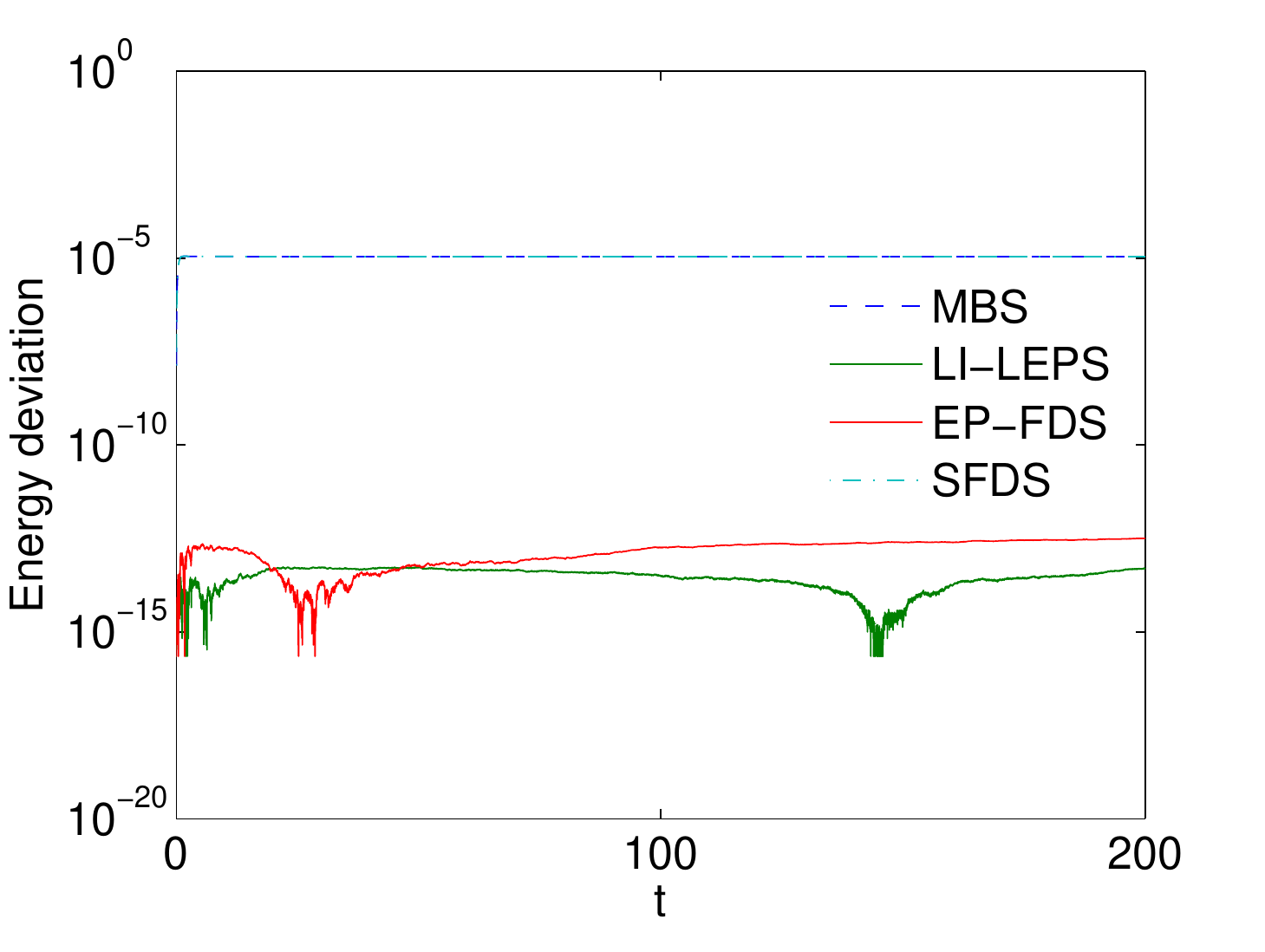}
\end{minipage}
\caption{\footnotesize Double-polesolution (left plot), and the energy deviation (right plot). Spatial and temporal mesh sizes are taken as $h=0.1$ and $\tau=0.01$, respectively.}\label{figure-2}
\end{figure}

\subsection{Two dimensional nonlinear sine-Gordon equation}
In this section, we show the performance of LI-LEPS for the two dimensional nonlinear sine-Gordon equation:
\begin{align}\label{2SG:eq:5.6}
u_{tt}=u_{xx}+u_{yy}-\sin(u),\ (x,y)\in\Omega.
\end{align}
\subsubsection{Accuracy test}
In this example, we consider the sine-Gordon equation \eqref{2SG:eq:5.6} with the initial conditions
\begin{align*}
&f(x,y)=4\tan^{-1}(\exp(x+y)),\ -7\le x,y\le 7,\\
&g(x,y)=-\frac{4\exp(x+y)}{1+\exp(2x+2y)},\ -7\le x,y\le 7,
\end{align*}
and the boundary conditions
\begin{align*}
&u(x,y,t)=4\tan^{-1}(\exp(x+y-t)),\ \text{for}\ x=-7\ \text{and} \ x=7, -7\le y\le 7,\ t>0,\\
&u(x,y,t)=4\tan^{-1}(\exp(x+y-t)),\ \text{for}\ y=-7\ \text{and} \ y=7, -7\le x\le 7,\ t>0.
\end{align*}
Eq. \eqref{2SG:eq:5.6} possesses the analytical solution
\begin{align*}
u(x,y,t)=4\tan^{-1}(\exp(x+y-t)).
\end{align*}

 The comparison of the spatial and temporal errors of LI-LEPS and EP-FDS at time $t=1$ is displayed in Tab. \ref{Tab_2SG:4}, which shows that LI-LEPS and EP-FDS have second order accuracy in time and space errors and the error provided by LI-LEPS is same order of magnitude as the one provided by EP-FDS. In Fig. \ref{fig3}, we carry out comparisons on the computational costs
between LI-LEPS and EP-FDS by refining the mesh size gradually, which behaves similarly as that of Fig. \ref{figure-1}.

%

\begin{table}[H]
\tabcolsep=9pt
\small
\renewcommand\arraystretch{1.2}
\centering
\caption{{Numerical error and convergence order of LI-LEPS and EP-FDS at time $t=1$.}}\label{Tab_2SG:4}
\begin{tabular*}{\textwidth}[h]{@{\extracolsep{\fill}}c l l l l l}\hline
{Scheme}&{$(h,\tau)$}  &{$L^{2}$-error} & {order}  & {$L^{\infty}$-error} & {order} \\ \hline    
\multirow{4}{*}{LI-LEPS}&{($\frac{1}{2}$,$\frac{1}{100}$)}  & {1.2129e-01} & {-} & {2.7812e-02}&{-} \\[1ex]
{}& {($\frac{1}{4}$,$\frac{1}{200}$)}  & {3.0043e-02} & {2.01} & {7.8107e-03}&{1.83}\\[1ex]
 {}&{($\frac{1}{8}$,$\frac{1}{400}$)}  & {7.4920e-03} & {2.00} & {1.9545e-03}&{2.00}\\[1ex]
  {}&{($\frac{1}{16}$,$\frac{1}{800}$)} & {1.8718e-03} & {2.00}  & {4.8891e-04} & {2.00}  \\ \hline  
\multirow{4}{*}{EP-FDS}& {($\frac{1}{2}$,$\frac{1}{100}$)}  & {1.2132e-01} & {-} & {2.7774e-02}&{-}\\[1ex]
 {}&{($\frac{1}{4}$,$\frac{1}{200}$)}  & {3.0049e-02} & {2.01} & {7.8030e-03}&{1.83}\\[1ex]
  {}&{($\frac{1}{8}$,$\frac{1}{400}$)} & {7.4935e-03} & {2.00}  & {1.9525e-03} & {2.00}  \\ [1ex]  
{}&{($\frac{1}{16}$,$\frac{1}{800}$)}  & {1.8722e-03} & {2.00} & {4.8840e-04}&{2.00} \\\hline
\end{tabular*}
\end{table}
\begin{figure}[H]
\centering \begin{minipage}[t]{80mm}
\includegraphics[width=80mm]{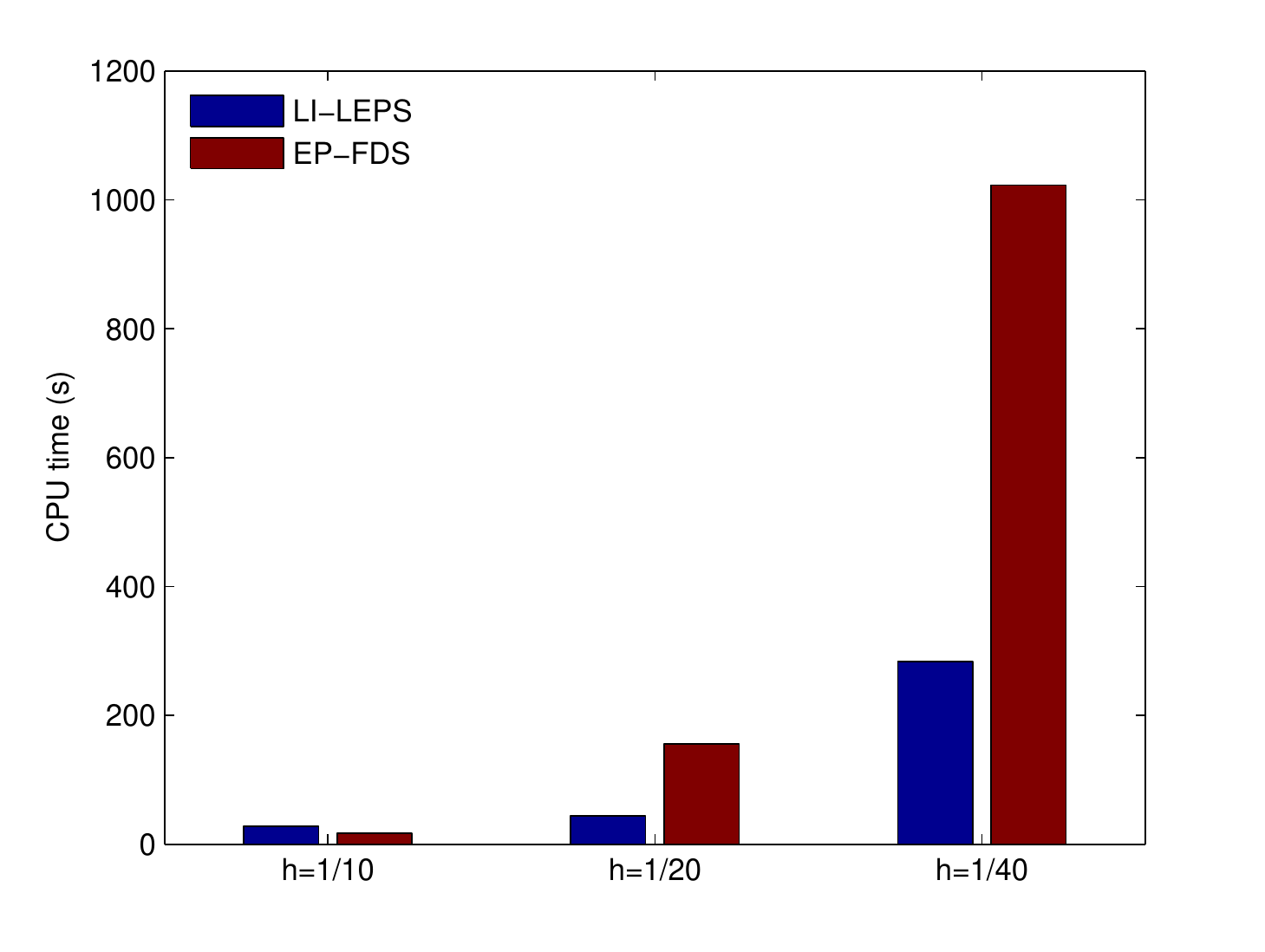}
\end{minipage}
\caption{\footnotesize CPU time of the two schemes for the soliton with different mesh
sizes till $t=1$ under $\tau=0.01$. The computation is carried out via Matlab 7.0 with AMD A8-7100 and RAM 4GB.}\label{fig3}
\end{figure}

\subsubsection{Circular ring solitons}

We then consider the circular ring solitons with the initial conditions \cite{Argyris91,CL81,DPT95,SKV10}
\begin{align*}
&f(x,y)=4\tan^{-1}\left[\exp\left(3-\sqrt{x^{2}+y^{2}}\right)\right],\ -14\le x,y\le 14,\\
& g(x,y)=0,\ -14\le x,y\le 14,
\end{align*}
and the periodic boundary conditions.

Fig. \ref{2SG:fig2} presents the initial condition as well as numerical solutions at times $t=0,4,8,11.5$ and $t=15$, respectively. The ring solitons simulated shrink at the initial
stage, but oscillations and radiations begin to form and continue slowly as time goes on. This fact is consistent with the results obtained in Refs. \cite{Argyris91,CL81,DPT95,SKV10} and can be clearly viewed in the contour plots. The energy deviations of two
schemes over the time interval $t\in[0,50]$ are given in Fig. \ref{2SG:err1},
which shows that the two schemes can precisely preserve the energy in long time computation, and the error provided by LI-LEPS is much smaller than the one provided by the EP-FDS.

\begin{figure}[H]
\centering\begin{minipage}[t]{60mm}
\includegraphics[width=60mm]{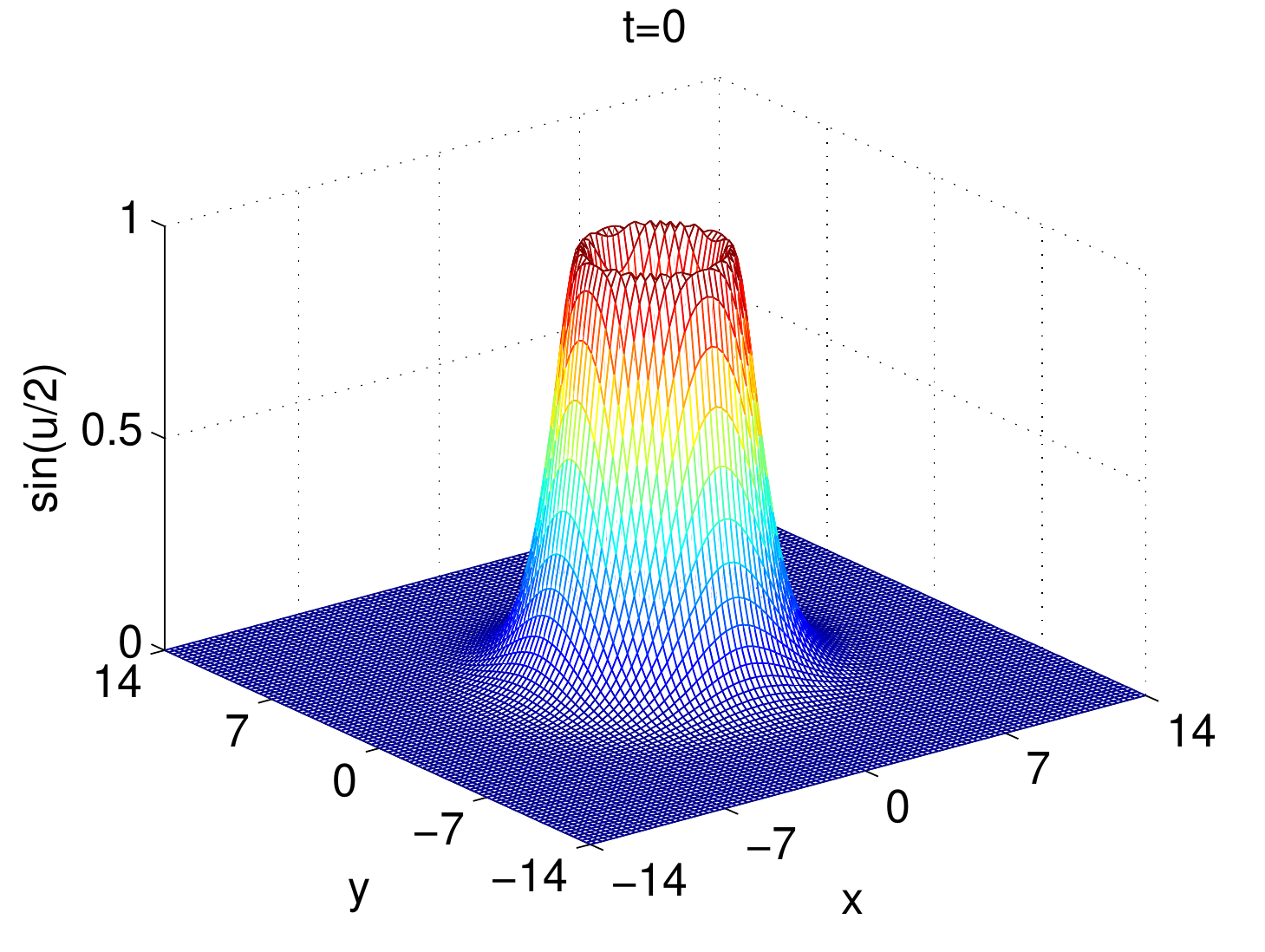}
\end{minipage}
\begin{minipage}[t]{60mm}
\includegraphics[width=60mm]{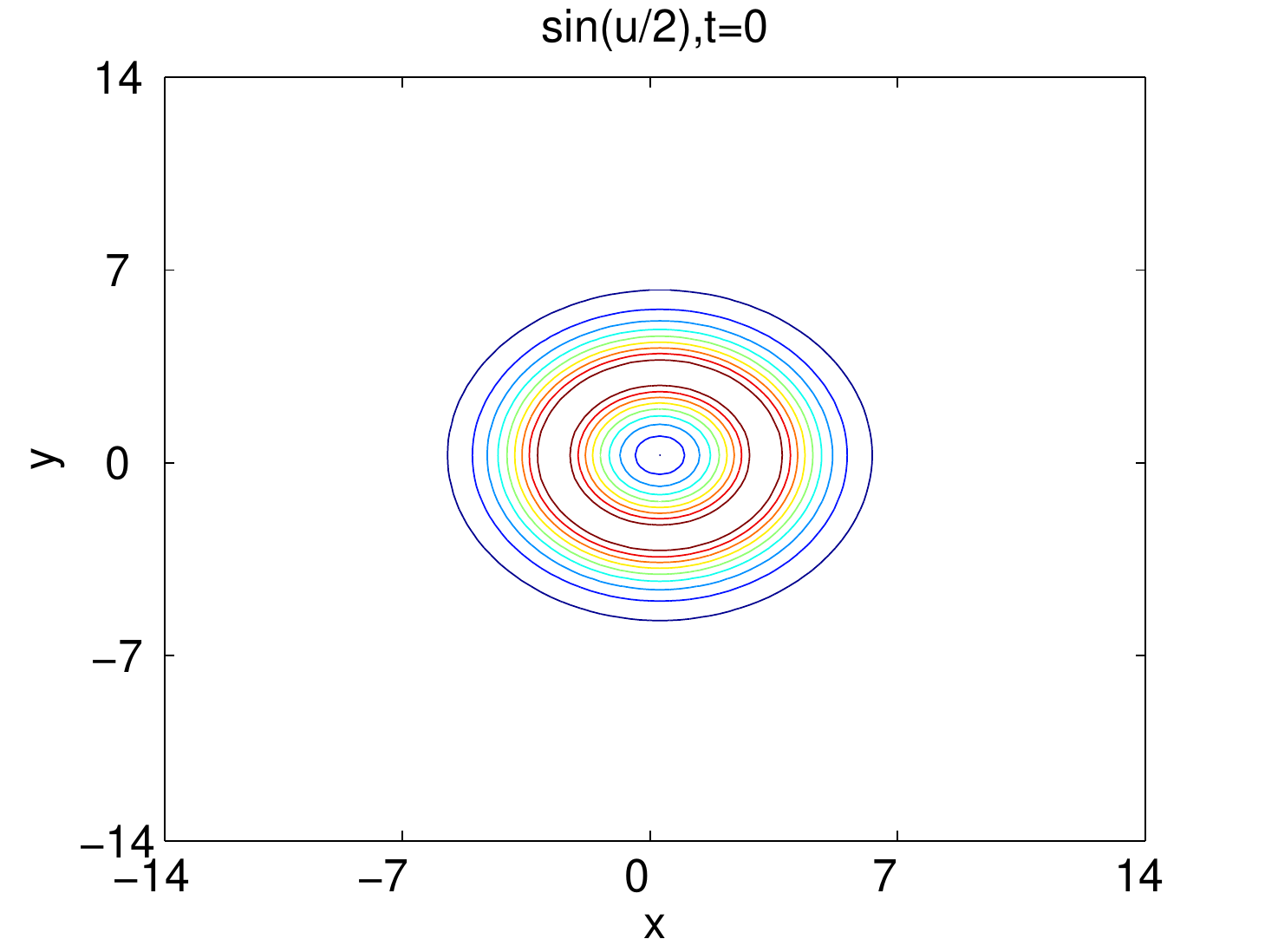}
\end{minipage}
\begin{minipage}[t]{60mm}
\includegraphics[width=60mm]{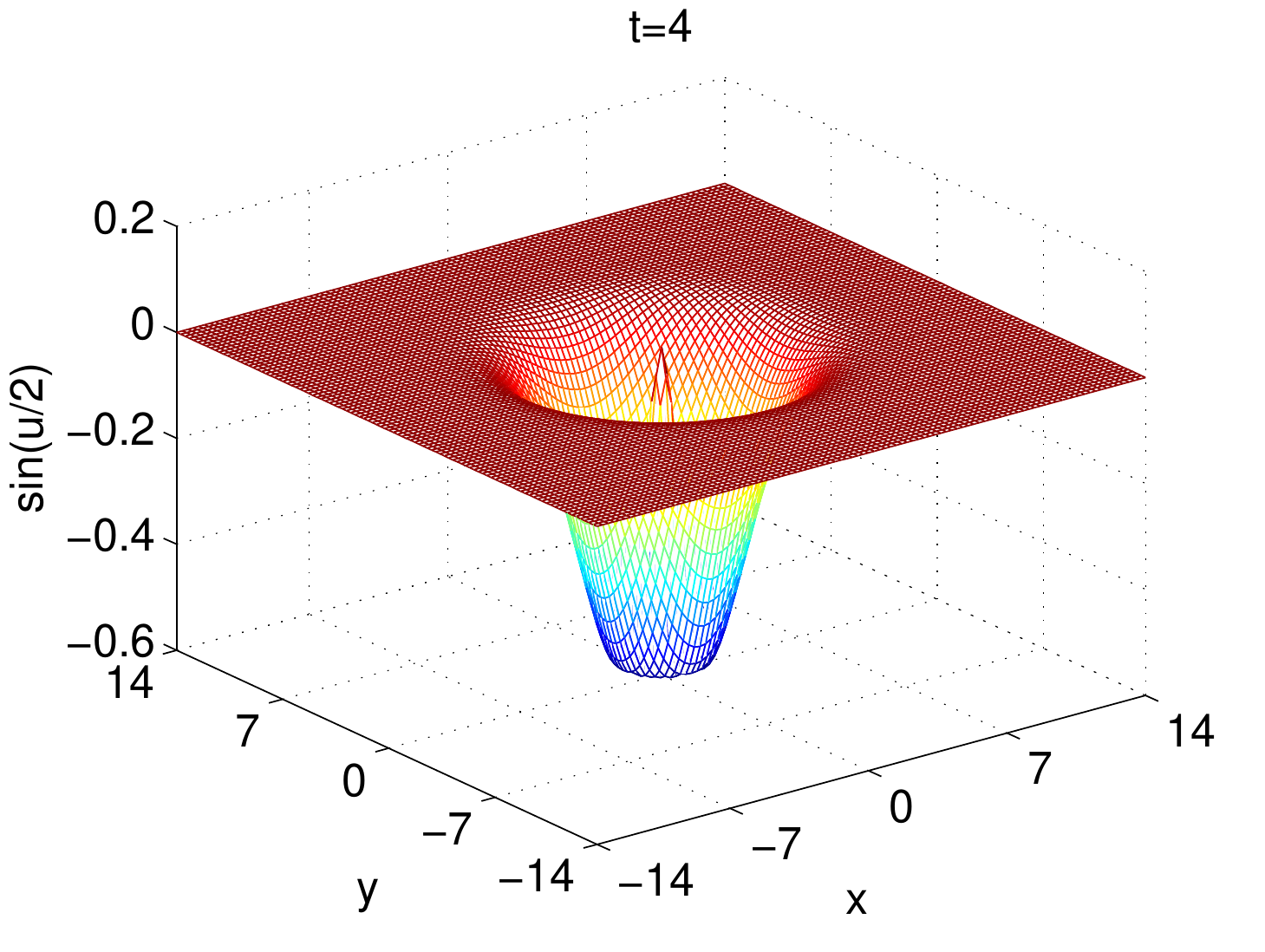}
\end{minipage}
\begin{minipage}[t]{60mm}
\includegraphics[width=60mm]{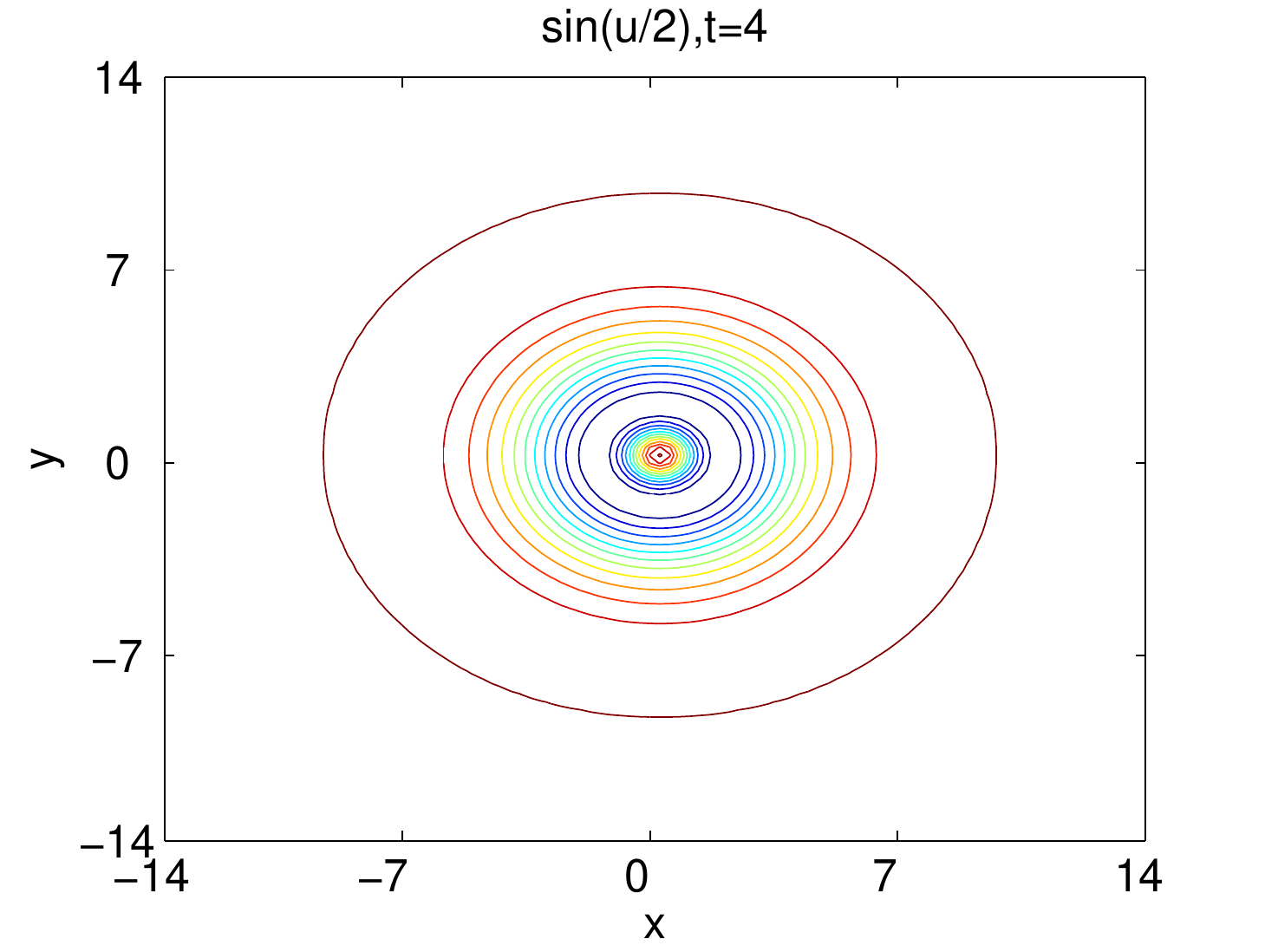}
\end{minipage}
\begin{minipage}[t]{60mm}
\includegraphics[width=60mm]{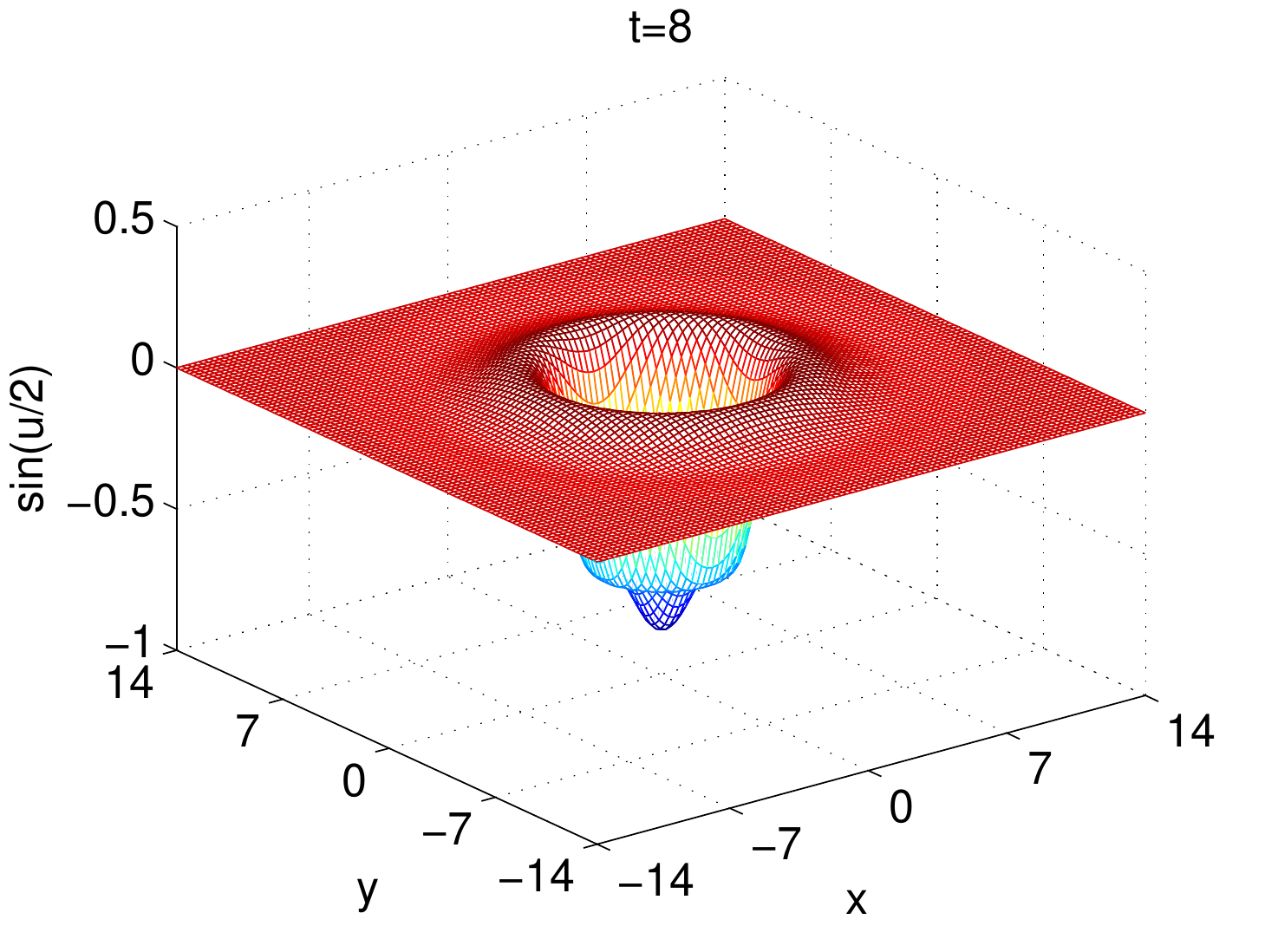}
\end{minipage}
\begin{minipage}[t]{60mm}
\includegraphics[width=60mm]{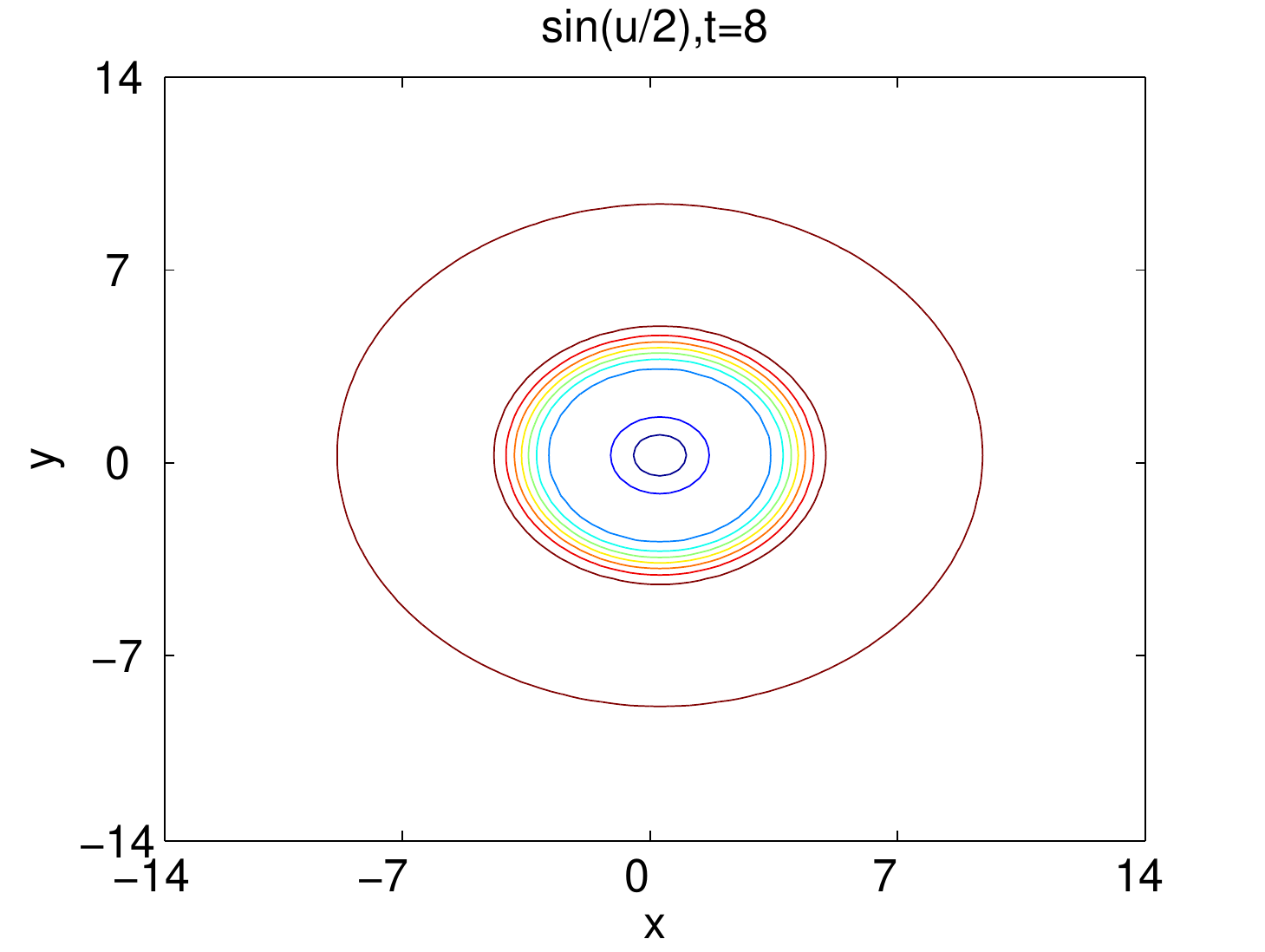}
\end{minipage}
\begin{minipage}[t]{60mm}
\includegraphics[width=60mm]{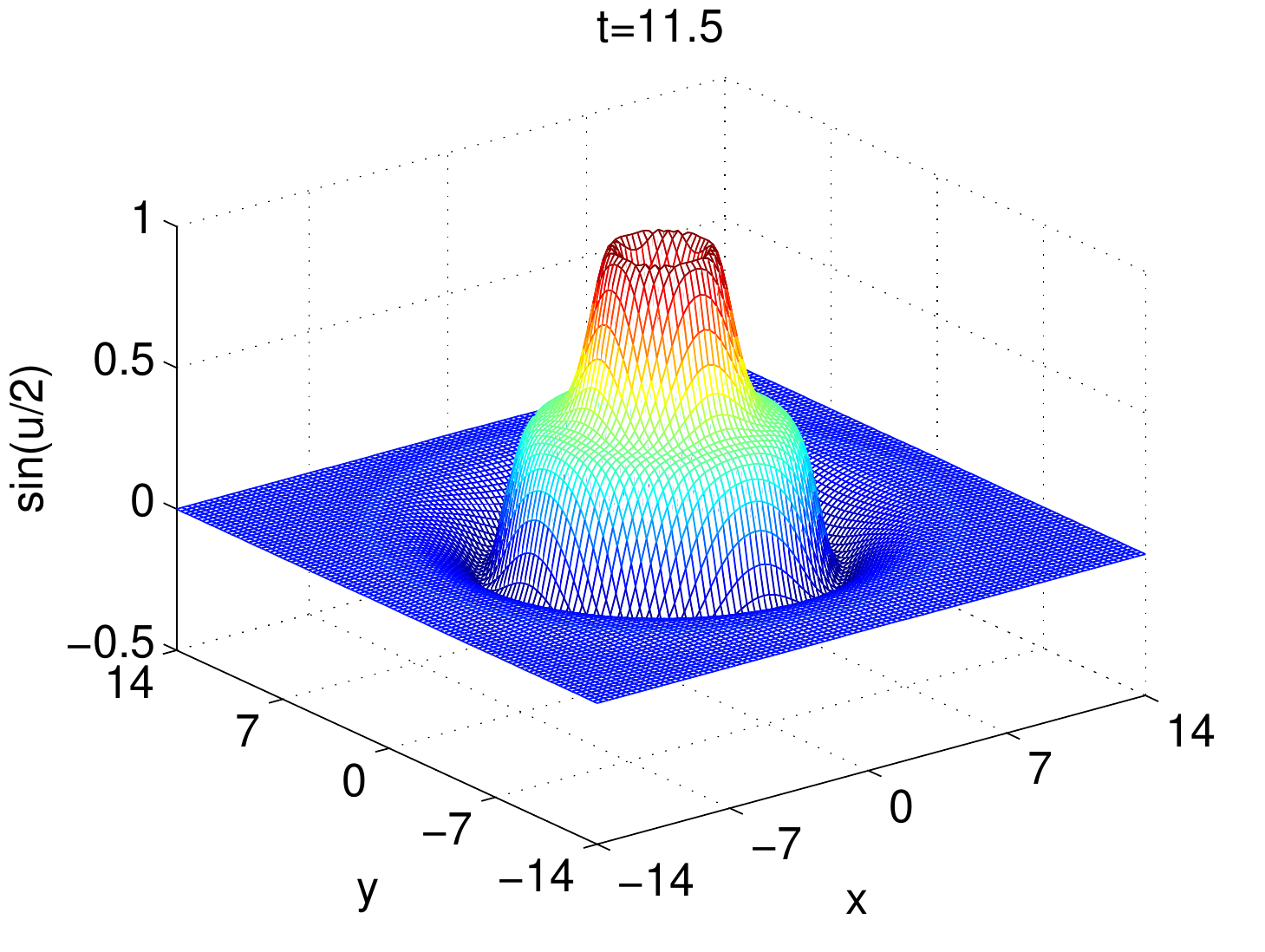}
\end{minipage}
\begin{minipage}[t]{60mm}
\includegraphics[width=60mm]{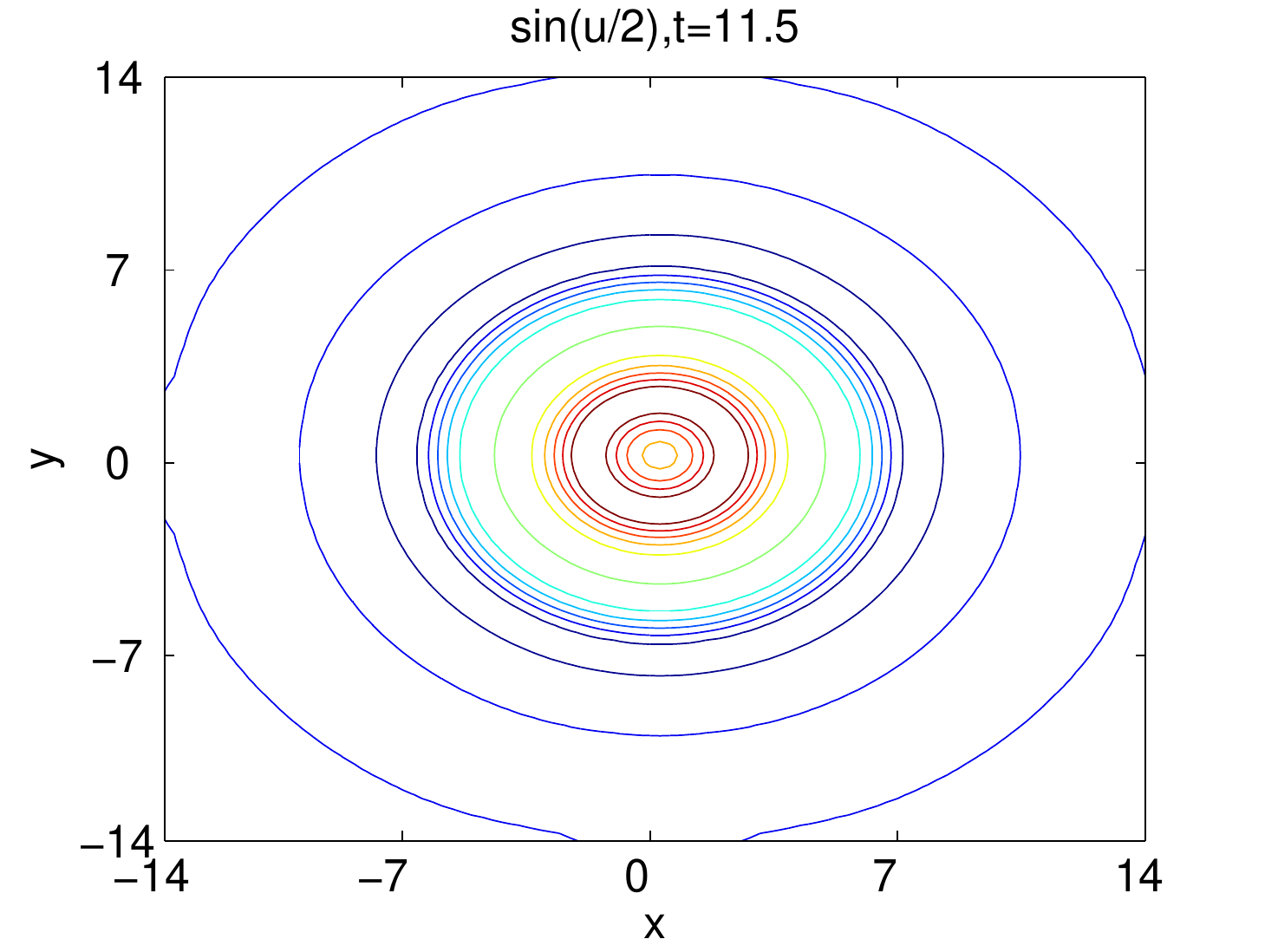}
\end{minipage}
\end{figure}
\begin{figure}
\begin{minipage}[t]{60mm}
\includegraphics[width=60mm]{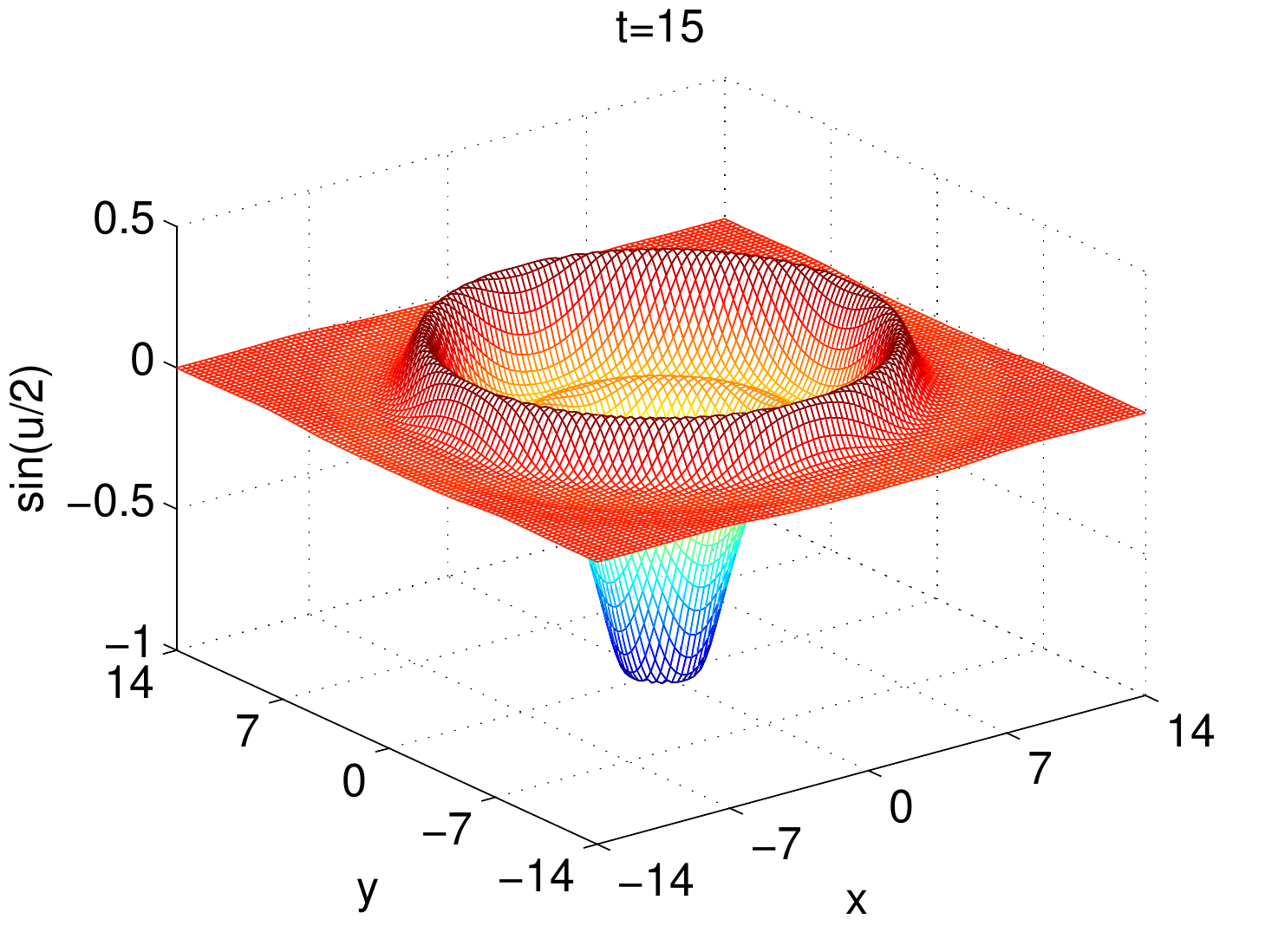}
\end{minipage}
\begin{minipage}[t]{60mm}
\includegraphics[width=60mm]{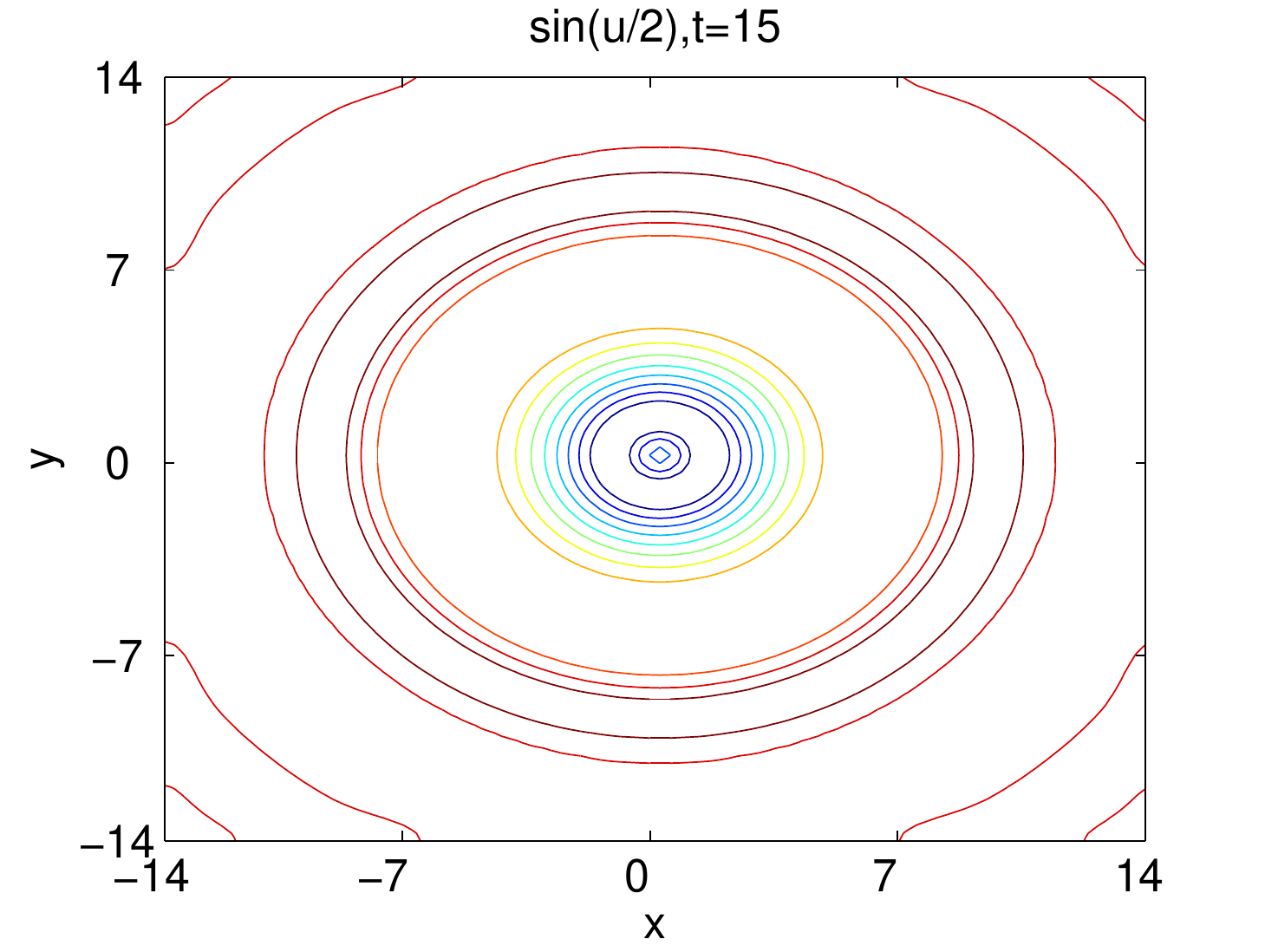}
\end{minipage}
\caption{ Circular ring soliton: surface and contours plots of initial condition and numerical solutions at times $t=0,4,8,11.5$ and $t=15$, respectively, in terms of $\sin(u/2)$. Spatial and temporal mesh sizes are taken as $h=0.14$ and $\tau=0.01$.}\label{2SG:fig2}
\end{figure}

\begin{figure}[H]
\centering\begin{minipage}[t]{70mm}
\includegraphics[width=70mm]{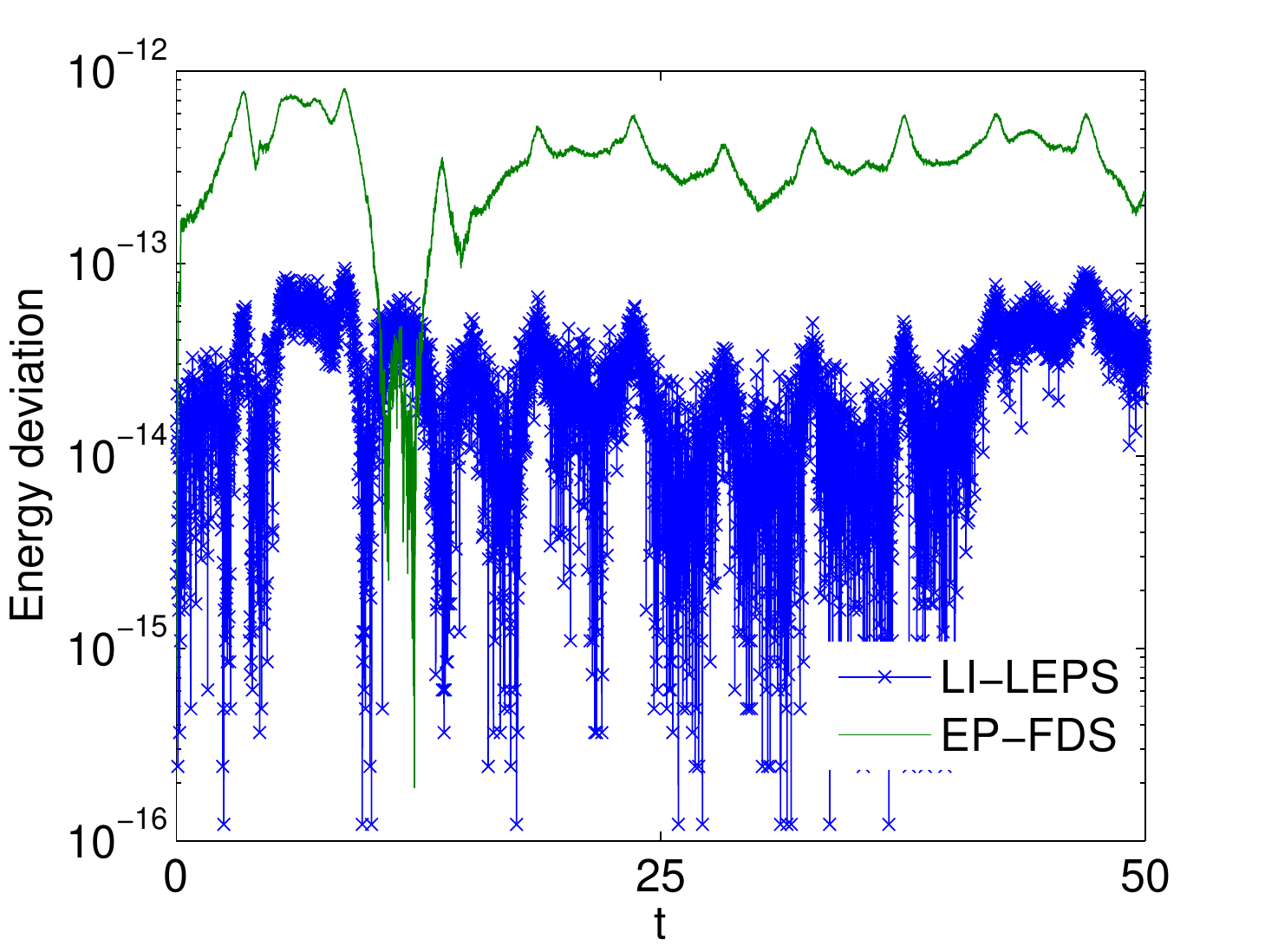}
\end{minipage}
\caption{Energy deviation over the time interval $t\in[0,50]$
with $h=0.14$ and $\tau=0.01$.}\label{2SG:err1}
\end{figure}
\subsubsection{Elliptical breather}
Next, elliptical ring solitons are obtained by the initial conditions \cite{Argyris91,Bratsos07b,CL81}
\begin{align*}
&f(x,y)=4\tan^{-1}\Bigg[2\text{sech}\Bigg(0.866\sqrt{\frac{(x-y)^2}{3}+{\frac{(x+y)^2}{2}}}\Bigg)\Bigg],\ -7\le x,y\le 7,\\
& g(x,y)=0,\ -7\le x,y\le 7,
\end{align*}
and the periodic boundary conditions.

The initial condition and numerical solutions at times $t=0,1.6,8,9.6,11.2,12.8,14.4$ and $t=15.2$, respectively, are displayed in Fig. \ref{EBS:1}. The major axis of the breather from its initial direction $y=-x$ seems to be turned clockwise. Then a shrinking and and a reflexion phase are observed. At $t=11.2$ the major axis has almost recovered its initial direction $y=-x$ but strong oscillations are observed. Finally, at $t=12.8$ an expansion phase is observed. The
results are in agreement with those given in Refs. \cite{Argyris91,Bratsos07b,CL81}. The long time energy deviations of the two
schemes also given in Fig. \ref{2SG:err2}, which demonstrates a consistent result as that in circular ring solitons.
\begin{figure}[H]
\centering\begin{minipage}[t]{60mm}
\includegraphics[width=60mm]{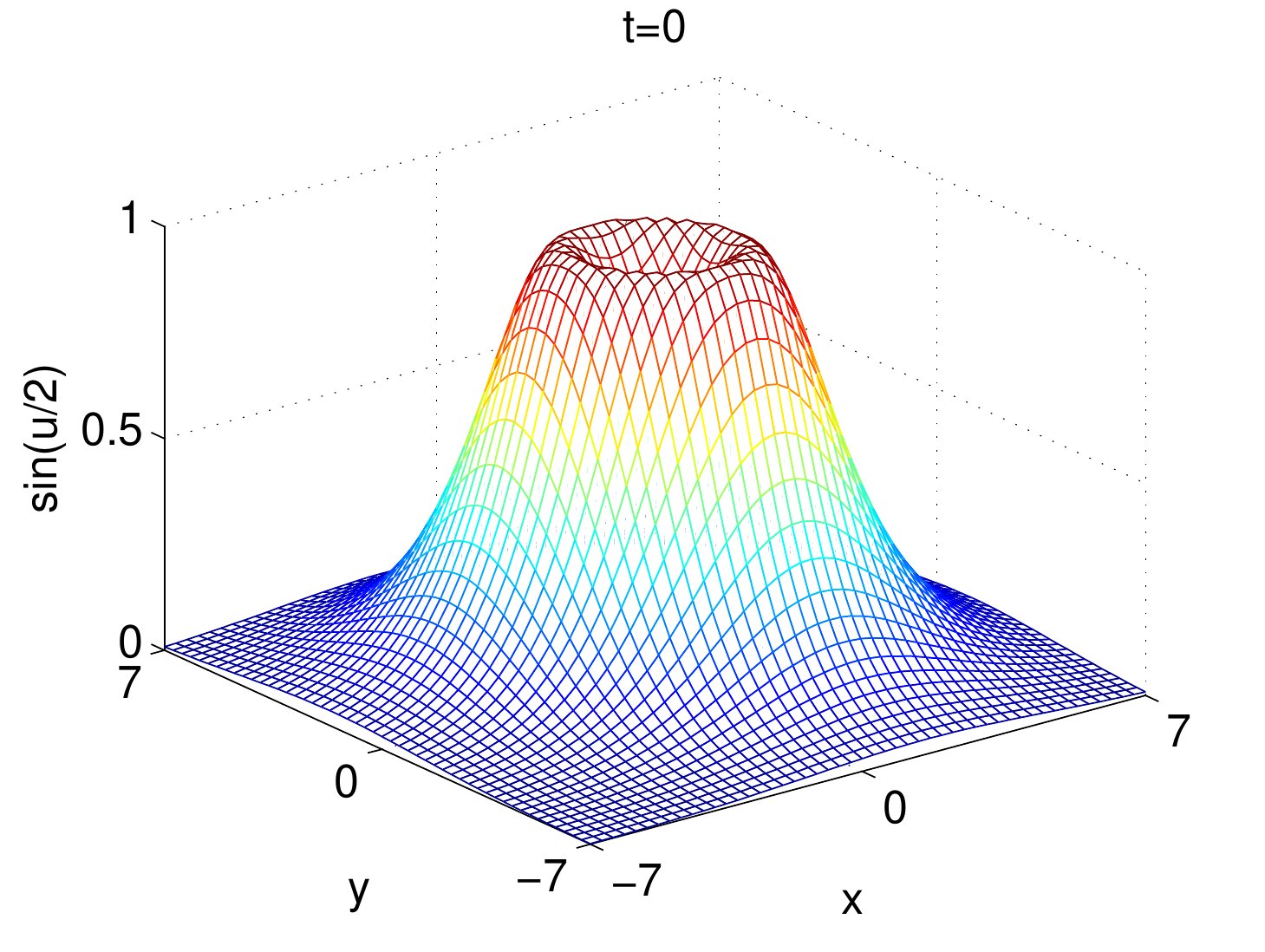}
\end{minipage}
\begin{minipage}[t]{60mm}
\includegraphics[width=60mm]{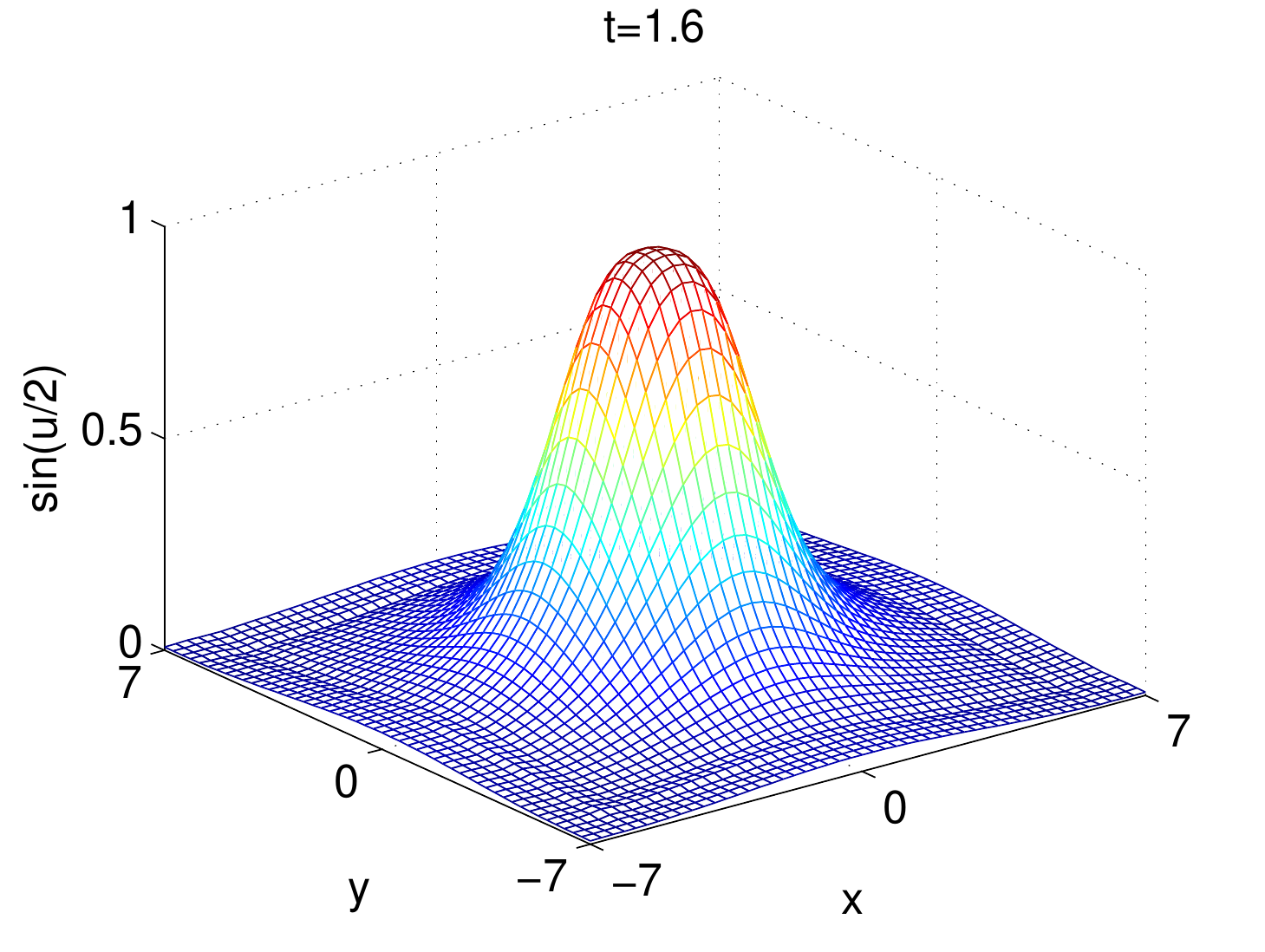}
\end{minipage}
\begin{minipage}[t]{60mm}
\includegraphics[width=60mm]{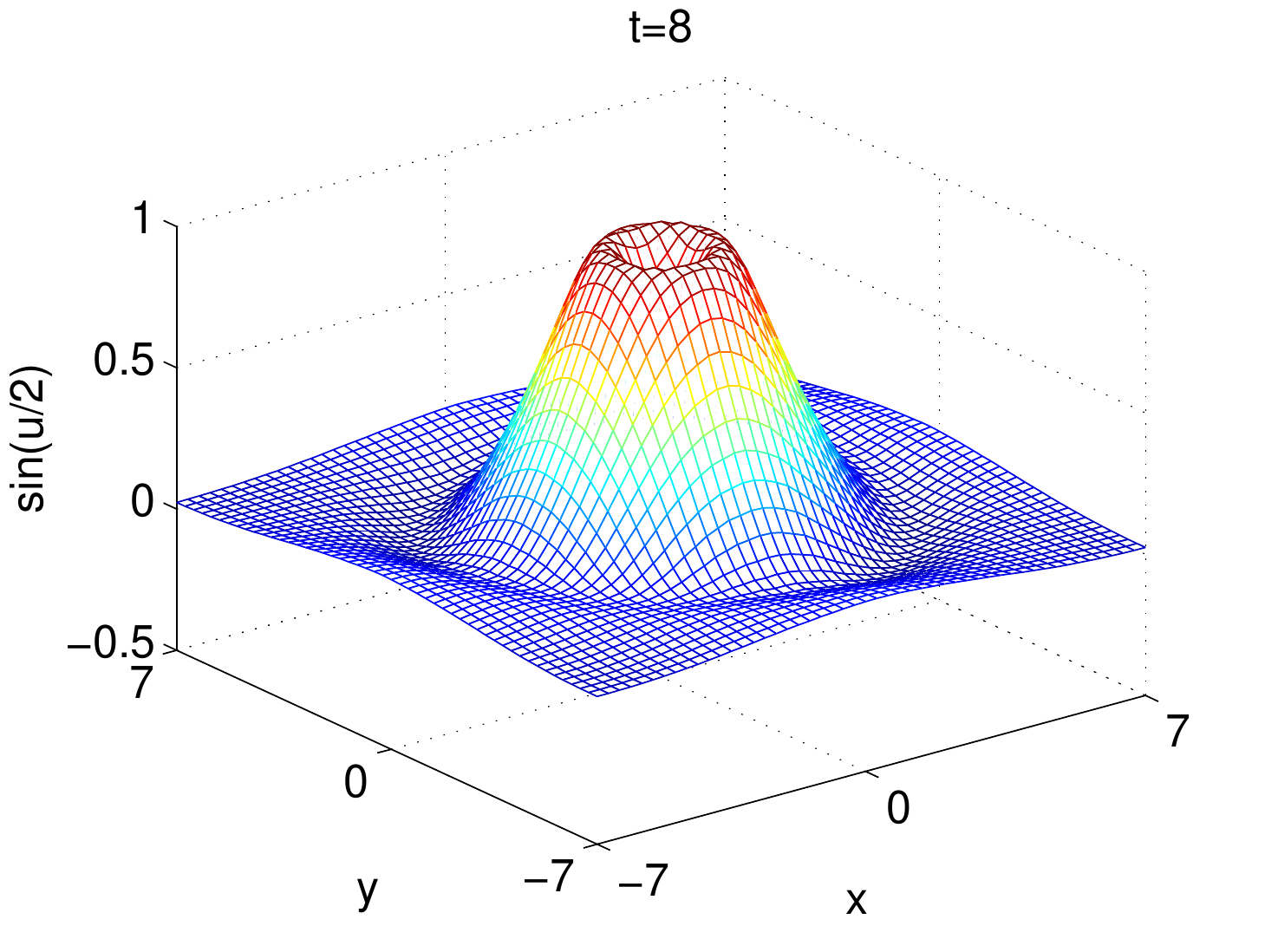}
\end{minipage}
\begin{minipage}[t]{60mm}
\includegraphics[width=60mm]{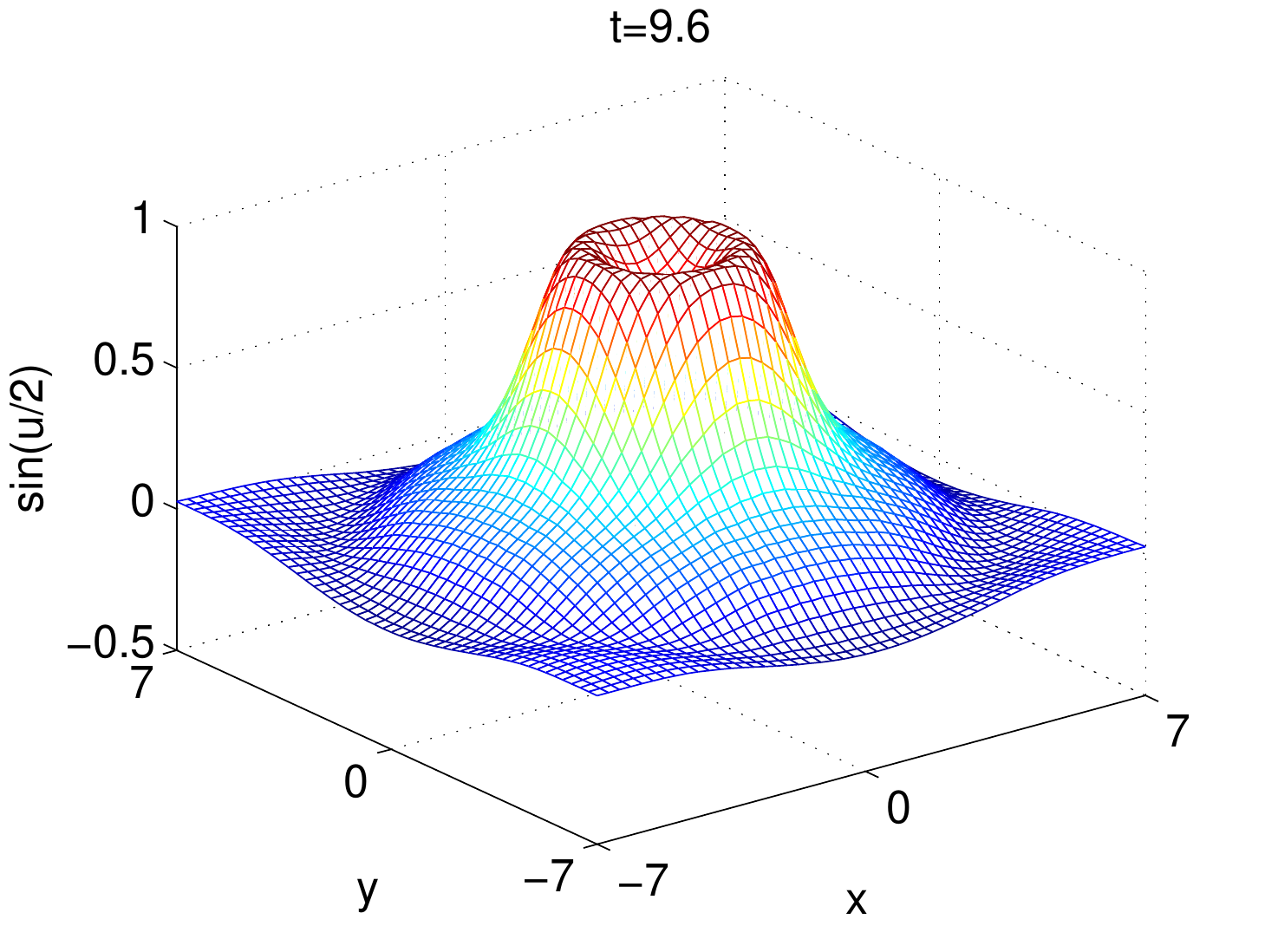}
\end{minipage}
\begin{minipage}[t]{60mm}
\includegraphics[width=60mm]{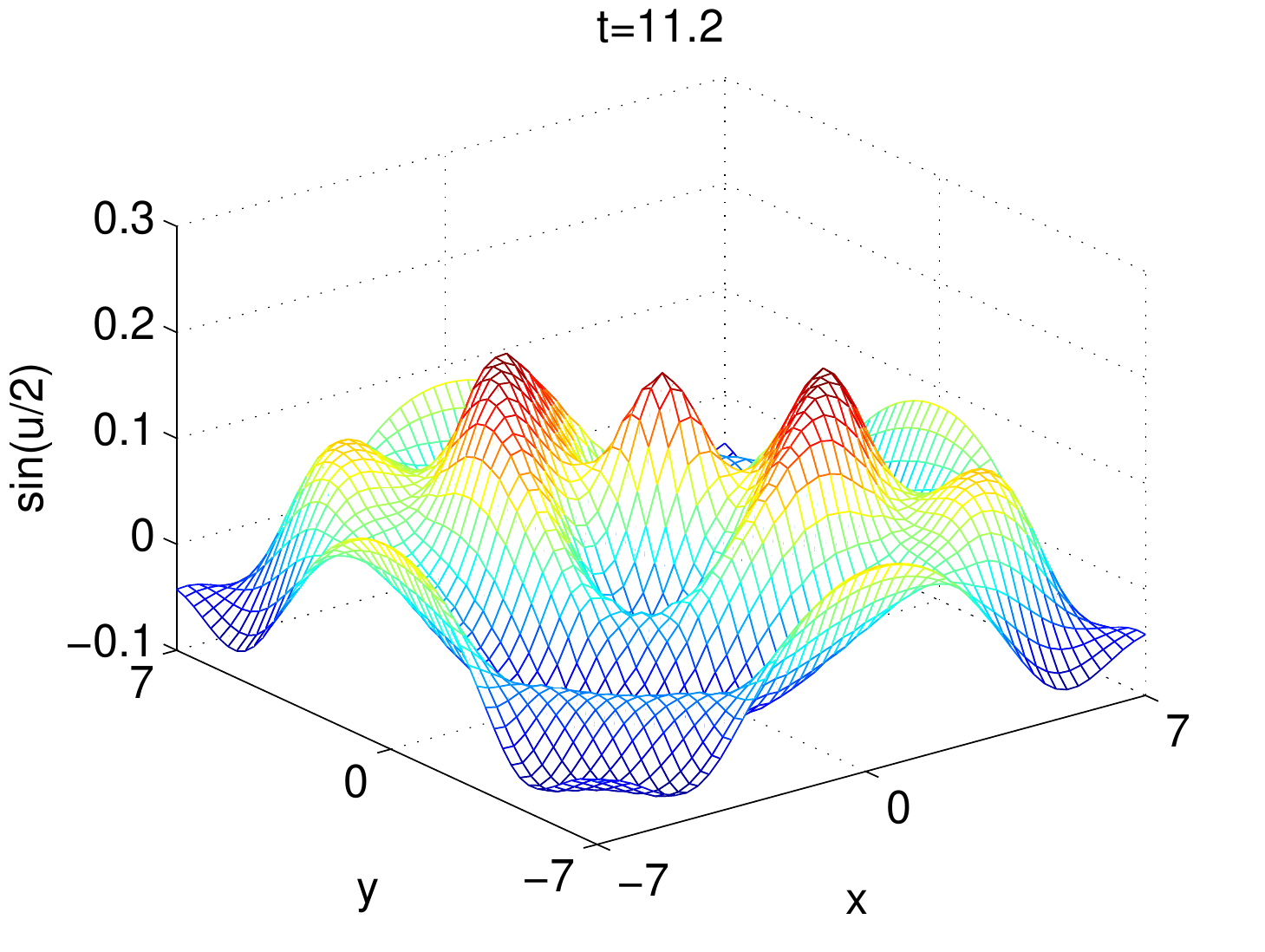}
\end{minipage}\begin{minipage}[t]{60mm}
\includegraphics[width=60mm]{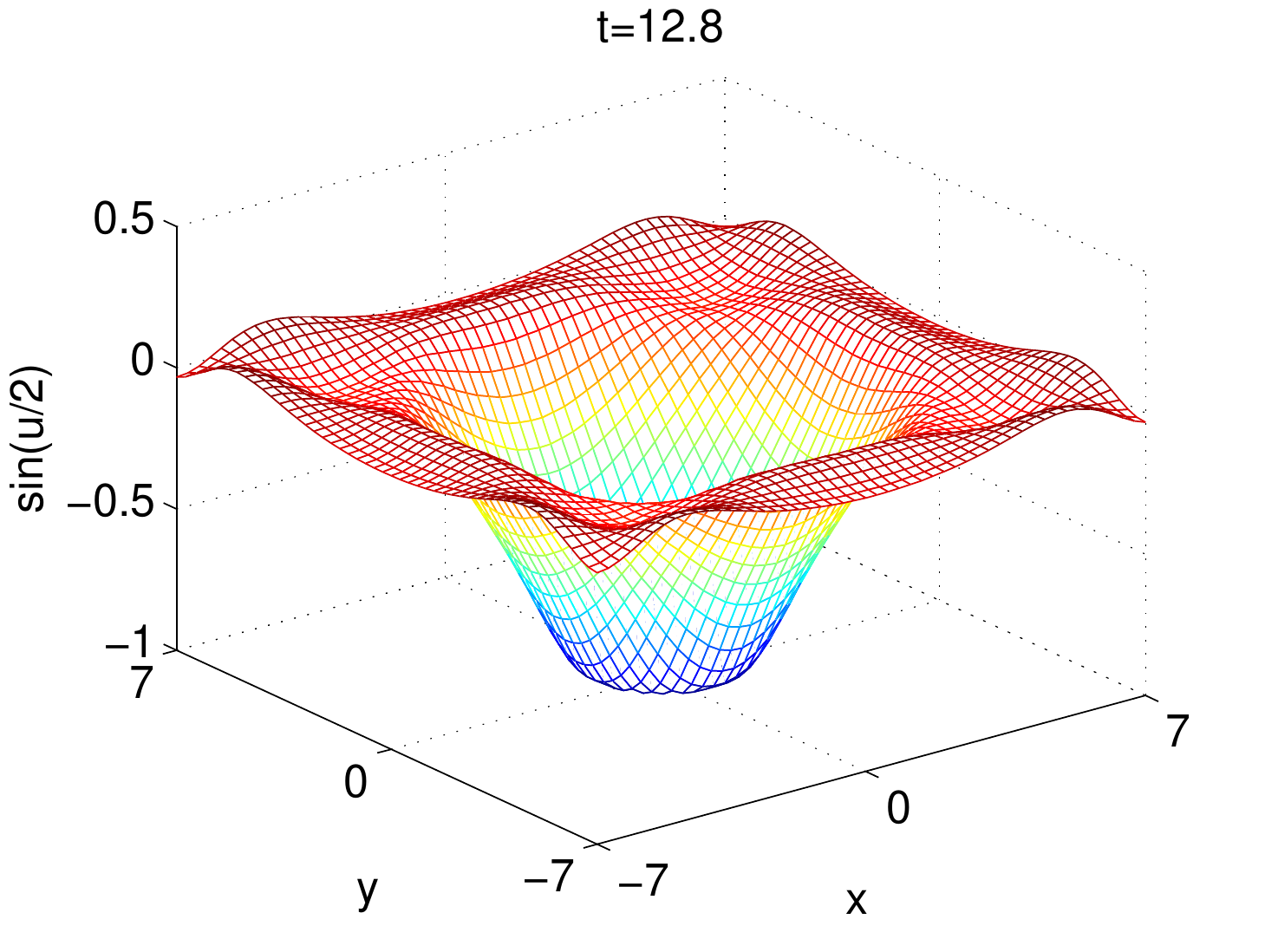}
\end{minipage}
\begin{minipage}[t]{60mm}
\includegraphics[width=60mm]{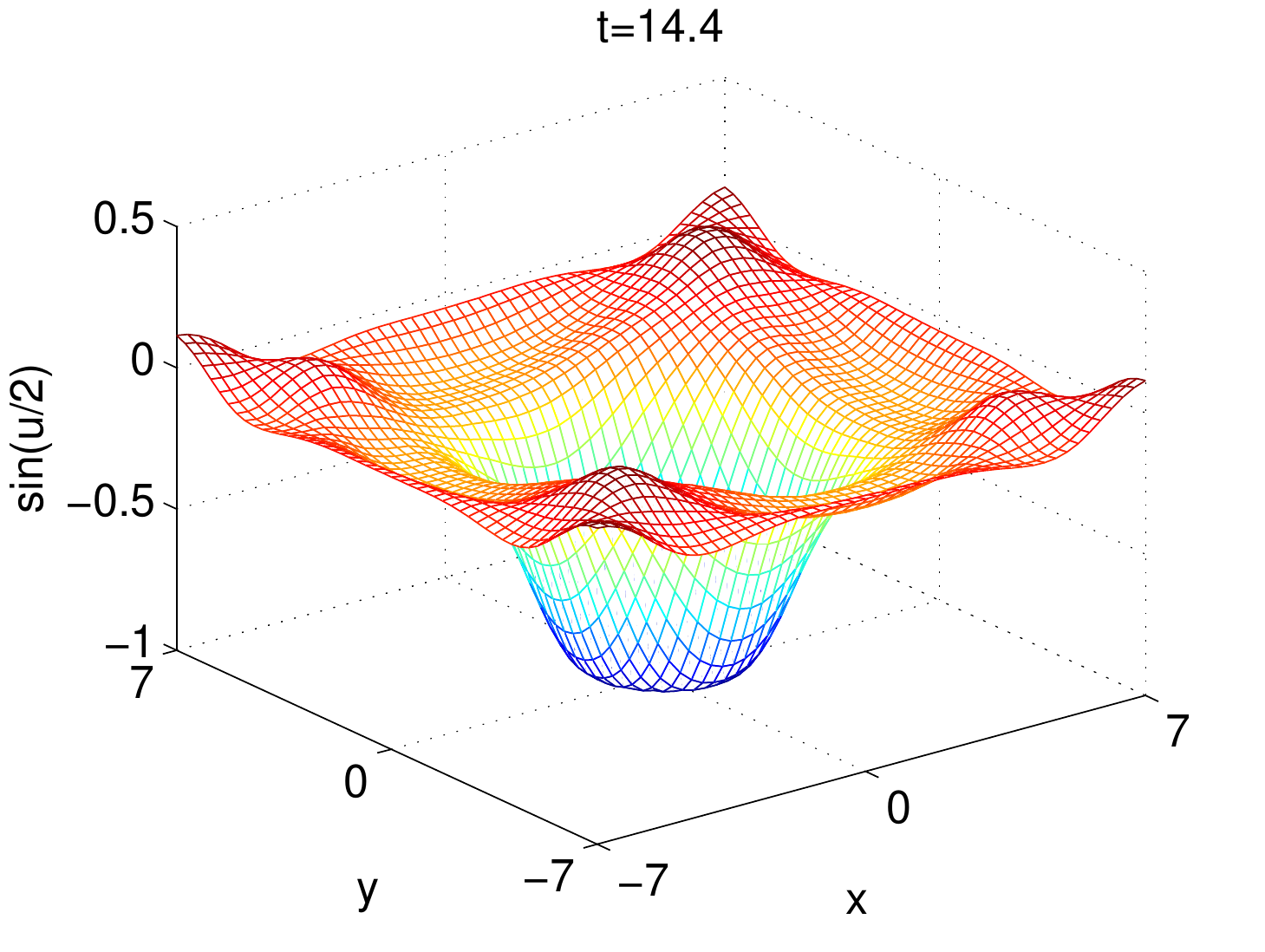}
\end{minipage}\begin{minipage}[t]{60mm}
\includegraphics[width=60mm]{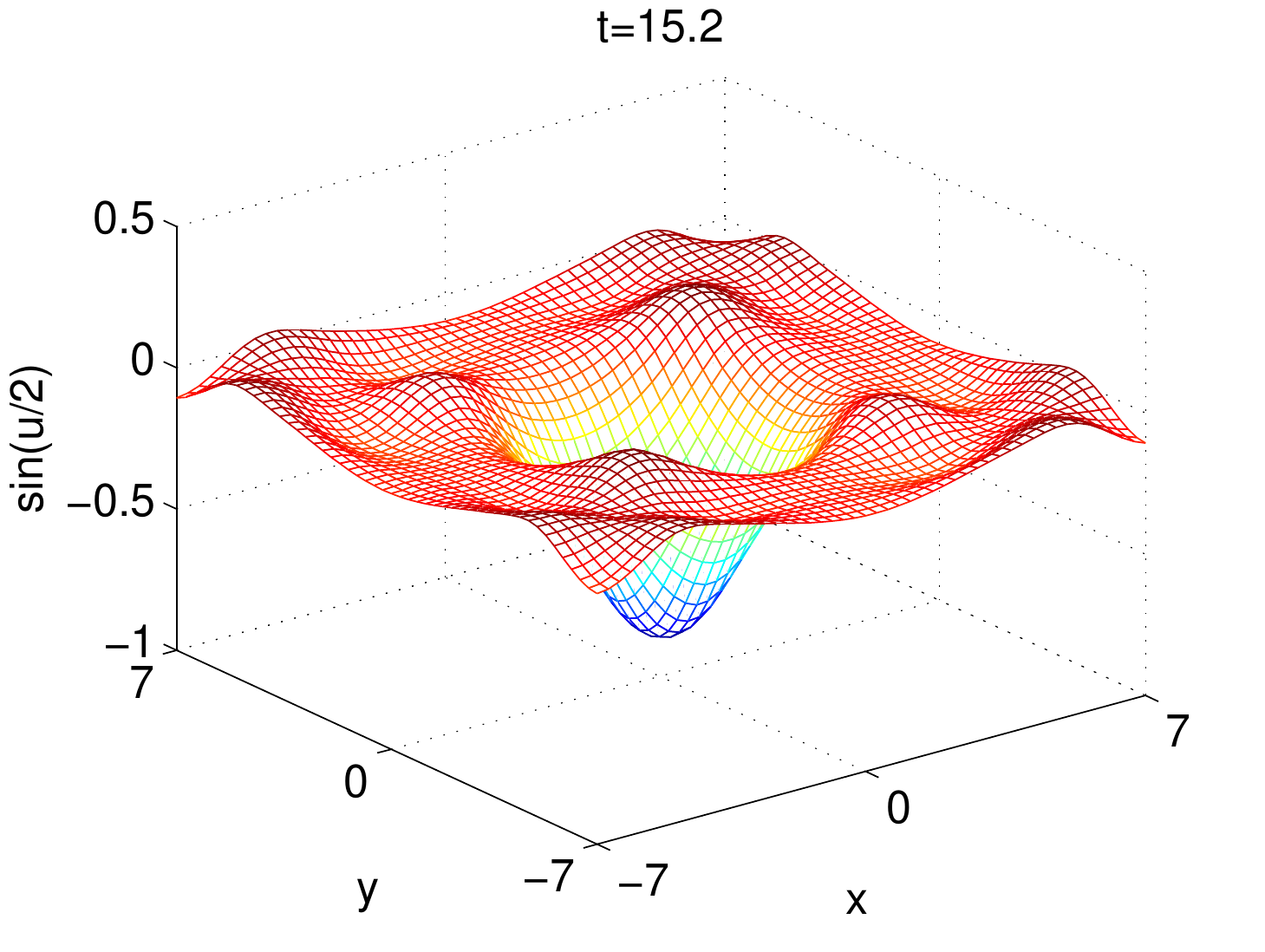}
\end{minipage}
\caption{Elliptical breather solitons: surface plots of initial condition and numerical solutions at times $t=0,1.6,8,9.6,11.2,12.8,14.4$ and $t=15.2$, respectively, in terms of $\sin(u/2)$. Spatial and temporal mesh sizes are taken as $h=0.14$ and $\tau=0.01$.}\label{EBS:1}
\end{figure}
\begin{figure}[H]
\centering\begin{minipage}[t]{70mm}
\includegraphics[width=70mm]{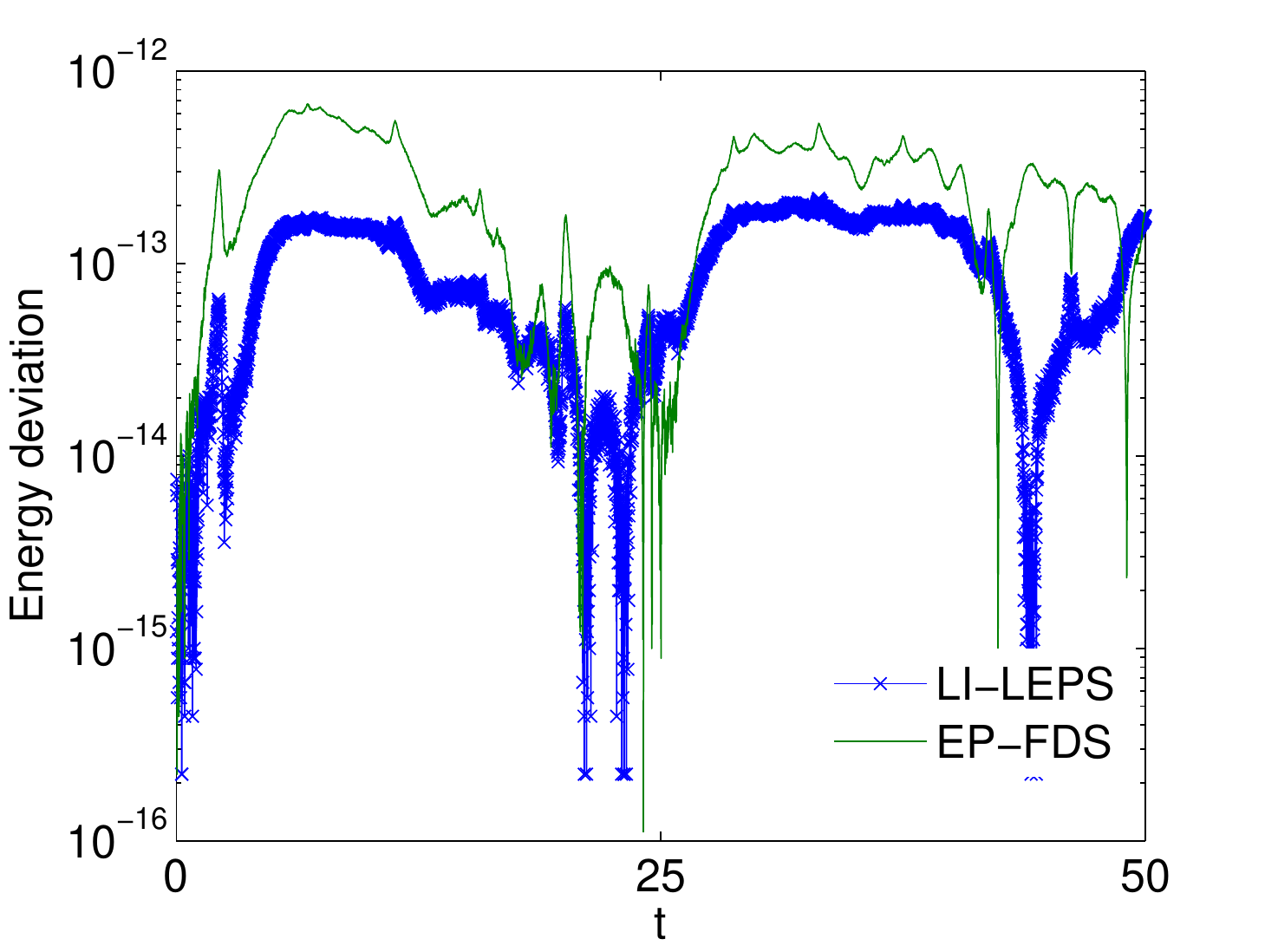}
\end{minipage}
\caption{Energy deviation over the time interval $t\in[0,50]$
with $h=0.14$ and $\tau=0.01$.}\label{2SG:err2}
\end{figure}

\subsubsection{Collisions of two circular solitons}

Subsequently, we consider the collisions of two expanding circular solitons by setting the initial conditions \cite{CL81,DPT95,KAS00,SKV10}
\begin{align*}
&f(x,y)=4\tan^{-1}\Bigg[\exp\Big(\frac{4-\sqrt{(x+3)^{2}+(y+7)^{2}}}{0.436}\Big)\Bigg],\ -30\le x\le 10,\ -21\le y\le 7,\\
& g(x,y)=4.13\text{sech}\Big(\frac{4-\sqrt{(x+3)^{2}+(y+7)^{2}}}{0.436}\Big), \ -30\le x\le 10,\ -21\le y\le 7,
\end{align*}
and the periodic boundary conditions.

Fig. \ref{2SG:fig3} shows the collision of two expanding circular solitons.
The solution shown includes the extension across $x =-10$ and $y =-7$ by symmetry properties of the problem \cite{DPT95}.
As illustrated in the figure, the collision between two expanding circular ring solitons in which two smaller ring
solitons bounding an annular region emerge into a large ring one. The movement of solitons can be
  more clearly observed in contour maps. The results are in good agreement with those published in Refs. \cite{CL81,DPT95,LiuW17,SKV10}. The results of energy
conservation are presented in Fig. \ref{2SG:err2}, which indicates that LI-LEPS shows a better conservation of energy than EP-FDS.

\begin{figure}[H]
\centering\begin{minipage}[t]{60mm}
\includegraphics[width=60mm]{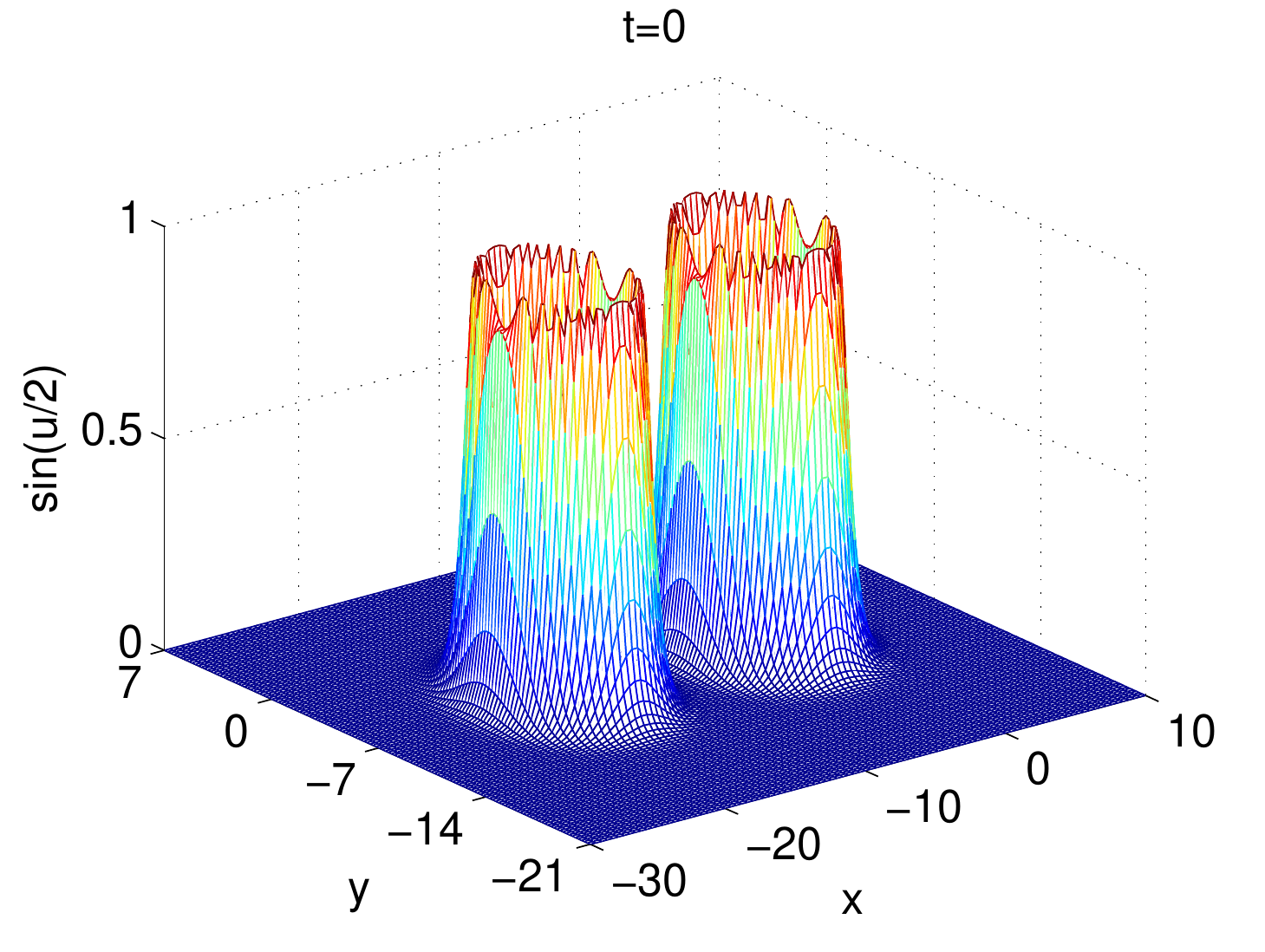}
\end{minipage}
\begin{minipage}[t]{60mm}
\includegraphics[width=60mm]{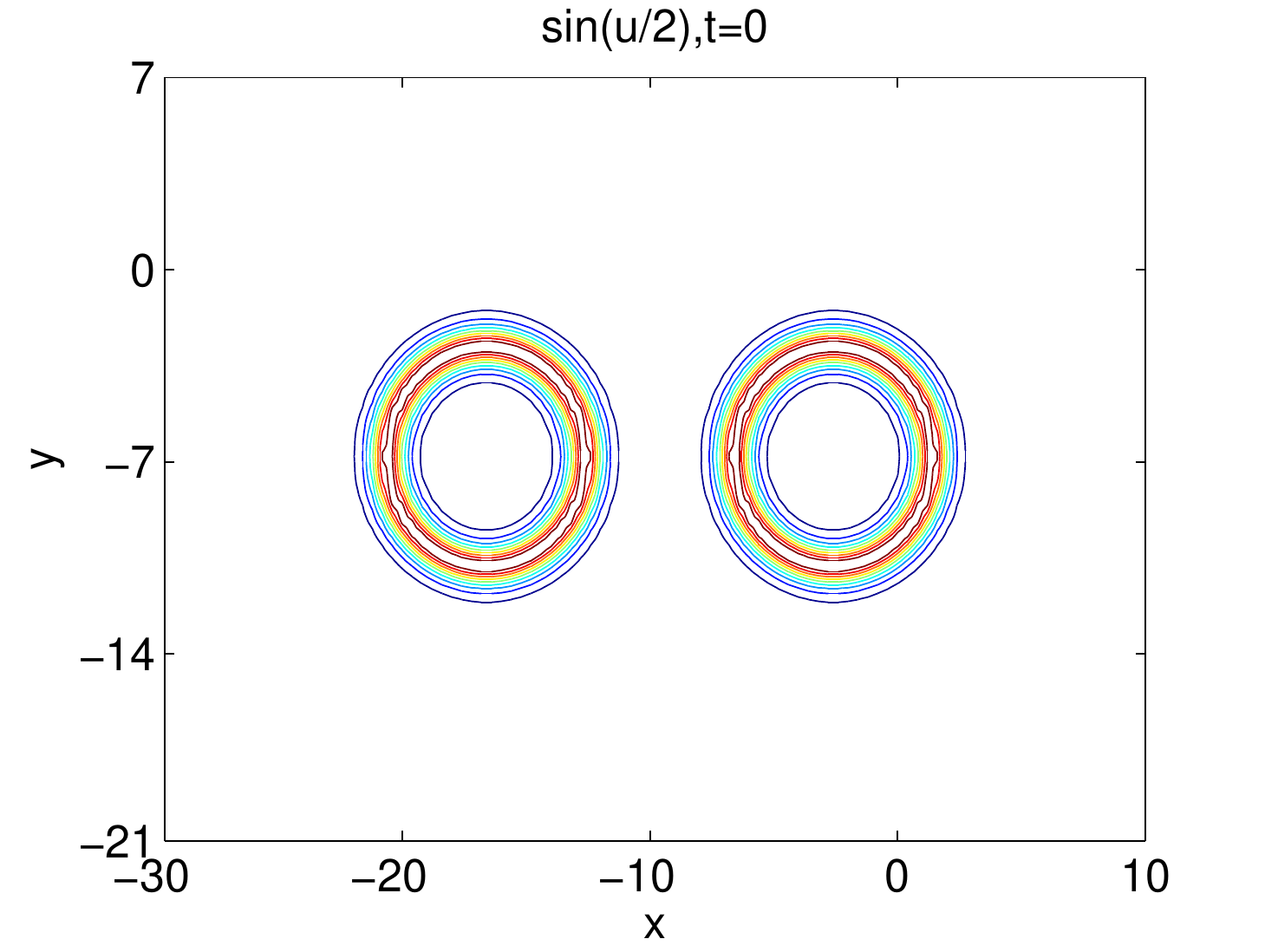}
\end{minipage}
\begin{minipage}[t]{60mm}
\includegraphics[width=60mm]{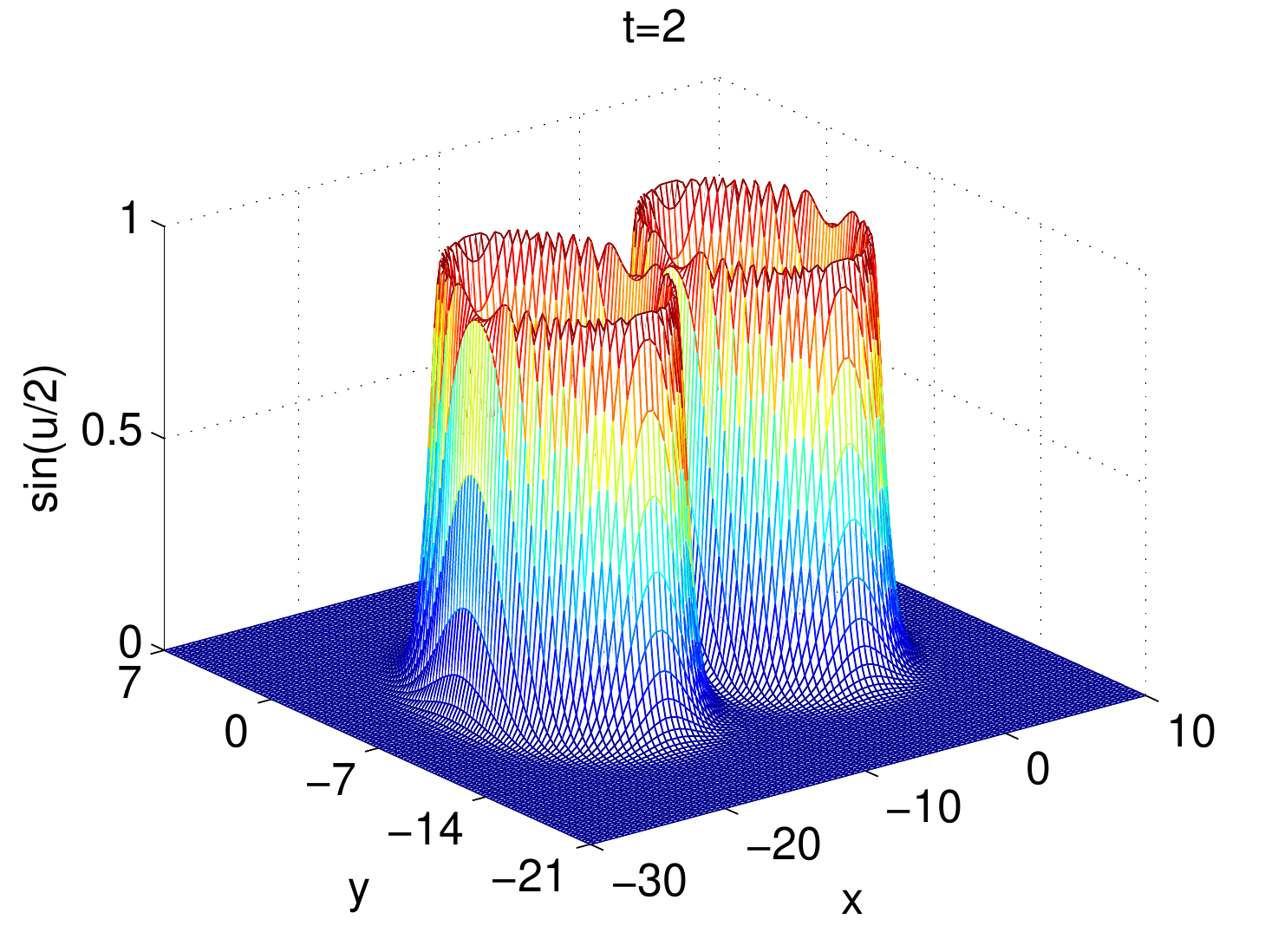}
\end{minipage}
\begin{minipage}[t]{60mm}
\includegraphics[width=60mm]{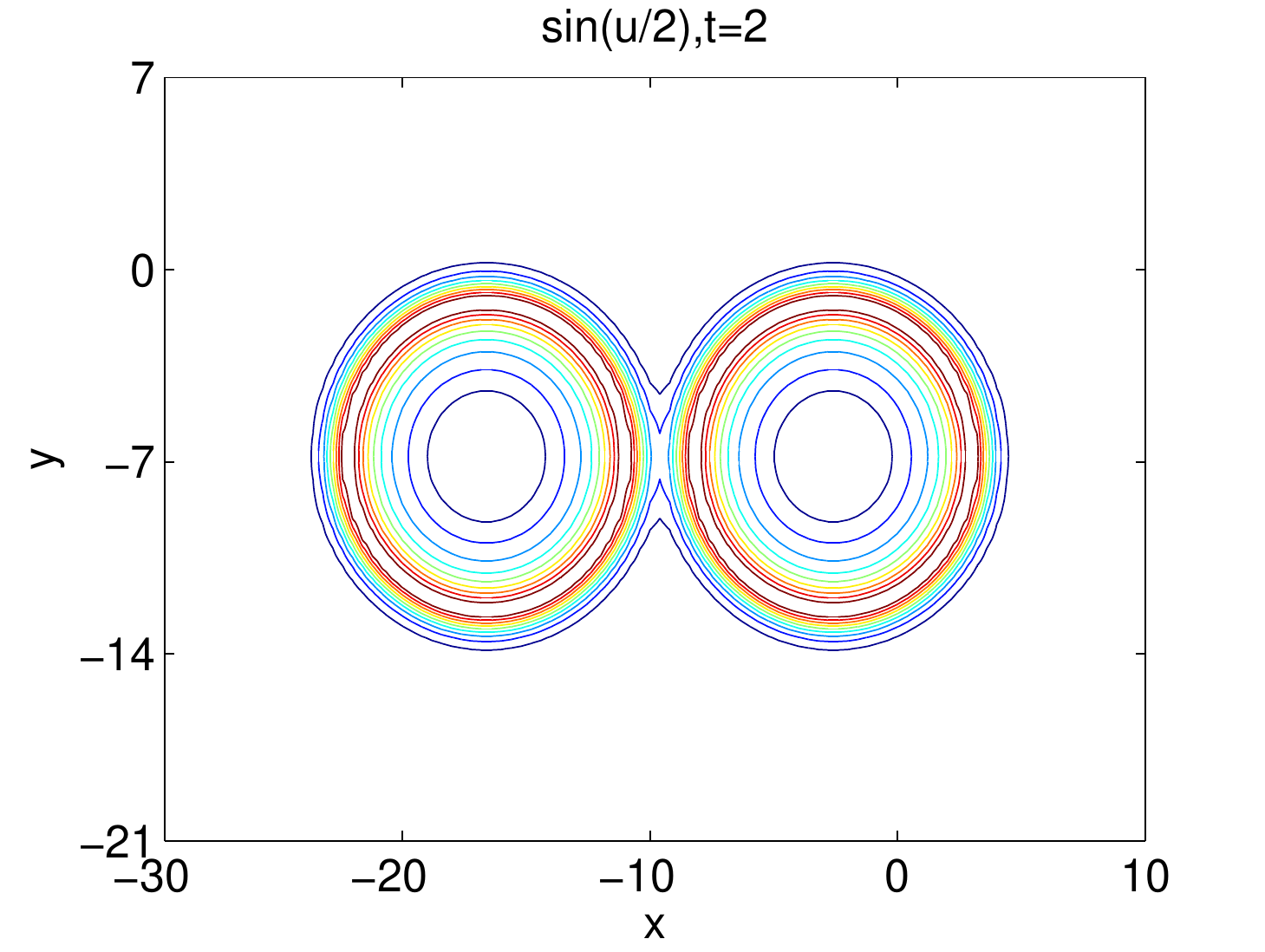}
\end{minipage}
\begin{minipage}[t]{60mm}
\includegraphics[width=60mm]{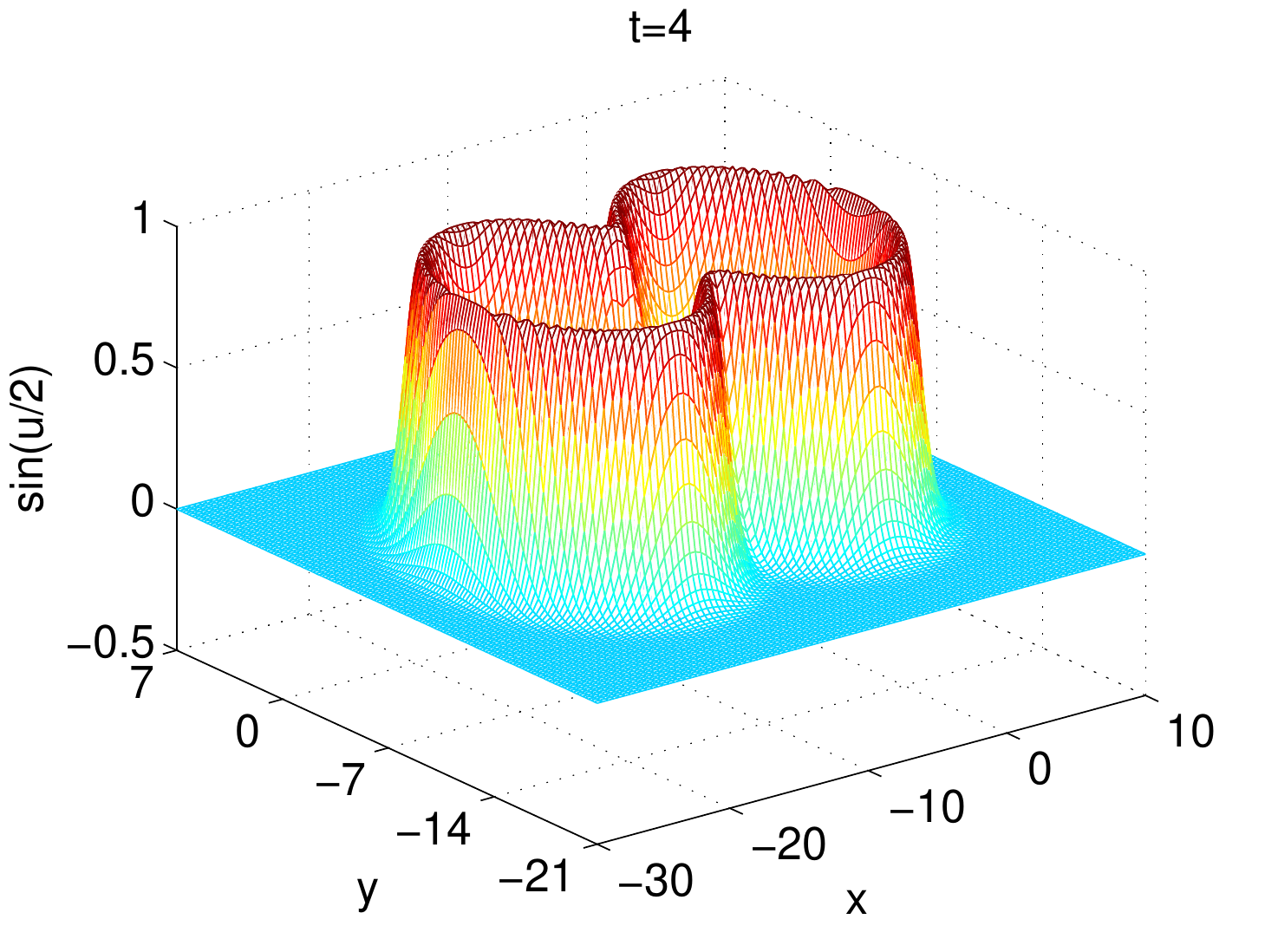}
\end{minipage}
\begin{minipage}[t]{60mm}
\includegraphics[width=60mm]{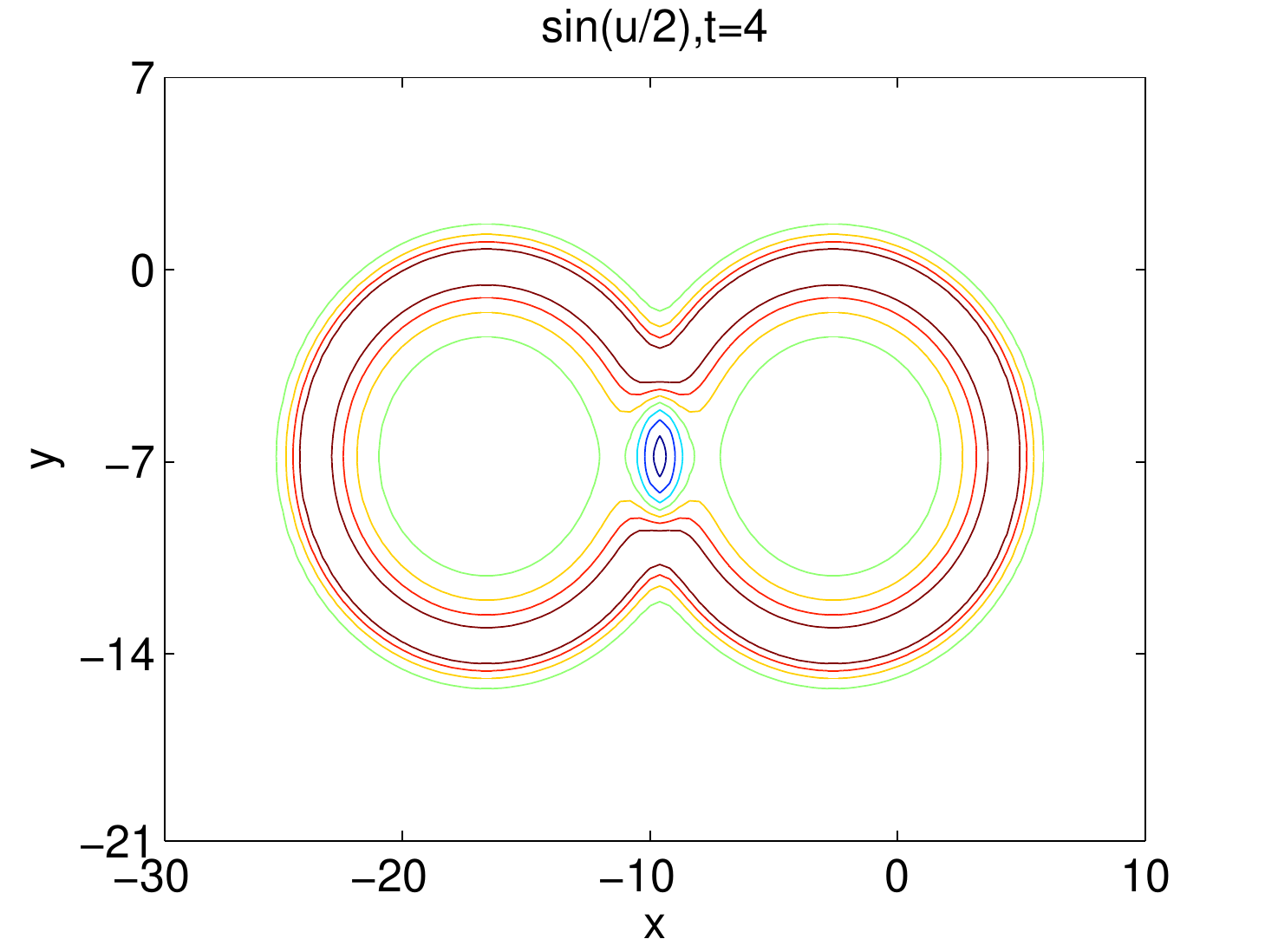}
\end{minipage}
\begin{minipage}[t]{60mm}
\includegraphics[width=60mm]{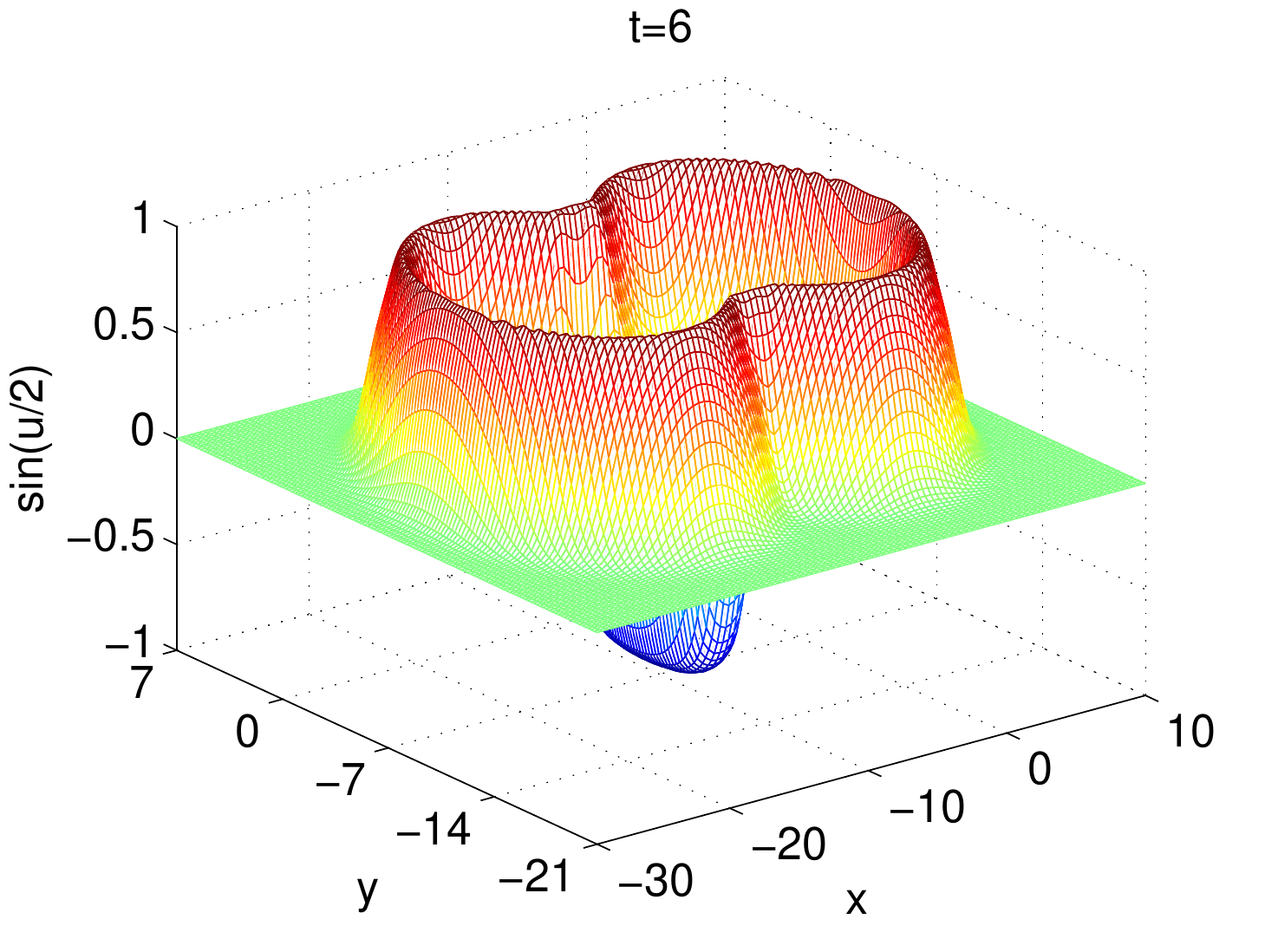}
\end{minipage}
\begin{minipage}[t]{60mm}
\includegraphics[width=60mm]{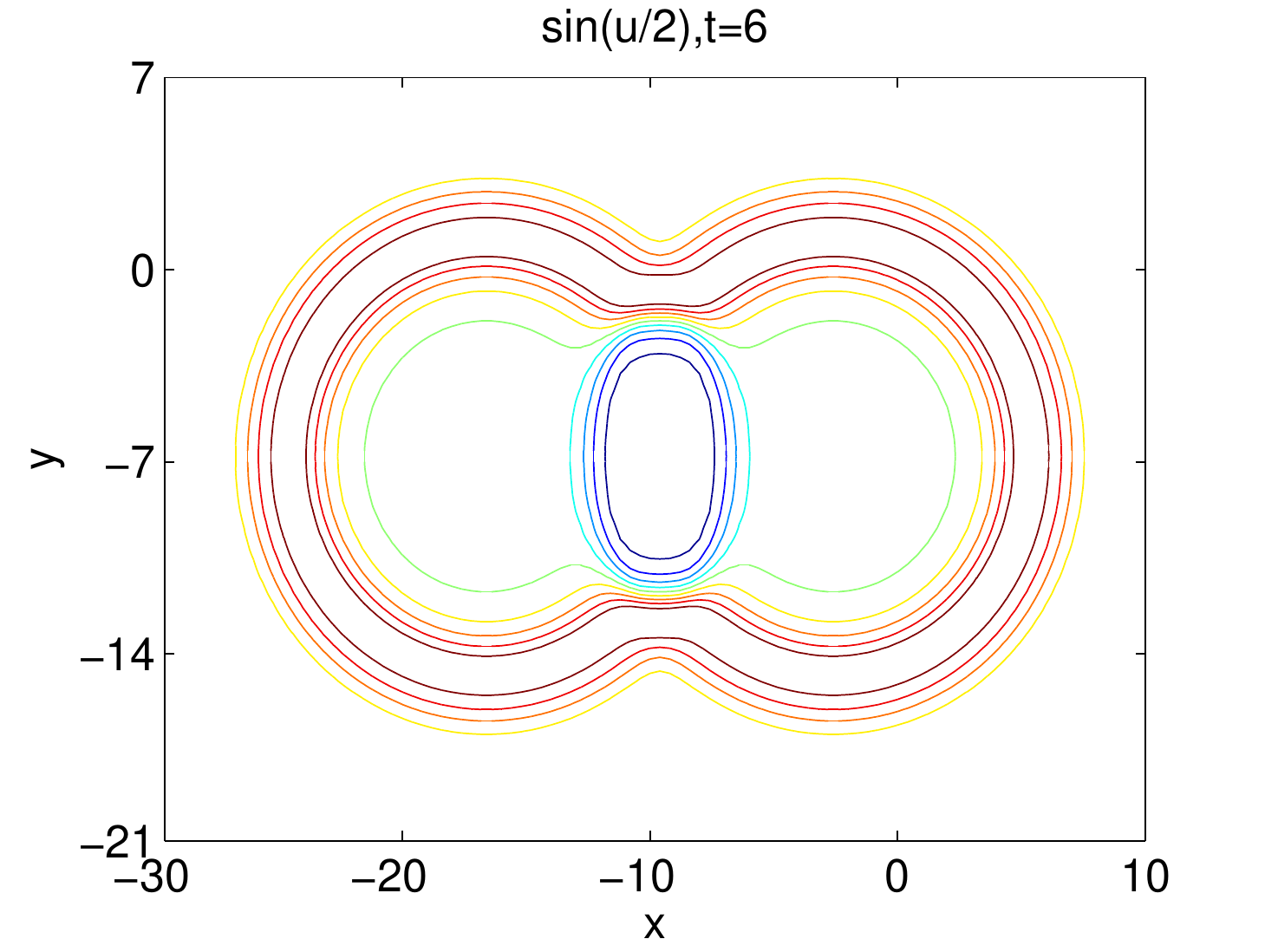}
\end{minipage}
\begin{minipage}[t]{60mm}
\includegraphics[width=60mm]{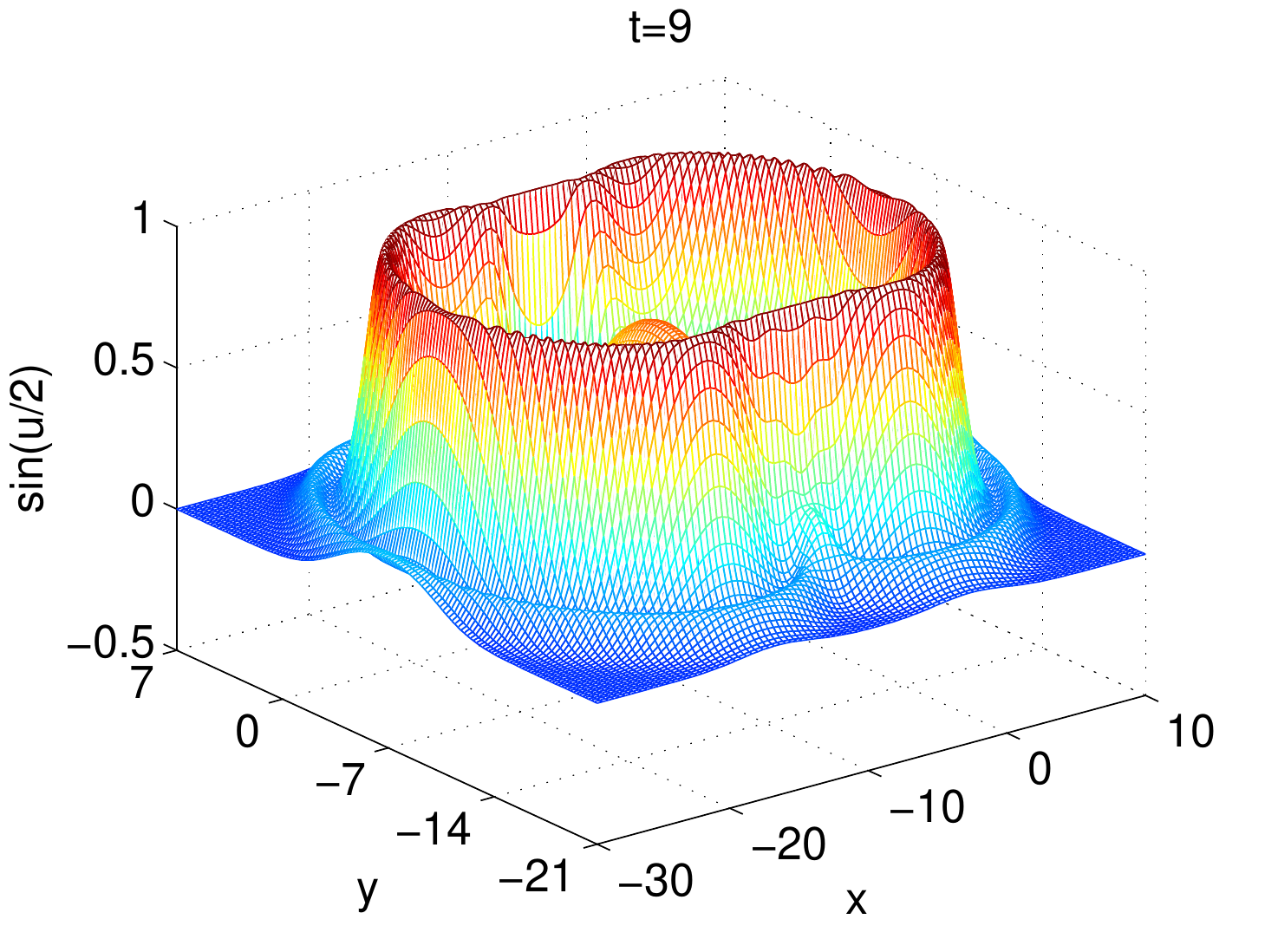}
\end{minipage}
\begin{minipage}[t]{60mm}
\includegraphics[width=60mm]{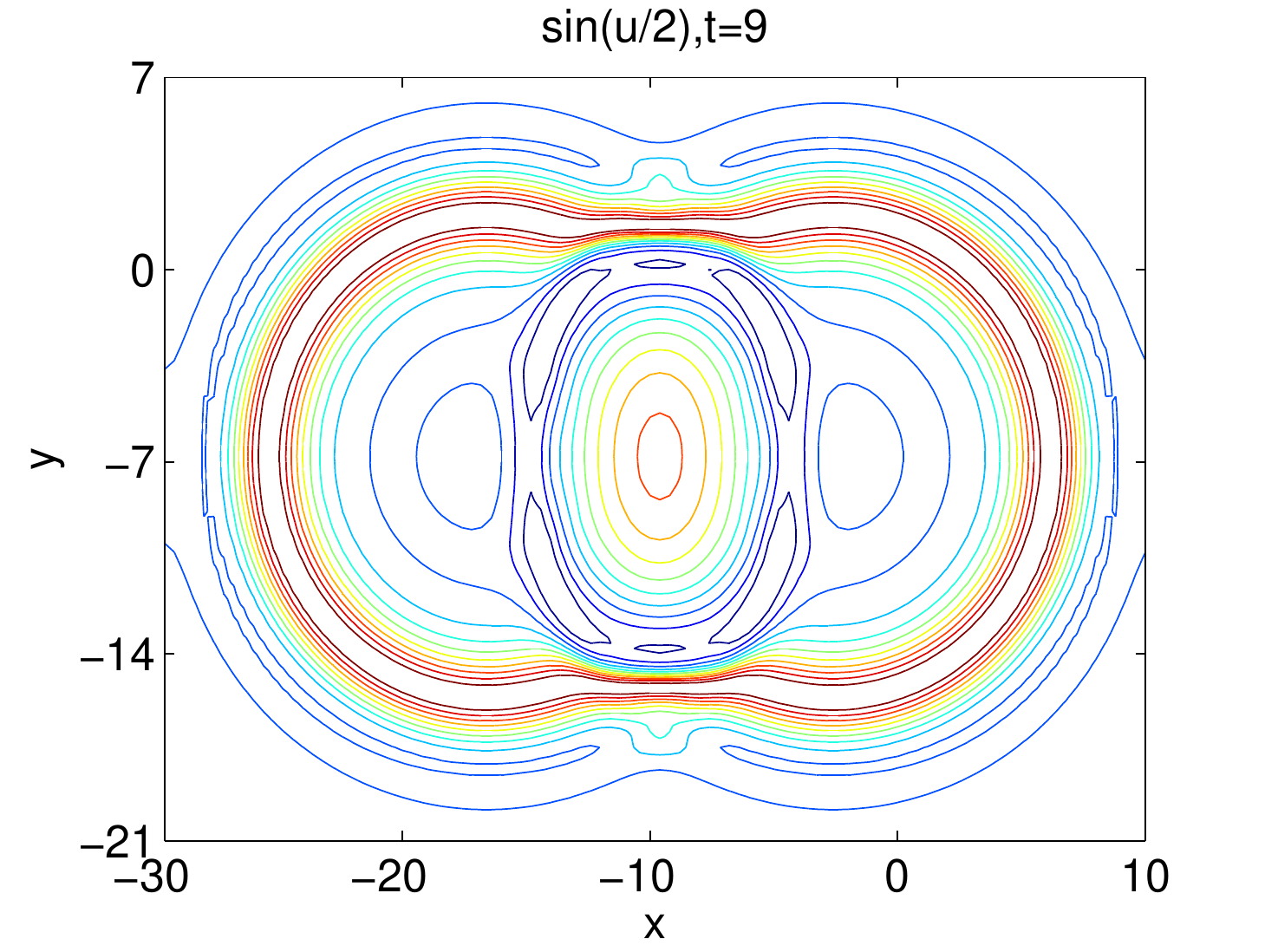}
\end{minipage}
\caption{Collision of two expanding ring solitons: surface and contours plots of initial condition and numerical solutions at times $t=0,2,4,6$ and $t=9$, respectively, in terms of $\sin(u/2)$. Spatial and temporal mesh sizes are taken as $h=0.2$ and $\tau=0.01$.}\label{2SG:fig3}
\end{figure}
\begin{figure}[H]
\centering\begin{minipage}[t]{70mm}
\includegraphics[width=70mm]{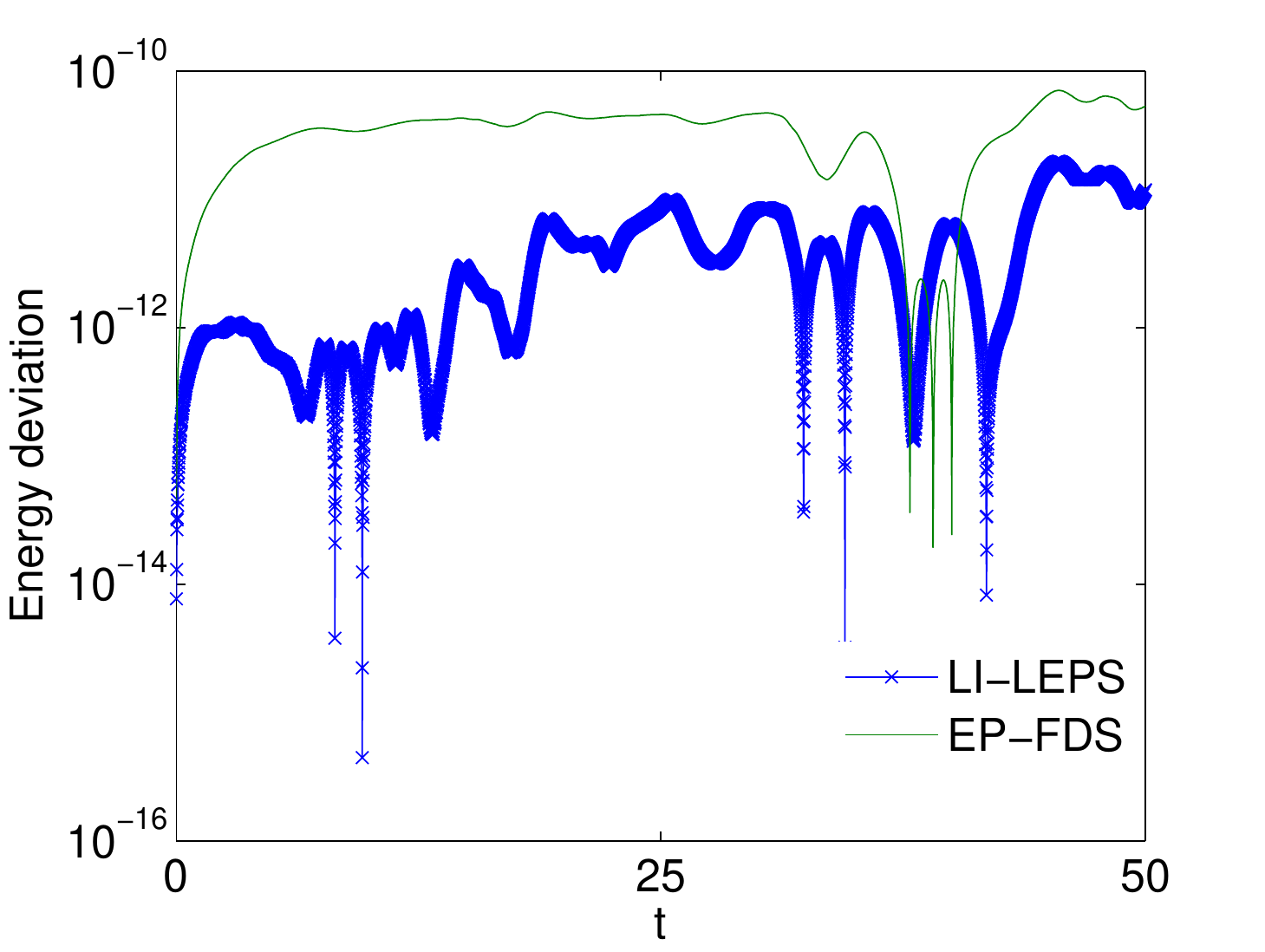}
\end{minipage}
\caption{Energy deviation over the time interval $t\in[0,50]$
with $h=0.2$ and $\tau=0.01$.}\label{2SG:err2}
\end{figure}
\subsubsection{Collisions of four circular solitons}
Finally, the collision of four expanding circular ring solitons is reported by taking initial conditions \cite{Argyris91,KAS00,SKV10}
\begin{align*}
&f(x,y)=4\tan^{-1}\left[\exp\left(\frac{4-\sqrt{(x+3)^{2}+(y+7)^{2}}}{0.436}\right)\right],\ -30\le x,y\le 10,\\
& g(x,y)=\frac{4.13}{\cosh\left(\frac{4-\sqrt{(x+3)^{2}+(y+7)^{2}}}{0.436}\right)},\ -30\le x,y\le 10,
\end{align*}
and the periodic boundary conditions.

The solution shown includes the extension across $x =-10$ and $y =-10$ by symmetry properties of the problem \cite{Argyris91,DPT95,KAS00,SKV10}.
 Fig. \ref{2SG:fig4} shows the collision precisely among four expanding circular ring solitons in which the smaller ring solitons bounding an
annular region emerge into a large ring one. This matches known experimental results \cite{Argyris91,DPT95}. Also,
contour maps are shown to illustrate more clearly the movement of the solitons.
The long time energy deviations of the two schemes are shown in Fig. \ref{2SG:err3}. As is clear, LI-LEPS shows the remarkable advantage in energy preservation over EP-FDS again.

\begin{figure}[H]
\centering\begin{minipage}[t]{60mm}
\includegraphics[width=60mm]{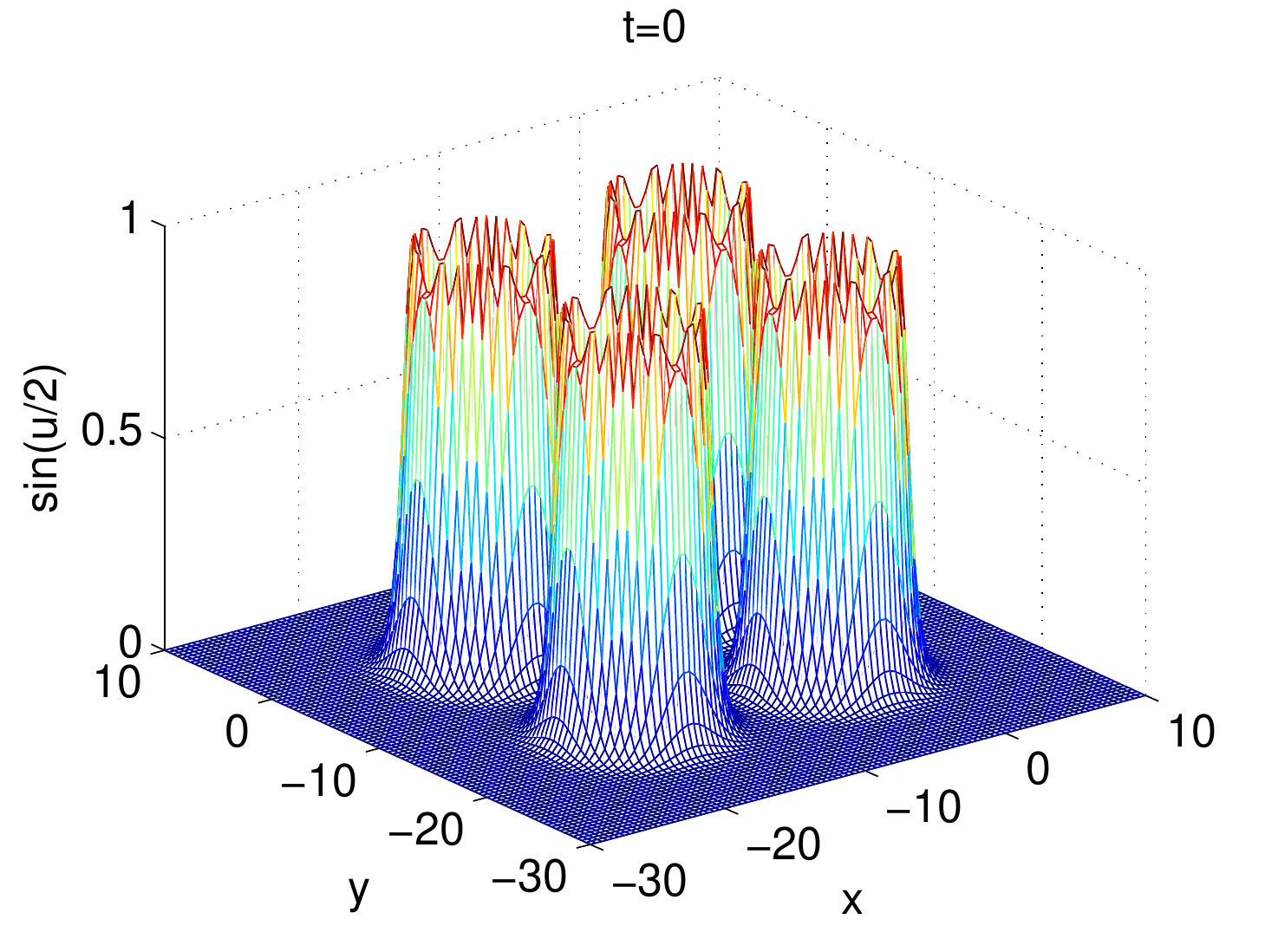}
\end{minipage}
\begin{minipage}[t]{60mm}
\includegraphics[width=60mm]{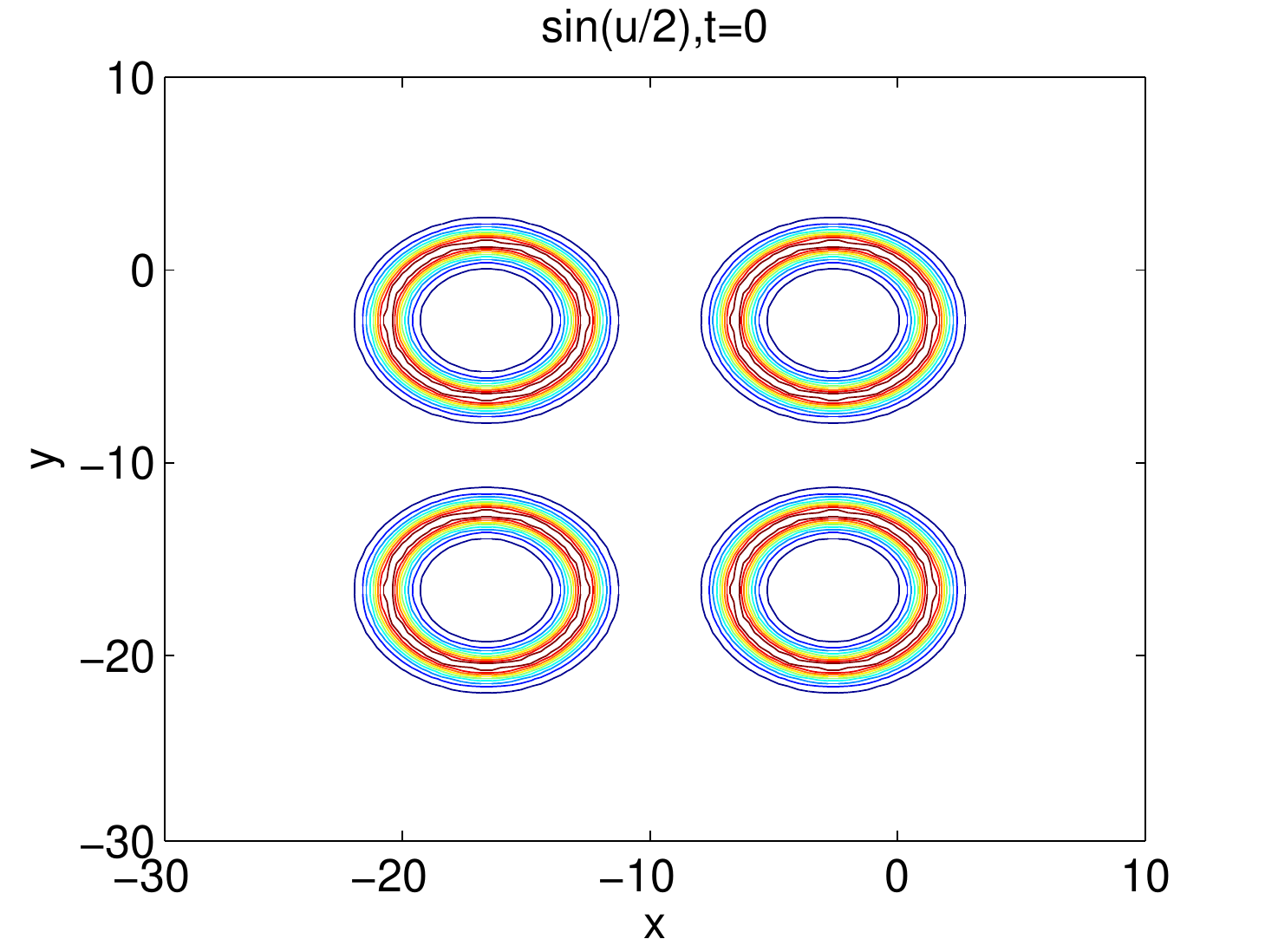}
\end{minipage}
\begin{minipage}[t]{60mm}
\includegraphics[width=60mm]{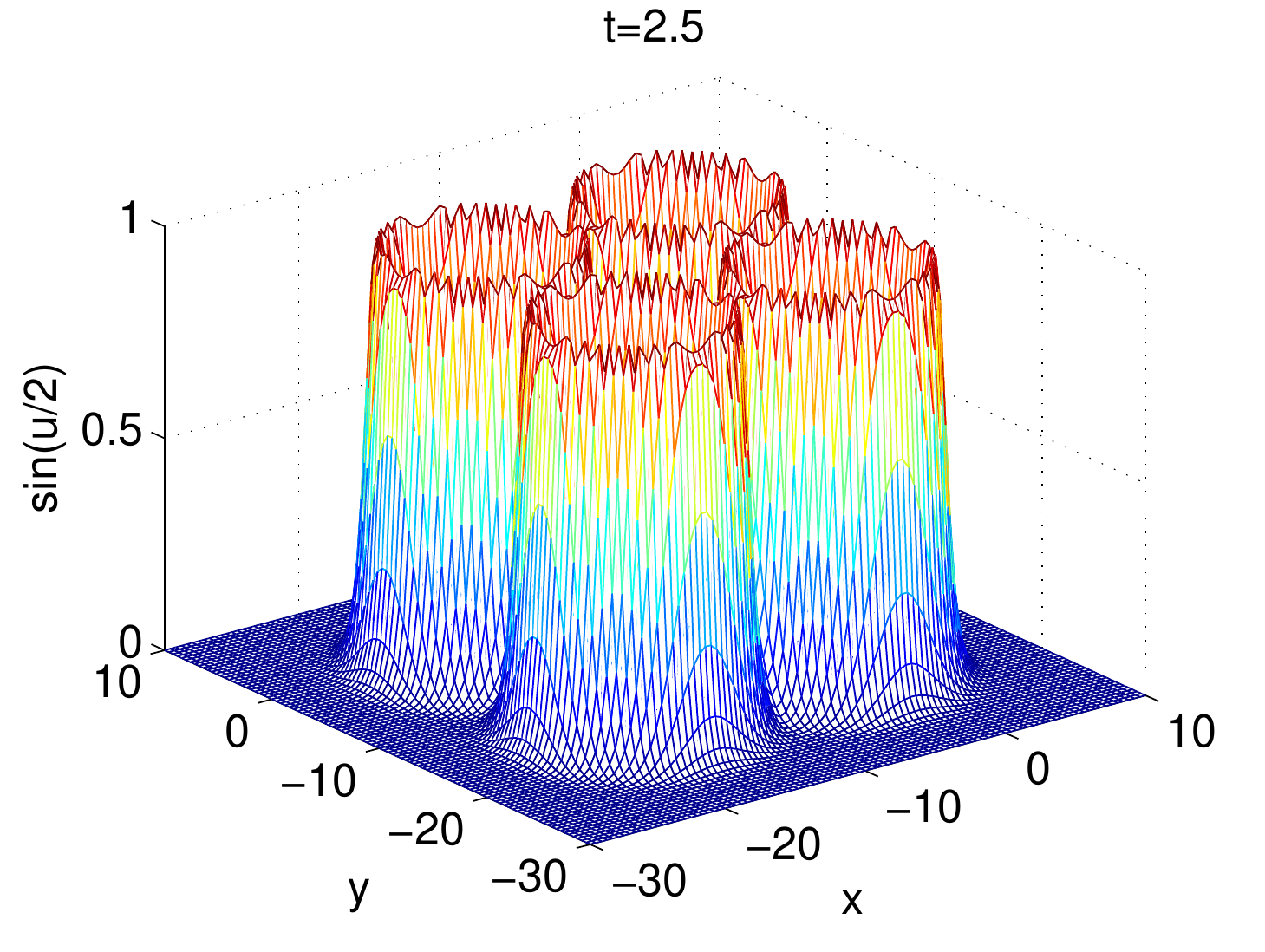}
\end{minipage}
\begin{minipage}[t]{60mm}
\includegraphics[width=60mm]{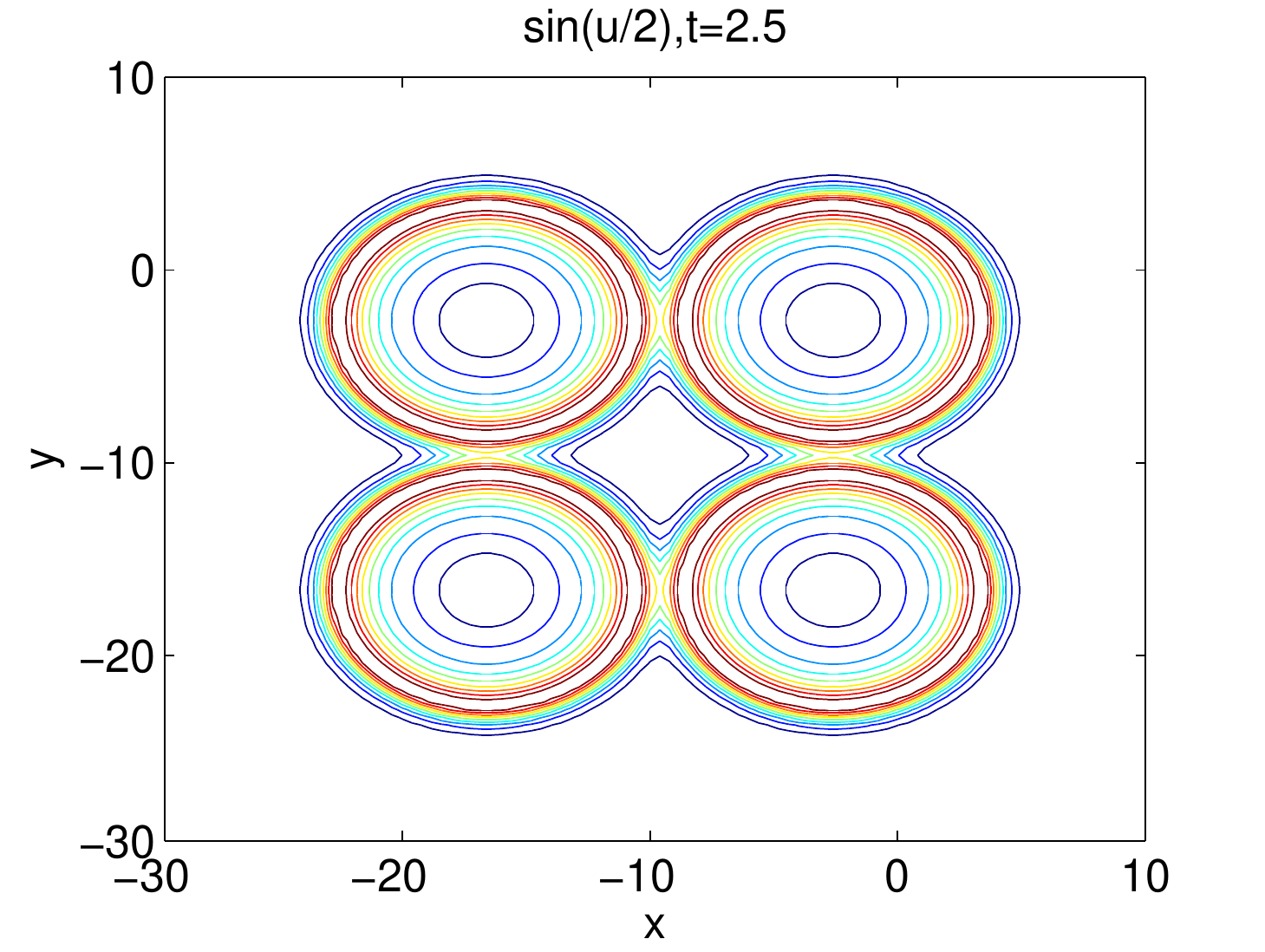}
\end{minipage}
\begin{minipage}[t]{60mm}
\includegraphics[width=60mm]{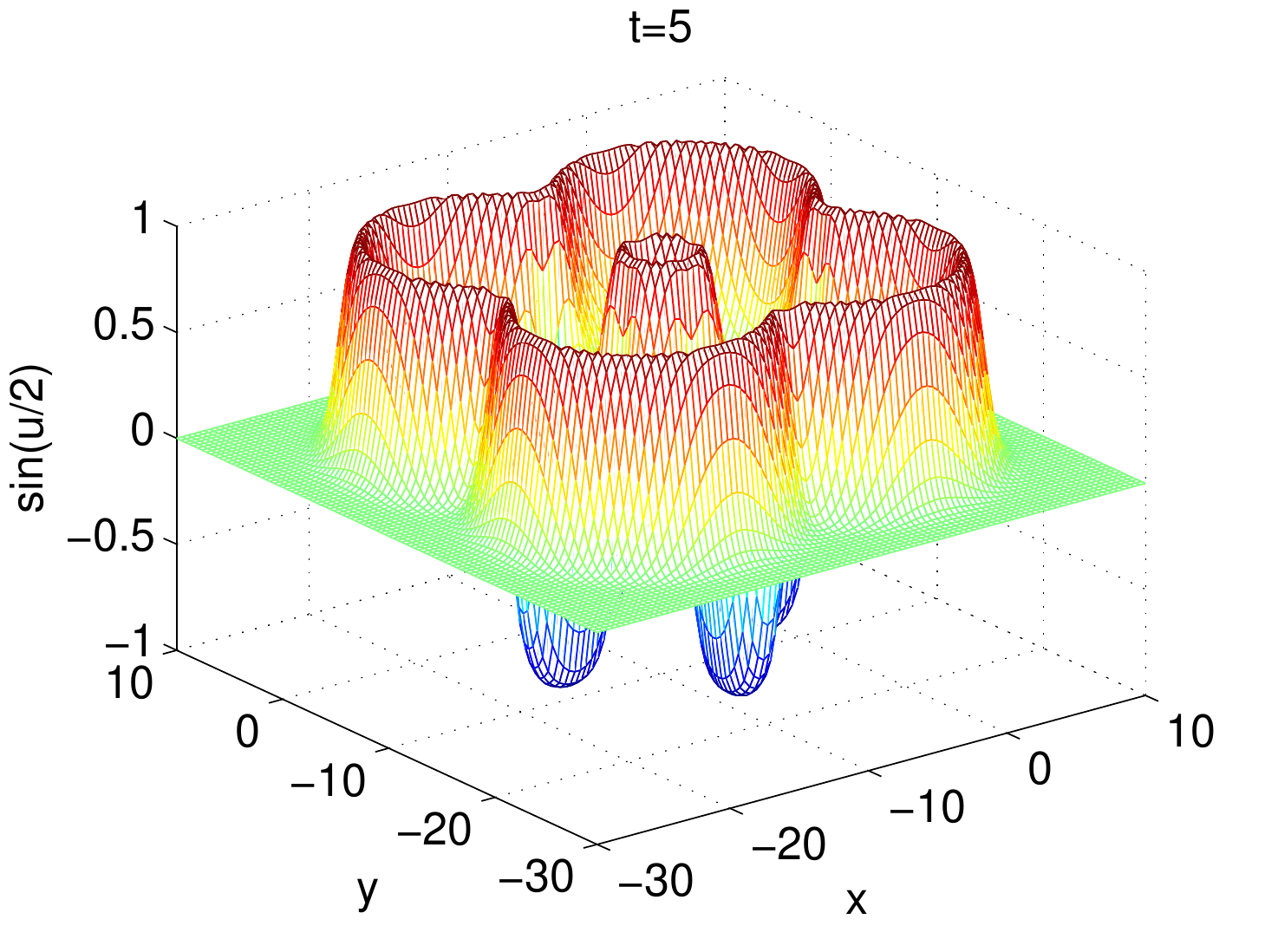}
\end{minipage}
\begin{minipage}[t]{60mm}
\includegraphics[width=60mm]{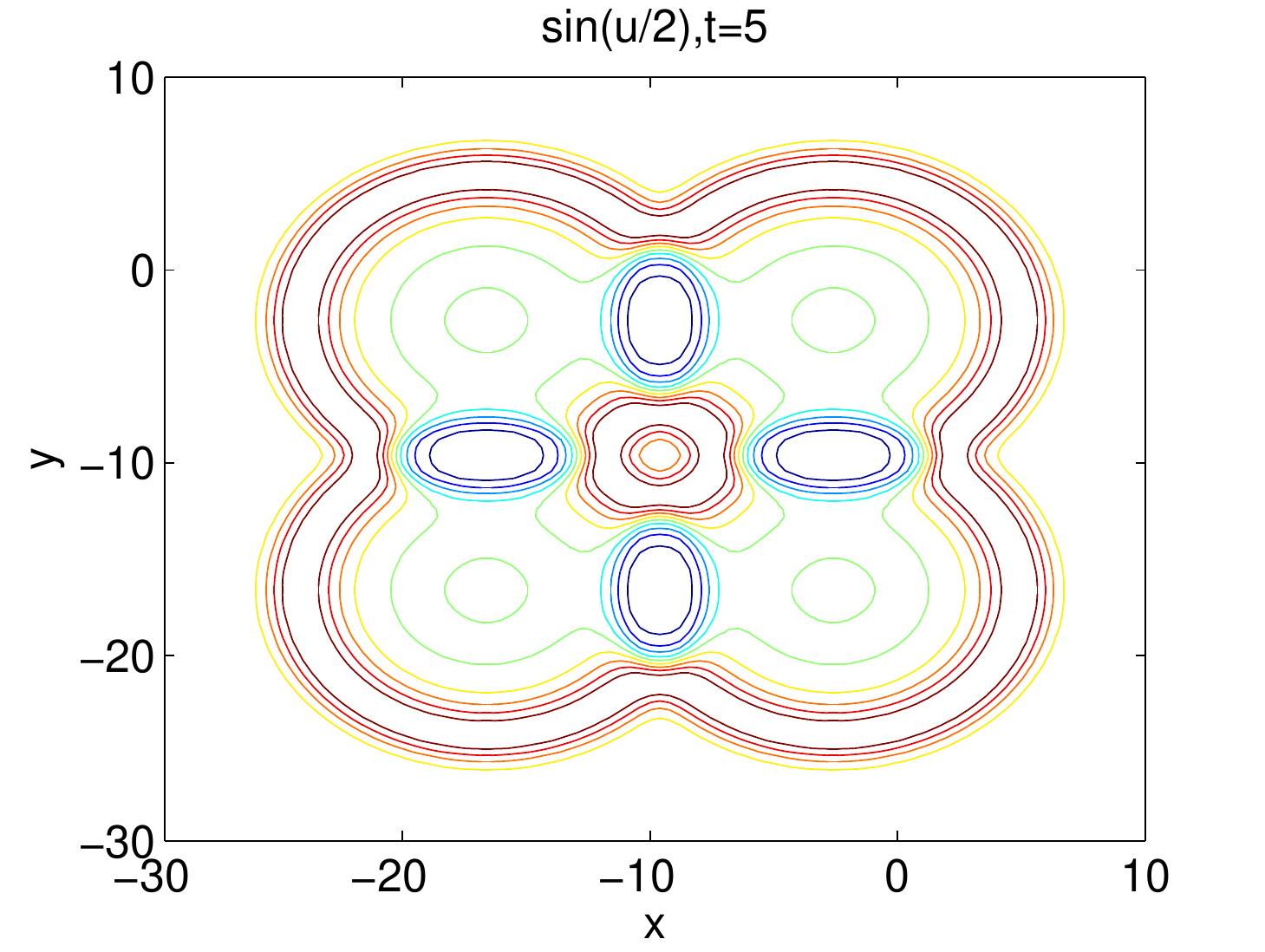}
\end{minipage}
\begin{minipage}[t]{60mm}
\includegraphics[width=60mm]{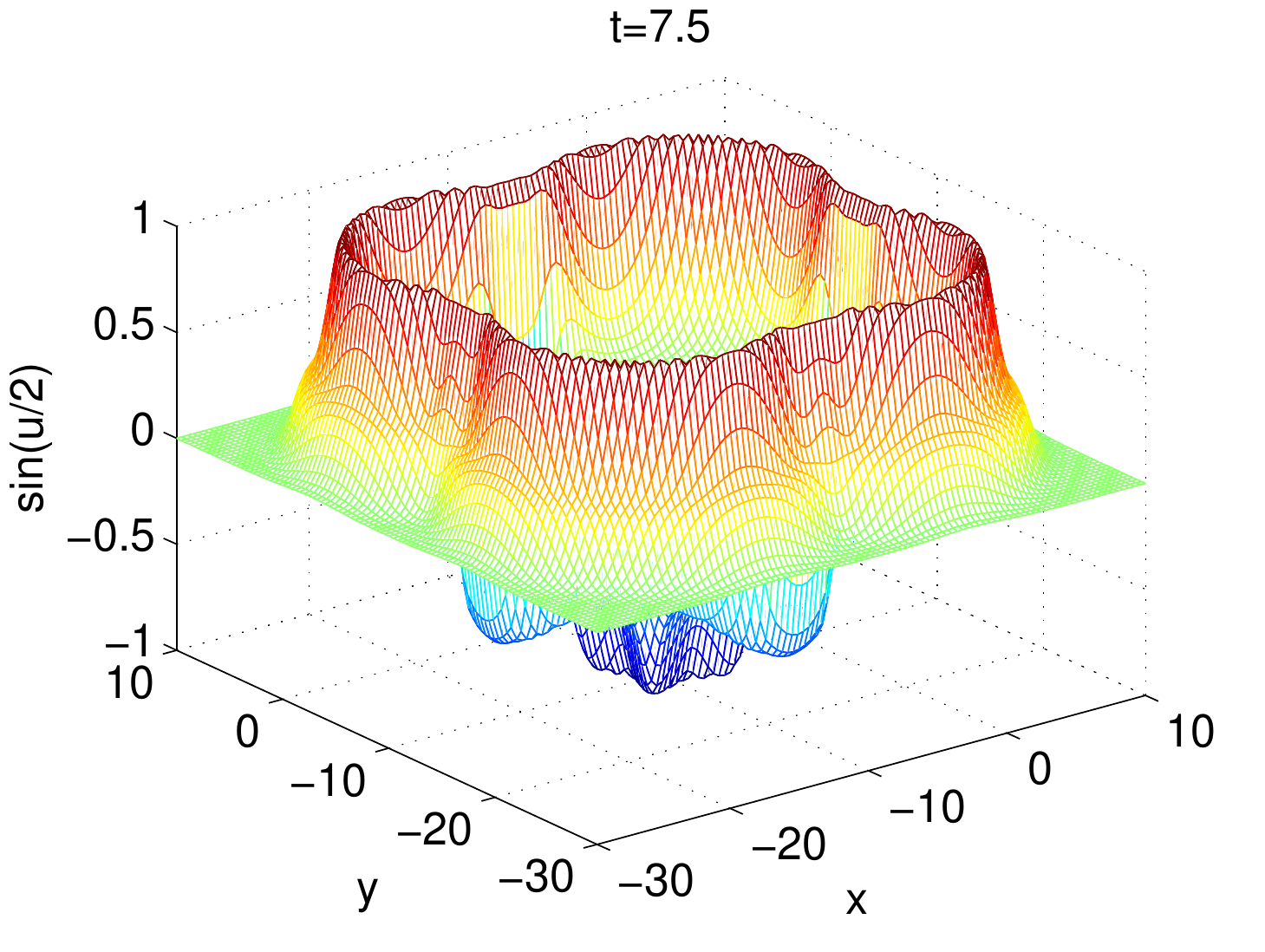}
\end{minipage}
\begin{minipage}[t]{60mm}
\includegraphics[width=60mm]{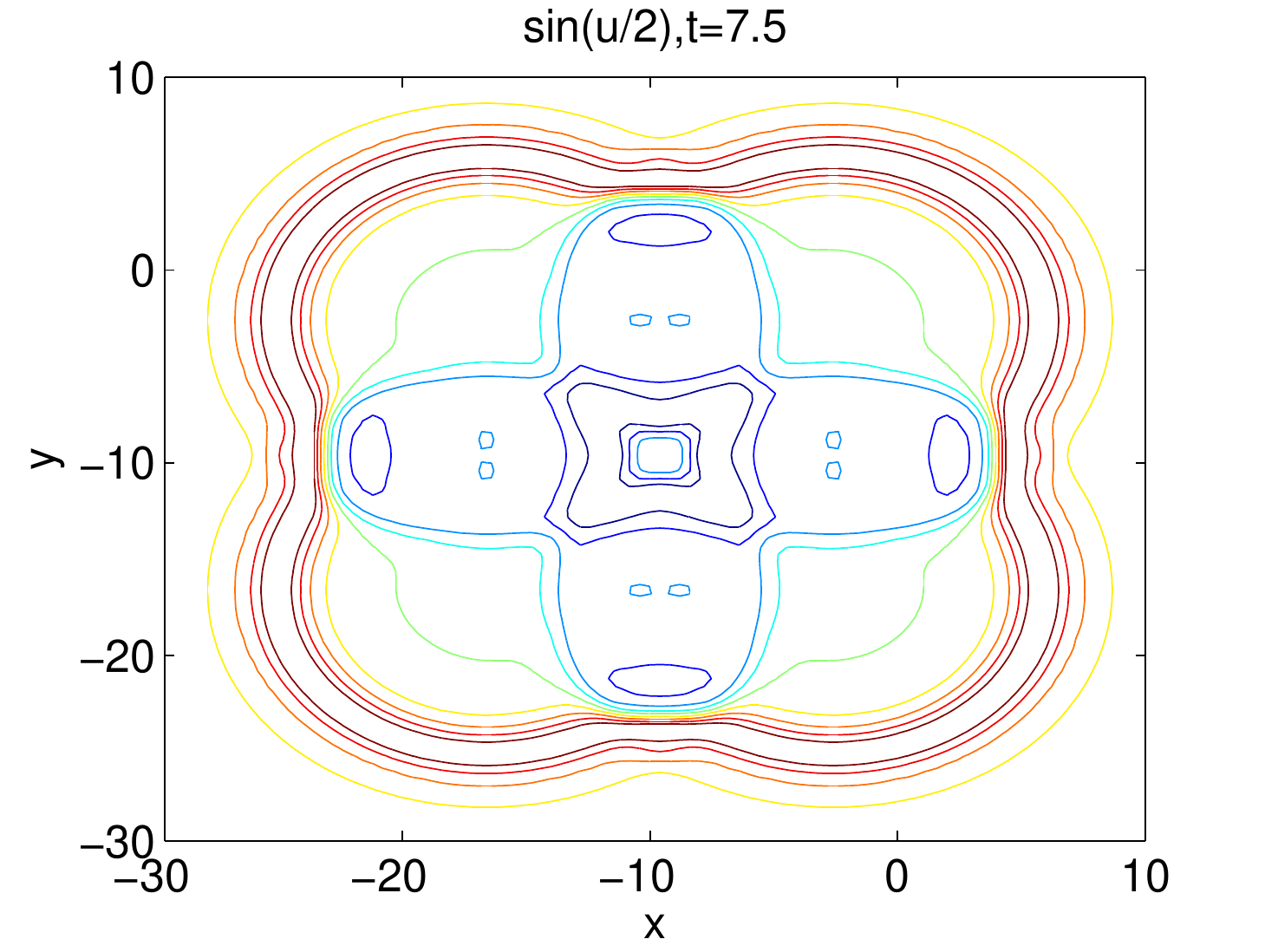}
\end{minipage}
\begin{minipage}[t]{60mm}
\includegraphics[width=60mm]{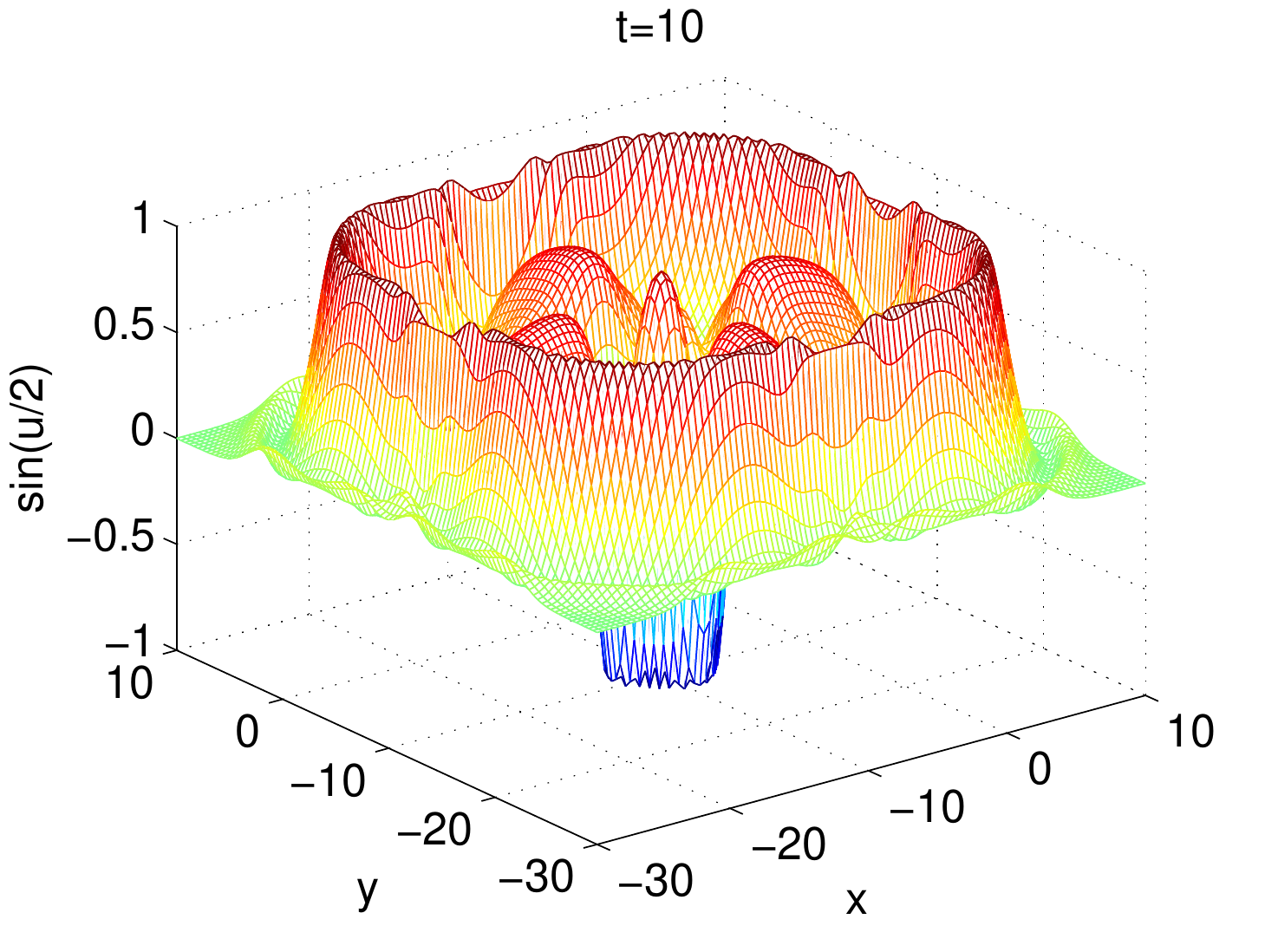}
\end{minipage}
\begin{minipage}[t]{60mm}
\includegraphics[width=60mm]{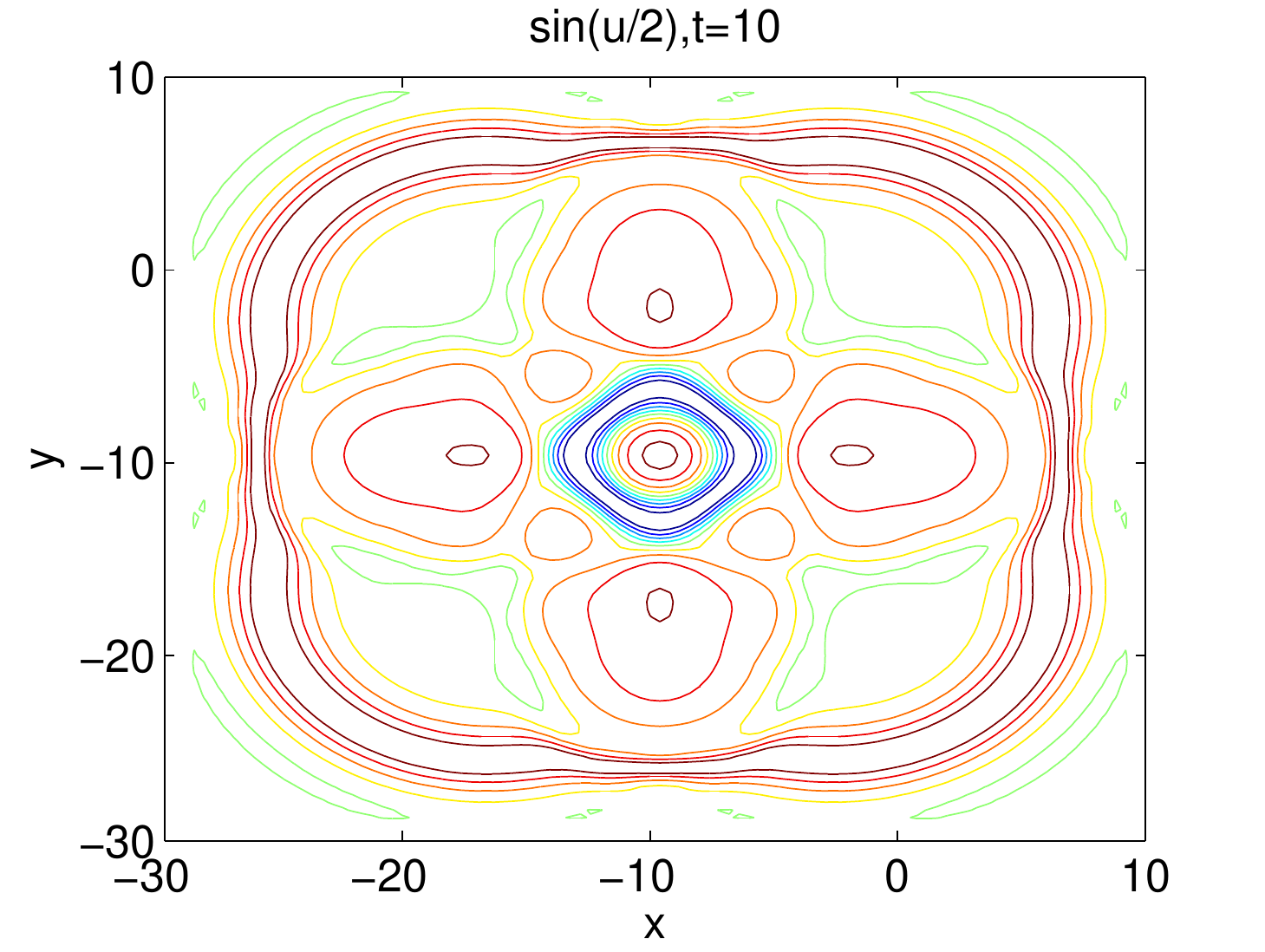}
\end{minipage}
\caption{Collision of four expanding ring solitons: surface and contours plots of initial condition and numerical solutions at times $t=0,2.5,5,7.5,$ and $t=10$, respectively, in terms of $\sin(u/2)$. Spatial and temporal mesh sizes are taken as $h=0.2$ and $\tau=0.01$.}\label{2SG:fig4}
\end{figure}
\begin{figure}[H]
\centering\begin{minipage}[t]{70mm}
\includegraphics[width=70mm]{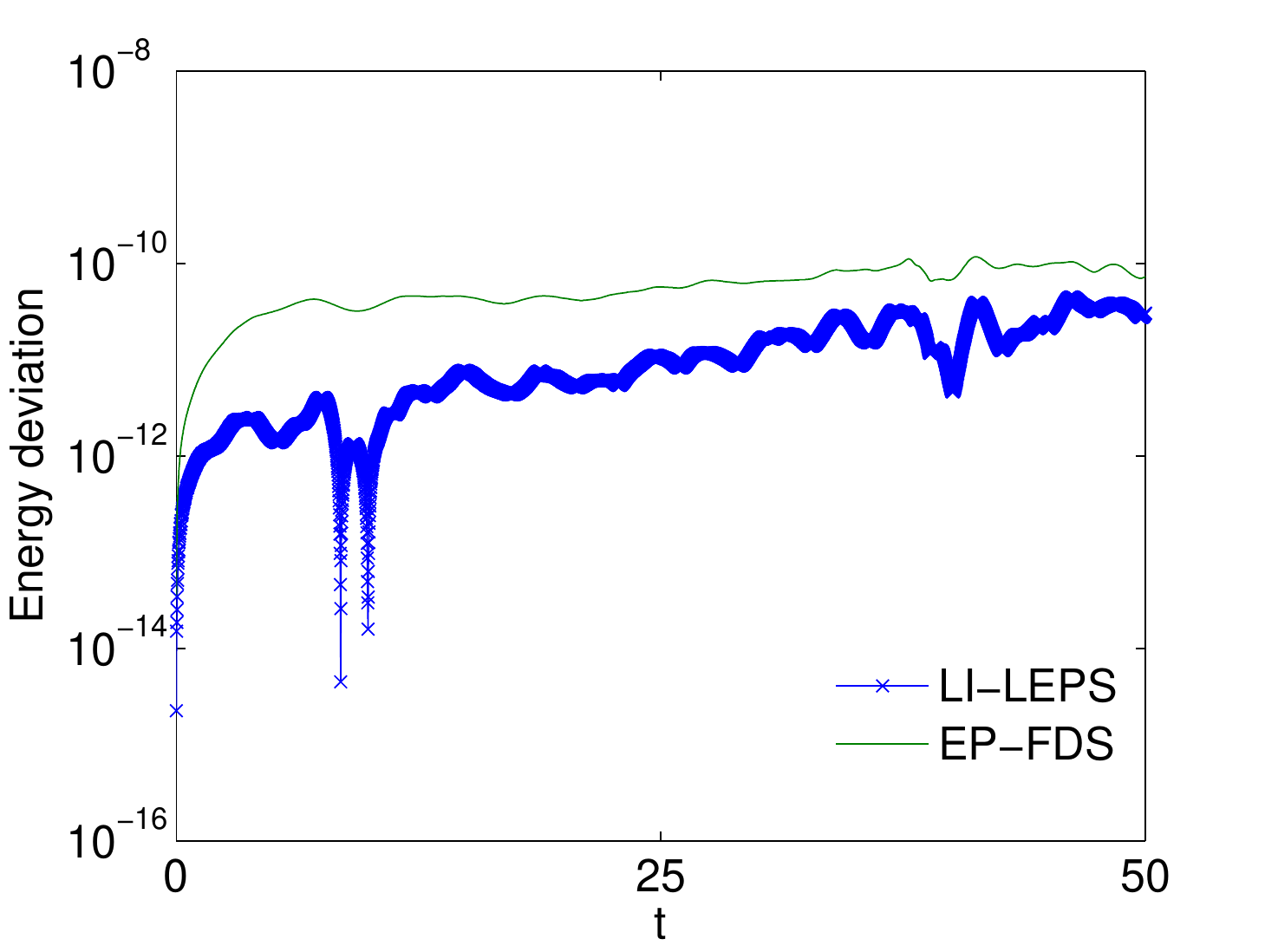}
\end{minipage}
\caption{ Energy deviation over the time interval $t\in[0,50]$
with $h=0.2$ and $\tau=0.01$.}\label{2SG:err3}
\end{figure}

\section{Concluding remarks}\label{2SG:Sec6}
In this paper, we combine the idea of the IEQ approach with the linearly implicit structure-preserving method to develop a novel, linearly implicit scheme for the sine-Gordon equation, which inherits local energy conservation law. Based on the classical energy method, we prove that, without
any restriction on the mesh ratio, the proposed scheme is convergent with order $O(h^2+\tau^2)$  in discrete $H_h^1$-norm. Numerical results verify the theoretical analysis. Compared with some existing local structure-preserving schemes in the literature, the new scheme shows remarkable efficiency and the advantage in energy preservation. The strategy presented in this paper is rather general and useful
so that it can be applied to study a broad class of conserving-systems, such as the nonlinear Schr\"odinger equation, the Gross-Pitaevskii equation, the nonlinear Klein-Gordon equation, etc.

Very recently, Shen et al. have proposed a new approach, termed as scalar auxiliary variable (SAV) approach, for gradient flows. The SAV approach can not only enjoy all the advantages of the IEQ approach, but also has some additional advantages, such as independence of specific forms of the nonlinear part of the free energy, easy to implement, etc \cite{SXY17,SXY18}. Thus, an interesting topic for future studies would be devising linearly implicit and local structure-preserving schemes for conserving-systems by the idea of the SAV approach.

\section*{Acknowledgments}
The authors would like to express sincere gratitude to the referees for their insightful
comments and suggestions. This work is supported by the National Natural Science Foundation of China (Grant Nos. 11771213, 61872422, 11871418),
the National Key Research and Development Project of China (Grant Nos. 2016YFC0600310, 2018YFC0603500, 2018YFC1504205),
the Major Projects of Natural Sciences of University in Jiangsu Province of China (Grant Nos. 15KJA110002, 18KJA110003), the Natural Science Foundation of Jiangsu Province, China (Grant No. BK20171480), the Foundation of Jiangsu Key Laboratory for Numerical Simulation
of Large Scale Complex Systems (201905) and the Yunnan Provincial Department
of Education Science Research Fund Project (2019J0956).


\end{document}